\documentclass{amsart}
\usepackage[all]{xy}
\usepackage{amsmath,amsfonts,amssymb,amsthm,epsfig,amscd,comment,latexsym,psfrag}
\usepackage{nicefrac,xspace,tikz}
\usetikzlibrary{arrows}

\newtheorem{Theorem}{Theorem}[section]
\newtheorem{prop}[Theorem]{Proposition}
\newtheorem{lemma}[Theorem]{Lemma}
\newtheorem{cor}[Theorem]{Corollary}
\newtheorem{conj}[Theorem]{Conjecture}

\theoremstyle{definition}
\newtheorem{Remark}[Theorem]{Remark}

\def\dU{\dot{U}}
\def\hW{{\widehat{W}}}
\def\bA{{\mathbb{A}}}
\def\A{{\sf{A}}}
\def\B{{\sf{B}}}
\def\Cone{\mbox{Cone}}
\def\D{{\mathcal{D}}}
\def\U{\mathcal{U}}
\def\min{{\rm{min}}}

\def\I{{I}}
\def\hI{\hat{I}}
\def\hX{\widehat{X}}
\def\hY{\widehat{Y}}

\def\Kom{{\mathsf{Kom}}}

\def\hg{\widehat{\mathfrak{g}}}
\def\1{{\bf{1}}}
\def\fC{\mathsf{C}}
\def\E{\mathsf{E}}
\def\F{\mathsf{F}}
\def\P{\mathsf{P}}
\def\Q{\mathsf{Q}}

\def\sF{\mathcal{F}}

\def\la{\langle}
\def\ra{\rangle}

\def\k{\Bbbk}
\def\K{\mathcal{K}}

\def\sl{\mathfrak{sl}}

\def\H{\mathcal{H}}
\def\h{\widehat{\mathfrak{h}}}

\def\Z{\mathbb Z}
\def\N{\mathbb N} 
\def\C{\mathbb C}
\def\g{\mathfrak{g}}
\def\hg{\widehat{\mathfrak{g}}}
\def\G{\Gamma}
\def\l{\lambda}

\def\End{\mathrm{End}}
\def\Sym{\mathrm{Sym}}

\def\id{\mathrm{id}}

\newcommand{\Hom}{{\rm Hom}}

\newcommand{\Ind}{{\rm{Ind}}}

\psfrag{Qpm}{$Q_{+-}$} \psfrag{Qmp}{$Q_{-+}$}
\psfrag{Ione}{$\quad\mathbf{1}$}
\psfrag{Xiota}{$\iota$}\psfrag{Xiotap}{$\iota'$}
\psfrag{Xgi}{$g_i$}\psfrag{Xgim1}{$g_i^{-1}$}
\psfrag{Xsumi}{\large{$\sum_{i\in I}$}}
\psfrag{Xbeta}{$\overline{\beta}$}
\psfrag{Xgj}{$g_j$} \psfrag{ifinotj}{\large{if $i\not= j$}}
\psfrag{XSnLm}{$S^n_-\otimes \Lambda^m_+$} 
\psfrag{XLmSn}{$\Lambda^m_+\otimes S^n_-$}
\psfrag{XLm1Sn1}{$\Lambda^{m-1}_+\otimes S^{n-1}_-$}
\psfrag{XSumkb}{$-\sum_{b=0}^{k-2} (k-b-1)$}

\title{Vertex operators and 2-representations of quantum affine algebras}

\begin{document} 
\setcounter{tocdepth}{1}

\author{Sabin Cautis}
\email{cautis@math.ubc.ca}
\address{Department of Mathematics\\ University of British Columbia \\ Vancouver, Canada}

\author{Anthony Licata}
\email{amlicata@gmail.com}
\address{Mathematical Sciences Institute \\ The Australian National University \\ Canberra, Australia}


\begin{abstract}
We construct 2-representations of quantum affine algebras from 2-representations of quantum Heisenberg algebras. The main tool in this construction are categorical vertex operators, which are certain complexes in a Heisenberg 2-representation that recover vertex operators after passing to the Grothendieck group. As an application we categorify the Frenkel-Kac-Segal homogeneous realization of the basic representation of (simply laced) quantum affine algebras. This gives rise to categorical actions of quantum affine (and toroidal) algebras on derived categories of coherent sheaves on Hilbert schemes of points of ALE spaces. 
\end{abstract}

\maketitle

\begin{center}
Dedicated to our advisors Igor Frenkel and Joe Harris on the occasion of their sixtieth birthdays.
\end{center}

\tableofcontents

\section{Introduction}

Affine Lie algebras are central objects at the intersection of representation theory and mathematical physics. One feature which distinguishes them amongst the infinite dimensional Lie algebras is that their irreducible representations admit explicit realizations. Explicit constructions of the basic representation (the first of the irreducible highest weight representations) already connect affine Lie algebras to many other mathematical topics, including symmetric functions, modular forms, and cohomology of Hilbert schemes. 

Historically, one of the first explicit constructions of the basic representation was the homogeneous vertex operator construction due to Frenkel-Kac \cite{FK} and Segal \cite{S}. In this realization, one constructs the basic irreducible module for a simply-laced affine Lie algebra directly from an explicit direct sum of copies of the canonical Fock module of the Heisenberg algebra.  The key components in this construction are the homogeneous Heisenberg subalgebra of an affine Lie algebra, the Fock space representation of this Heisenberg algebra, and the associated vertex operators.

In this paper we categorify the Frenkel-Kac-Segal construction,  The first two ingredients, the homogeneous Heisenberg and its Fock space, were categorified in \cite{CLi1}. The main result in this paper is a categorification of the vertex operators. Thus our categorification allows one to construct categorical actions of quantum affine algebras from categorical actions of quantum Heisenberg algebras. 

Our categorification of the basic representation is closely related to a geometric realisations of the basic representation due to Nakajima \cite{Nak-ICM,Nak-book} and Grojnowski \cite{Groj}.  In the Nakajima-Grojnowski construction, the underlying vector space of the basic representation is the cohomology of the moduli spaces of rank one torsion-free sheaves on the resolution $\widehat{\C^2/\G}$ of a singularity of type $A,D$ or $E$.  In \cite{CLi1} we lift the Nakajima-Grojnowski Heisenberg action to a 2-representation of the {\emph {quantum}} Heisenberg algebra. This gives a 2-representation of the Heisenberg algebra on the derived category of $\C^\times$-equivariant coherent sheaves on the Hilbert schemes of points of $\widehat{\C^2/\G}$.  The categorical vertex operators in this paper extend this 2-representation of the quantum Heisenberg algebra to a 2-representation of the entire quantum affine algebra (in fact, to a 2-representation of quantum toroidal algebras). Subsequently, we recover quantum toroidal algebra actions on the $\C^\times$-equivariant K-theory of the moduli spaces of rank one sheaves, as conjectured by several mathematicians in the 1990s.

The presentation of the quantum affine algebra which appears naturally in this paper is essentially the loop (or Drinfeld) realization. This loop realization is important in relation to conformal field theory and low dimensional topology. We expect that our categorical vertex operators are the beginning of a larger categorification program for vertex algebras. For example, the entire vertex operator algebra structure in the basic representation, which contains not only the vertex operators of this paper but also the Virasoro algebra and other structure, should be categorified. The problem of categorifying vertex algebras was posed over ten years ago by Igor Frenkel.   

In the rest of this introduction, we briefly recall the Frenkel-Kac-Segal construction and explain how it categorifies.  We also explain the relationship between this categorification and the ``Kac-Moody" categorifications in the Khovanov-Lauda/Rouquier framework.

\subsection{Vertex operators and the Frenkel-Kac-Segal construction}

Let $\g$ be a finite dimensional Lie algebra of type $A,D$ or $E$, and let $\hg = \g\otimes \C[t,t^{-1}] \oplus \C c$ be its affinization.  The affine Lie algebra $\hg$ contains a Lie subalgebra known as the homogeneous Heisenberg Lie algebra; the enveloping algebra of this Lie subalgebra, which we refer to as a Heisenberg algebra and denote by $\h$, has essentially one irreducible representation $\mathcal{F}$, known as the Fock space.

Let $V_{\Lambda_0}$ be the basic representation of $\hg$, that is, the highest weight irreducible representation of highest weight $\Lambda_0$.  The basic representation is the ``simplest" of the representations of $\hg$ that can be integrated to the group. The Frenkel-Kac-Segal construction of $V_{\Lambda_0}$ begins with the observation that the restriction of $V_{\Lambda_0}$ from $\hg$ to the homogeneous Heisenberg algebra $\h$ decomposes as a direct sum of copies of the Fock space
$$ V_{\Lambda_0} = \bigoplus_{\l \in Y} \mathcal{F}, $$
with the summands indexed by the root lattice $Y$ of $\g$.  Thus, in order to give an explicit construction of $V_{\Lambda_0}$, one can take a direct sum of copies of the Fock space, one for each element of the root lattice, and explain how to extend the action on this space from 
$\h$ to $\hg$.

Frenkel-Kac \cite{FK} and Segal \cite{S} show that the action of $\hg$ can be constructed from the action of $\h$ and translation in the root lattice $Y$ via the use of vertex operators.  A basic example of a vertex operator is the formal series
$$X(i,z) = \exp\big(\sum_{n=1}^\infty \frac{z^n}{n}h_{i,-n}\big) \exp\big( - \sum_{n=1}^\infty \frac{z^{-n}}{n} h_{i,n} \big) \exp(\log z \cdot \alpha_i(0) + \alpha_i)$$
where $h_{i,n}$ are generators of $\h$ and $\alpha_i$ is a simple root of $\g$. The above expression contains three exponentials. The homogeneous components in $z$ of the first two exponentials are endomorphisms of the Fock space $\mathcal{F}$, while the homogeneous components of the term $\exp(\log z \cdot \alpha_i(0) + \alpha_i)$ are ``lattice translation" operators which moves the various copies of $\mathcal{F}$ along the lattice (we refer to \cite{FK} for the precise definition of the lattice translation operator).  Each homogeneous component of $X(i,z)$ is thus an endomorphism of $V_{\Lambda_0}$.  These endomorphisms, together with their adjoints, generate the action of the affine Lie algebra, giving an explicit construction of the basic representation.

It was later shown in \cite{FJ} that this construction admits a $q$-deformation. (The $q$-deformed Frenkel-Kac-Segal construction can also be extended from the quantum affine algebra to the quantum toroidal algebra.)  In this $q$-deformation, the Heisenberg algebra $\h$ is replaced by the quantum Heisenberg algebra, and the vertex operators $X(i,z)$ are replaced by $q$-vertex operators
$$X_q(i,z) = \exp\big(\sum_{n=1}^\infty \frac{z^n}{[n]}h_{i,-n}\big) \exp\big( - \sum_{n=1}^\infty \frac{z^{-n}}{[n]} h_{i,n} \big) \exp(\log z \cdot \alpha_i(0) + \alpha_i)$$
where $[n] = \frac{q^{n}-q^{-n}}{q-q^{-1}}$ is the quantum integer. The homogeneous components of the $X_q(i,z)$, together with their adjoints, then give the basic representation of the quantum affine (or toroidal) algebra $U_q(\hg)$.

In order to categorify the vertex operators $X_q(i,z)$ we consider new operators $P_i^{(n)}$ and $Q_i^{(1^n)}$ defined by the formulas
$$ \exp \left( \sum_{m \geq 1} \frac{h_{i,-m}}{[m]} z^m \right) = \sum_{n \geq 0} P_i^{(n)} z^n \text{ and } \exp \left(- \sum_{m\geq 1} \frac{h_{i,m}}{[m]} z^m \right) = \sum_{n \geq 0} (-1)^n Q_i^{(1^n)} z^n. $$
In terms of these operators, a single homogeneous component in $z$ of $X_q(i,z)$ becomes an expression of the form
\begin{equation}\label{eq:alt}
	\big[ \sum_{n \ge 0} (-1)^n P_i^{(n+k)} Q_i^{(1^{n})} \big] t_i,
\end{equation}
where $t_i$ acts only in the lattice and $\sum_{n \ge 0} (-1)^n P_i^{(n+k)} Q_i^{(1^{n})}$ acts only in the Fock space. Thus, the $q$-deformed Frenkel-Kac-Segal construction says that these alternating sums, together with their adjoints, generate an action of the quantum affine (or toroidal) algebra $U_q(\hg)$.

\subsection{Categorification of the Frenkel-Kac-Segal construction}

A categorification of the Heisenberg algebra $\h$ and its Fock space representation was given in \cite{CLi1}. We recall this definition in Section \ref{sec:hei}, where we also define a the notion of a 2-representations of $\h$. Very roughly, a 2-representation of $\h$ consists of a 2-category $\K$ where the objects are indexed by the natural numbers, the 1-morphisms are compositions of generating 1-morphisms $\P_i$ and $\Q_i$, and there are 2-morphisms with specified relations (these relations are described using a graphical calculus of planar diagrams).  The relations for 2-morphisms imply that the 1-morphisms $\P_i$ and $\Q_i$ satisfy categorical analogs of the relations in the Heisenberg algebra.  Thus a 2-representation of $\h$ is a categorification of a representation of the Heisenberg algebra.

To categorify the Frenkel-Kac-Segal construction of the basic representation of $\hg$ we need to lift the operators in (\ref{eq:alt}) from vector spaces to categories. In a 2-representation $\K$ of $\h$, the operators $P_i^{(n)}$ and $Q_i^{(1^n)}$ are lifted to indecomposable 1-morphisms $\P_i^{(n)}$ and $\Q_i^{(1^n)}$. It is natural that alternating sums like (\ref{eq:alt}) should lift to complexes 
\begin{equation}\label{eq:example}
 \dots \rightarrow \P_i^{(n+k)}\Q_i^{(1^n)}\la l \ra \rightarrow \P_i^{(n+k+1)} \Q_i^{(1^{n+1})} \la l+1 \ra \rightarrow \dots 
\end{equation}
in the homotopy category $\Kom(\K)$ of $\K$. In Section \ref{sec:vertexops} we define such complexes.

It is immediate from the definition that these complexes descend to the homogeneous components of vertex operators after passing to the Grothendieck group. In fact, this would be true regardless of the differentials we choose in our complex (even using the zero differential would work). The interesting content is that these complexes satisfy categorical relations of the quantum affine algebra inside $\Kom(\K)$ itself, before passing to the Grothendieck group. A summary of the relations we check is in the statement of Theorem \ref{thm:main1}. 

\subsection{Categorical $\hg$ actions and higher relations}\label{sec:cataction}

Categorical actions of quantum groups have been studied in several situations and flavours \cite{KL2,KL3,R,CK,CKL2}. The framework given in \cite{KL3,R} imposes a rather rigid structure on 2-representations by specifying explicitly in terms of generators and relations the algebras which should act as natural transformations of the generating 1-morphisms $\E_i, \F_i$ and their compositions.  These algebras are known as KLR algebras or quiver Hecke algebras.
In order to fit our categorification to the Khovanov-Lauda/Rouquier framework, we must therefore therefore explain how the KLR algebras appear in the endomorphism algebras of our vertex operator complexes.

Unfortunately, it turns out that defining the KLR generators directly, much less checking the required relations, is difficult.   Indeed, our defining 1-morphisms $\E_i$ are complexes in a homotopy category, and since it is difficult to study their endomorphism spaces directly it is not clear how best to go about defining the basic KLR generators and checking relations . For example, it follows as a consequence of the KLR algebra action that the composition $\E_i \E_i$ breaks up as a direct sum $\E_i^{(2)} \la -1 \ra \oplus \E_i^{(2)} \la 1 \ra$. This translates into the statement that the composition of two vertex operator complexes in the homotopy category breaks up as a direct sum of two smaller complexes (those corresponding to $\E_i^{(2)}$) which are identical up to a grading shift. However, even defining the complexes which correspond to $\E_i^{(2)}$ is a difficult task (see section \ref{sec:powers}).

It is thus useful to develop techniques which detect the presence of KLR algebras but which bypass the need to introduce the KLR generators and check the KLR relations directly.   To this end we introduced in \cite{Cau} the idea of a $(\g,\theta)$ action. The main result of that paper is that a $(\g,\theta)$ action induces an action of the corresponding KLR algebra on 2-morphisms.   In our current paper, the quantum affine algebra which acts appears not in its Kac-Moody presentation but rather in its Drinfeld presentation, and as a result the axioms of \cite{Cau} need to be slightly modified in order to be applied. Thus in this paper we give an affinization of the notion from \cite{Cau}, and define what we call $(\hg,\theta)$ actions. An immediate Corollary (Cor. \ref{cor:main1}) to the main result in this paper (Thm. \ref{thm:main1}) is that the complexes defined above induce a $(\hg,\theta)$ action. 

Now any $(\hg,\theta)$ action may be restricted to a $(\g,\theta)$ action. As a consequence of \cite{Cau} we then obtain an action of the KLR algebras associated to $\g$ on 2-morphisms in a $(\hg,\theta)$ action.  It is a striking indication of the ridigity of higher representation theory that the KLR algebras appear. The fact that there is an action of KLR algebras on our vertex operator complexes may be phrased more precisely as follows. 

\begin{Theorem}\label{thm:klr}
Given an integrable 2-representation $\K$ of the Heisenberg 2-category $\h$ there exists a 2-functor $\dot{\mathbb{U}}_q(\g) \rightarrow \Kom(\K)$. 
\end{Theorem}

Here $\dot{\mathbb{U}}_q(\g)$ is the 2-category defined in \cite{KL3} which categorifies the quantum group $U_q(\g)$ and whose 2-morphism spaces contain the KLR algebras. In some ways Theorem \ref{thm:klr} says that a large part of the higher representation theory of quantum groups is captured in the (much simpler) higher representation theory of an associated Heisenberg algebra.  

It is natural to conjecture that a $(\hg,\theta)$ induces an action of the affine type KLR algebras though we have not checked this.  In fact, in order to completely describe the endomorphism algebras of our vertex operators one would need to extend the KLR definitions one more step to the \emph{toroidal} setting, and as of yet the KLR-type definitions for categorical actions of quantum toroidal algebras have not been given.

\subsection{Generalization to other Kac-Moody algebras}

The essential idea at the core of our construction is that 2-representations of $\h$ give rise to 2-representation of $\hg$. On the other hand, the only input data used to define 2-representations of $\h$ is the finite dimensional algebra $B^\G := \C[\G] \ltimes \wedge^*(\C^2)$, where $\G \subset SL_2(\C)$ is the finite subgroup associated to the Dynkin diagram of $\hg$ using the McKay correspondence.  Thus we obtain what can be viewed as a categorical form of the McKay correspondence: starting from the finite subgroup $\G \subset SL_2(\C)$, we construct the finite-dimensional algebra $B^\G$, and using this algebra we construct 2-representations of the associated quantum affine algebra.  

The algebra $B^\G$ can be described directly in terms of the underlying affine Dynkin diagram without direct reference to the finite group $\G$. This description suggests how to generalize some of the constructions of this paper to other Kac-Moody type.  We hope to describe the details elsewhere, and for the time being content ourselves with briefly describing how to generalize the algebra $B^\G$ to other Dynkin types.

Fix a simply laced Dynkin diagram $D$ and choose an orientation of the edges. Let $\C[dD]$ denote the path algebra of the doubled quiver $dD$. Thus a path in $dD$ is described as a sequence of vertices $(a_1 | a_2 | \dots | a_k)$ where $a_i$ and $a_{i+1}$ are connected by an edge in $D$. We define $B^D$ to be the quotient of $\C[dD]$ by the two sided ideal generated by 
\begin{itemize}
\item $(a|b|c)$ if $a \ne c$ and
\item $(a|b|a) + (a|c|a)$ whenever $a$ is connected to both $b$ and $c$. 
\end{itemize}
The algebra $B^D$ can be used to define a 2-category which categorifies the Heisenberg algebra $\mathfrak{h}_D$ associated to the root lattice of $D$. Just like in \cite{CLi1}, one can also categorify the Fock space modeled on this lattice. Then the categorical vertex operators in this paper easily generalize and give 2-representations of the affine quantum algebra associated to $D$.

The algebra $B^D$ appears in \cite{HK}, where it is called the skew zig-zag algebra. In that paper $B^D$ is used to categorify the adjoint representation of $\g$ when $\g$ is of finite type. Our categorification of the basic representation is in fact an extension of theirs coming from the fact that the restriction of the basic representation $V_{\Lambda_0}$ of $\hg$ to $\g$ contains the adjoint representation of $\g$ as a direct summand. More precisely, if you restrict our categorification of $V_{\Lambda_0}$ from $\hg$ to the copy of the adjoint representation of $\g$ sitting inside it, our complexes recover the construction in \cite{HK}. In this special case, the complexes have length at most two, which makes checking the relations in the adjoint representation of $\g$ much easier than checking the relations in the basic representation of $\hg$.




\subsection{Work of Carlsson-Okounkov} 

In \cite{CO,Car} Carlsson and Okounkov describe an operator, denoted $\sf{W}$, on the cohomology of Hilbert schemes of points on a surface. This operator is defined using the Chern class of a virtual bundle over these Hilbert schemes, and their main theorem states that $\sf{W}$ can be expressed as a vertex operator.  It would be interesting to understand their result at the categorified level, and to define an analogue of their operator $\sf{W}$ as a functor between derived categories of coherent sheaves on Hilbert schemes.  For ALE spaces, the resulting functor should be related to the categorical vertex operators of this paper, but the details of the identification are not given yet.

\vspace{.5cm}

\noindent {\bf Acknowledgments:}
The authors benefited from discussions with Mikhail Khovanov, Aaron Lauda, Hiraku Nakajima, Raphael Rouquier, Travis Schedler and Joshua Sussan. S.C. was supported by NSF grants DMS-0964439, DMS-1101439 and the Alfred P. Sloan foundation. A.L. would like to thank the Institute for Advanced Study for support during the preliminary stages of this work.


\section{Notation and terminology} \label{sec:notation}
For the entirety of this paper we let $\k$ denote a field of characteristic zero. 
We let $\k(q)$ denote the field of rational functions of one variable, $q$.
We denote
$$[j] := q^{-j+1} + q^{-j+3} + \dots + q^{j-3} + q^{j-1}$$
the quantum integer. If $j \ge 0$ then $V_j$ denotes the graded vector space
\begin{equation}\label{eq:V_j}V_j := \k \la j \ra \oplus \k \la j-2 \ra \oplus \dots \oplus \k \la -j+2 \ra \oplus \k \la -j \ra
\end{equation}
where $\la 1 \ra$ is a shift of $1$ in the grading.

If $\l = (\l_1 \ge \l_2 \ge \dots \ge \l_k)$ is a partition then $|\l| := \sum_i \l_i$ denotes its size. We say that $\l' \subset \l$ if $\l'$ is contained in $\l$, meaning that $\l_i \ge \l'_i$ for all $i$. We denote by $\l^t$ the transposed partition of $\l$ (for example, $(n)^t = (1^n)$). 

\subsection{Dynkin data}\label{sec:data}
From now on fix a simply-laced Dynkin diagram of affine type and denote its vertex set by $\hI$. The special affine vertex in $\hI$ is labeled $0$ and we let $\I := \hI \setminus \{0\}$. The subdiagram whose vertex set is $\I$ is a Dynkin diagram of finite type $A,D,E$.  We denote the Lie algebras associated to these Dynkin diagrams by $\g$ and $\hg$. 

We denote the weight lattice of $\g$ by $X$ and the root lattice by $Y$. We denote $X_\k := X \otimes_\Z \k$ and $Y_\k = Y \otimes_\Z \k$. The standard pairing on the weight lattice is denoted with brackets $\la \cdot, \cdot \ra$, which should not be confused with the grading shift on categories. For $i \in \I$, $\alpha_i \in Y$ and $\Lambda_i \in X$ will denote the simple roots and fundamental weights, respectively. These satisfy the relation $C_{i,j} = \la \alpha_i, \alpha_j \ra$ where $C_{i,j}$ is the Cartan matrix of $\g$. In terms of the Dynkin diagram, we have 
$$\la \alpha_i, \alpha_j \ra = 
\begin{cases}
2 \text{ if } i=j \\
-1 \text{ if } i \ne j \text{ are joined by an edge } \\
0 \text{ if } i \ne j \text{ are not joined by an edge.}
\end{cases}$$
Moreover, $\la \Lambda_i, \alpha_j \ra = \delta_{i,j}$ for all $i, j \in \I$. We will write $\l_i$ and $\la i,j \ra$ instead of $\la \l, \alpha_i \ra$ and $\la \alpha_i, \alpha_j \ra$.   

Similarly, we denote by $\hX$ the weight lattice of $\hg$, $\hY$ the affine root lattice, etc.  We use the same notation $\la \cdot,\cdot \ra$ for the pairing on the affine weight lattice as for the finite weight lattice.  The imaginary root, which is denoted $\delta$, satisfies $\la \delta,\alpha_i \ra = 0$ for all $i$. The associated Weyl groups of $X$ and $\hX$ are denoted $W$ and $\hW$.  

Fix an orientation $\epsilon$ of the Dynkin diagram of $\g$. If $\la i,j \ra = -1$ then we set $\epsilon_{ij} = 1$ if the edge is oriented $i \rightarrow j$ by $\epsilon$ and $\epsilon_{ij} = -1$ if it oriented $j \rightarrow i$. If $\la i, j \ra = 0$ then we set $\epsilon_{ij} = 0$. Notice that in both cases we have $\epsilon_{ij} = -\epsilon_{ji}$.

\subsection{Graded 2-categories}

By a graded category we will mean a category equipped with an auto-equivalence $\la 1 \ra$. We denote by $\la l \ra$ the auto-equivalence obtained by applying $\la 1 \ra$ a total of $l$ times. The Grothendieck group $K_0(\mathcal{C})$ of an additive category $\mathcal{C}$ is the abelian group generated the set $\{[A]: A \in Ob(\mathcal{C})\}$ modulo the relation $[A]+[A']=[A'']$ if $A'' \cong A \oplus A'$. This group is a $\Z[q,q^{-1}]$-module where $q$ acts by the shift $\la -1 \ra$. We usually tensor this with the field $\k(q)$ to obtain a $\k(q)$-module.
 
A graded additive $\k$-linear 2-category $\K$ is a category enriched over graded additive $\k$-linear categories. This means that for any two objects $A,B \in \K$ the Hom category $\Hom_{\K}(A,B)$ is a graded additive $\k$-linear category. Moreover, the composition map $\Hom_{\K}(A,B) \times \Hom_{\K}(B,C) \to \Hom_{\K}(A,C)$ is a graded additive $\k$-linear functor. In a graded 2-category the 1-morphisms are equipped with an auto-equivalence $\la 1 \ra$. We denote by $\la l \ra$ the auto-equivalence obtained by applying $\la 1 \ra$ $l$ times. If $\A$ and $\B$ are 1-morphisms we also write $\Hom^l(\A,\B)$ for $\Hom(\A,\B \la l \ra)$ and $\End^l(\A)$ for $\Hom(\A,\A \la l \ra)$. 

A graded additive $\k$-linear 2-functor $F: \K \to \K'$ is a (weak) 2-functor that maps the Hom categories $\Hom_{\K}(A,B)$ to $\Hom_{\K'}(FA,FB)$ by additive functors that commute with the auto-equivalence $\la 1 \ra$.

If $\K$ is an additive 2-category, the Grothendieck group $K_0(\K)$ is a $\k(q)$-linear category whose objects are the same as those of $\K$ and whose morphism spaces are 
$$ \Hom_{K_0(\K)}(A,A') = K_0(\Hom_{\K}(A,A')).$$

{\bf Example:} Suppose $B_n$ is a sequence of graded $\k$-algebras indexed by $n \in \N$. Then one can define a 2-category $\K$ whose objects (0-morphisms) are indexed by $\N$, the 1-morphisms are graded $(B_m,B_n)$-bimodules and the 2-morphisms are maps of graded $(B_m,B_n)$-bimodules.

\subsubsection{Idempotent completeness}

An additive category $\mathcal{C}$ is said to be idempotent complete when every idempotent 1-morphism splits in $\mathcal{C}$. Similarly, we say that the additive 2-category $\K$ is idempotent complete when the Hom categories $\Hom_{\K}(A,B)$ are idempotent complete for any pair of objects $A, B \in \K$, (so that all idempotent 2-morphisms split). 

\subsubsection{Triangulated 2-categories}

A graded triangulated category is a graded category equipped with a triangulated structure where the autoequivalence $\la 1 \ra$ takes exact triangles to exact triangles. We denote the homological shift by $[\cdot]$ where $[1]$ denotes a downward shift by one. The Grothendieck group $K_0(\mathcal{C})$ of a graded triangulated category $\mathcal{C}$ is the abelian group generated the set $\{[A]: A \in Ob(\mathcal{C})\}$ modulo the relation $[A]+[A']=[A'']$ if there exists a distinguished triangle $A \rightarrow A'' \rightarrow A'$. As before, this is a $\Z[q,q^{-1}]$-module where $q$ acts by $\la -1 \ra$. 

A graded triangulated $\k$-linear 2-category ${\K'}$ is a category enriched over graded triangulated $\k$-linear categories. This means that for any two objects $A,B \in \K'$ the Hom category $\Hom_{\K'}(A,B)$ is a graded additive $\k$-linear triangulated category. Here are two examples to keep in mind. 

{\bf Example}: the homotopy category $\K' := \Kom(\K)$ of a graded additive $k$-linear 2-category ${\K}$. The objects of $\K'$ are the same as the objects of $\K$.  The 1-morphisms of $\K'$ are unbounded complexes of 1-morphisms in $\K$, and 2-morphisms are maps of complexes. Two complexes of 1-morphisms are then deemed isomorphic if they are homotopy equivalent.  This makes $\Hom_{\K'}(A,B)$ into a graded triangulated category. 



\section{Quantum Heisenberg algebras and 2-representations of $\h$}\label{hei}

\subsection{The quantum Heisenberg algebra}\label{sec:qhei}

The quantum Heisenberg algebra, which we denote by $\h$, plays a central role in all of the constructions to come. We begin by describing this algebra and its Fock space representation.

The traditional presentation for the quantum Heisenberg algebra is as a unital algebra generated by $h_{i,n}$, where $i \in \I$ and $n \in \Z \setminus \{0\}$.  The relations are
\begin{equation}\label{rel:as}
[h_{i,m}, h_{j,n}] = \delta_{m,-n} [n \la i,j \ra] \frac{[n]}{n}.
\end{equation}
When $q=1$, this presentation specializes to the standard presentation of the non-quantum Heisenberg algebra.  Sometimes, relation (\ref{rel:as}) appears in the literature with a minus sign on the right hand side, though this does not change the isomorphism class of the algebra itself (just replace $h_{i,m}$ with $-h_{i,m}$ if $m > 0$). 

Our preferred presentation of $\h$ is slightly less common.  We use generators for $P_i^{(n)}$ and $Q_i^{(n)}$ which are obtained from the standard generators $h_{i,m}$ via the generating functions
\begin{equation}\label{eq:gens1}
\exp \left( \sum_{m \geq 1} \frac{h_{i,-m}}{[m]} z^m \right) = \sum_{n \geq 0} P_i^{(n)} z^n 
\text{ and } 
\exp \left( \sum_{m\geq 1} \frac{h_{i,m}}{[m]} z^m \right) = \sum_{n \geq 0} Q_i^{(n)} z^n.
\end{equation}

\begin{lemma}\label{lem:PQrels}
The elements $\{P_i^{(n)},Q_i^{(n)} \}_{i \in \I, n \geq 0}$ also generate $\h$.  
They satisfy the following relations:
\begin{eqnarray*}
P_i^{(n)} P_j^{(m)} &=& P_j^{(m)} P_i^{(n)} \text{ and } Q_i^{(n)}Q_j^{(m)} = Q_j^{(m)} Q_i^{(n)} \text{ for all } i,j \in \I, \\
Q_i^{(n)} P_j^{(m)} &=& 
\begin{cases}
\sum_{k \ge 0} [k+1] P_i^{(m-k)} Q_i^{(n-k)} \text{ if } i=j, \\
P_j^{(m)} Q_i^{(n)} + P_j^{(m-1)} Q_i^{(n-1)} \text{ if } i \neq j \in I \text{ with } \la i,j \ra = -1 \\
P_j^{(m)} Q_i^{(n)} \text{ if } i \ne j \in I \text{ with } \la i,j \ra = 0.
\end{cases}
\end{eqnarray*}
\end{lemma}
\begin{proof}
This is proved in Lemma 1 of \cite{CLi1}. Note $P_i^{(0)} = Q_j^{(0)} = 1$ and $P_i^{(k)} = Q_i^{(k)} = 0$ when $k < 0$ so the summations in the relations above are all finite.
\end{proof}

There is an alternative generating set of $\h$ given by elements $P_i^{(1^n)}$ and $Q_i^{(1^n)}$. These are defined using generating functions similar to (\ref{eq:gens1}) as follows
\begin{equation}\label{eq:gens2}
\exp \left( -\sum_{m \geq 1} \frac{h_{i,-m}}{[m]} z^m \right) = \sum_{n \geq 0} (-1)^n P_i^{(1^n)} z^n
\text{ and } 
\exp \left( -\sum_{m \geq 1} \frac{h_{i,m}}{[m]} z^m \right) = \sum_{n \geq 0} (-1)^n Q_i^{(1^n)} z^n.
\end{equation}
The Heisenberg algebra admits an involution $\psi: \h \longrightarrow \h$ defined by
$$ 
P_i^{(n)} \mapsto P_i^{(1^n)}, \ \ Q_i^{(n)} \mapsto Q_i^{(1^n)}, \ \
P_i^{(1^n)} \mapsto P_i^{(n)}, \ \ Q_i^{(1^n)} \mapsto Q_i^{(n)}.
$$
In particular, the commutation relations among the $P_i^{(1^n)}$ and $Q_i^{(1^n)}$ are the same as those between the $P_i^{(n)}$ and $Q_i^{(n)}$ (just replace $(n)$ by $(1^n)$ everywhere). 

The involution $\psi$ is essentially the standard involution on symmetric functions. More precisely, fix $i \in I$ and let $\h_i^- \subset \h$ denote the subalgebra generated by the $\{P_i^{(n)}\}$. After setting $q=1$, there is an isomorphism of algebras 
$$ \h_i^- \cong \Sym = \Z[h_1,h_2,\hdots,h_n,\hdots]$$
which takes $P_i^{(n)}$ to the homogeneous symmetric function $h_n$. This isomorphism intertwines $\psi$ with the standard involution on symmetric functions which exchanges homogeneous and elementary symmetric functions.

We have the following relations in addition to those from Lemma \ref{lem:PQrels}.
\begin{eqnarray*}
P_i^{(m)} P_j^{(1^n)} &=& P_j^{(1^n)} P_i^{(m)} \text{ and } Q_i^{(m)} Q_j^{(1^n)} = Q_j^{(1^n)} Q_i^{(m)} \\
Q_i^{(1^m)} P_j^{(n)} &=&
\begin{cases}
P_i^{(n)} Q_i^{(1^m)} + [2] P_i^{(n-1)} Q_i^{(1^{m-1})} + P_i^{(n-2)} Q_i^{(1^{m-2})} \text{ if } i=j \\
\sum_{k \ge 0} P_j^{(n-k)} Q_i^{(1^{m-k})} \text{ if } \la i,j \ra = -1 \\
P_j^{(n)} Q_i^{(1^m)} \text{ if } \la i,j \ra = 0.
\end{cases}
\end{eqnarray*}
These relations can be checked in the same way as those in Lemma \ref{lem:PQrels}.

The unital algebra $\h$ also admits an idempotent modification where the unit is replaced by a collection of idempotents $1_n$, $n \in \Z$. The relations between these idempotents and the generators $P_i^{(n)}$ and $Q_j^{(n)}$ is 
$$1_{n+k} P_i^{(n)} = 1_{n+k} P_i^{(n)} 1_k = P_i^{(n)} 1_k \text{ and }
1_k Q_i^{(n)} = 1_k Q_i^{(n)} 1_{n+k} = Q_i^{(n)} 1_{n+k}$$
We also have that $1_m P_i^{(n)} 1_k$ and $1_k Q_i^{(n)} 1_m$ is zero unless $m = n+k$. Hence the idempotent Heisenberg algebra can be thought of as a category, where the objects are the integers and the morphisms from $n$ to $m$ are given by $1_m\h1_n$. From this point of view we should write $P_i 1_n$ or $1_n P_i$ to make clear the domain and codomain. However, the domain or codomain will often be obvious or irrelevant in which case we just write $P_i$.

A representation of $\h$ is said to be \emph{integrable} if the object $1_n$ is zero for $n \ll 0$. It is \emph{weakly integrable} if it is the direct sum of (possibly infinitely many) integrable representations. 

\subsubsection{The Fock space}

The Heisenbeg algebra $\h$ has a natural integrable representation $\sF$, known as the Fock space. Let $\h^+ \subset \h$ denotes the subalgebra generated by the $Q_i^{(n)}$ for all $i \in I$ and $n \ge 0$. Let ${\rm{triv}}_0$ denote the trivial representation of $\h^+$, where all $Q_i^{(n)}$ ($n>0$) act by zero. Then 
$$\sF := \Ind_{\h^+}^{\h}({\rm{triv}}_0)$$ 
is called the Fock space representation of $\h$. It inherits a $\Z$ grading $\sF = \oplus_{m \in \N} \sF(m)$ by declaring ${\rm{triv}}_0$ to have degree zero, $P_i^{(n)}$ degree $n$ and $Q_i^{(n)}$ degree $-n$.

\subsection{2-representations of $\h$}\label{sec:hei}
We now explain what it means to have a 2-representation of $\h$. This concept is closely related to the categorification of $\h$ defined in \cite{CLi1}. In that paper we defined a graded additive $\k$-linear 2-category $\H^\G$ together with an algebra isomorphism from $K_0(\H^\G)$ to $\h$ (here $\G\subset SL_2(\C)$ is the finite group associated to our affine Dynkin diagram by McKay correspondence). In essence, a 2-representation of $\h$ is a representation of the 2-category $\H^\G$ on a graded, $\k$-linear 2-category.  

A 2-representation of $\h$ consists of a graded, idempotent complete $\k$-linear category $\K$ where
\begin{itemize}
\item 0-morphisms (objects) are denoted $\D(n)$ and are indexed by $n \in \Z$.
\item 1-morphisms include the identity 1-morphisms $\1_n$ of $n \in \Z$ (these are mutually orthogonal idempotents) as well as 
$$ \P_i \1_n : \D(n) \rightarrow \D(n+1), \ \ \Q_i \1_n : \D(n) \rightarrow \D(n-1).$$
Other 1-morphisms are obtained from these by taking compositions, direct sums and grading shifts. 
\item 2-morphisms include the identity 2-morphisms, cups, caps and dots (see the pictures below). Other 2-morphisms are obtained from these by composition.
\end{itemize}

\subsubsection{2-morphisms in $\K$}

We require that the space of 2-morphisms between any two 1-morphisms be finite dimensional and that $\Hom(\1_n, \1_n \la \ell \ra)$ is zero if $\ell < 0$ and one-dimensinal if $\ell=0$. Moreover, the 2-morphisms must satisfy the defining relations in the Heisenberg 2-category defined in \cite{CLi1}. We now summarize these relations. 

The 2-morphisms are encoded by a graphical calculus similar to ones used in the categorifications of quatum groups and other Heisenberg algebras, for example \cite{L,KL1,KL2,KL3,K,LS}. 

Strands will be used to denote 1-morphisms. More precisely, an upward pointing strand labeled by $i$ denotes $\P_i$ while a downward pointing strand labeled $i$ denotes $\Q_i$. Composition of 1-morphisms is obtained by sideways concatenation of diagrams. The space of 2-morphisms between compositions of $\P_i$s and $\Q_j$s is a $\k$-algebra described by certain string diagrams with relations. By convention, composition of 2-morphisms is done vertically from the bottom and going up. 

We have the following generating 2-morphisms. For any $i,j \in \I$ with $\la i, j \ra = -1$ there is a 2-morphism $\P_i \rightarrow \P_j \la 1 \ra$ which is diagrammatically denoted by a solid dot: 
$$
\begin{tikzpicture}[scale=.75][>=stealth]
\draw (0,0) -- (0,1) [->][very thick];
\draw (0,-.25) node {$i$};
\filldraw [blue] (0,.5) circle (3pt);
\draw (0,1.25) node {$j$};
\end{tikzpicture}
$$
Note that such an $i-j$ dot is defined to have degree one.  For each $i \in \I$ there is also a 2-morphism $P_i \rightarrow P_i \la 2 \ra$ of degree two. 

The other generators are given by caps, cups and crossings. These, together with their gradings, are depicted below: 
$$
\begin{tikzpicture}[scale=.75][>=stealth]
\draw  (-.5,.5) node {$deg$};
\draw [->](0,0) -- (1,1) [very thick];
\draw [->](1,0) -- (0,1) [very thick];
\draw (1.5,.5) node{$ = 0$};
\end{tikzpicture}
$$
$$
\begin{tikzpicture}[scale=.75][>=stealth]
\draw  (-.5,-.25) node {$deg$};
\draw (0,0) arc (180:360:.5)[->] [very thick];
\draw (1.75,-.25) node{$ = deg$};
\draw (3.5,-.5) arc (0:180:.5) [->][very thick];
\draw (4.5,-.25) node{$ =-1$};
\end{tikzpicture}
$$
$$
\begin{tikzpicture}[scale=.75][>=stealth]
\draw  (-.5,-.25) node {$deg$};
\draw (0,0) arc (180:360:.5)[<-] [very thick];
\draw (1.75,-.25) node{$ = deg$};
\draw (3.5,-.5) arc (0:180:.5) [<-][very thick];
\draw (4.5,-.25) node{$ = 1$};
\end{tikzpicture}
$$

The diagrammatic relations include any planar isotopy which preserves the relative height of dots.
These planar isotopy relations imply that the caps and cups give canonical adjunctions, making $\P_i$ and $\Q_i$ biadjoint up to a shift. Explicitly, using the gradings above we see that the right and left adjoints of $\P_i$ are
$$(\P_i)_R \cong \Q_i \la -1 \ra \text{ and } (\P_i)_L \cong \Q_i \la 1 \ra.$$

In addition to the isotopy relations we have the following extra relations. First, dots are allowed to move freely through crossings.
Next, degree one dots on different strands supercommute when they pass each other, meaning that they pick up the sign $(-1)^{ab}$ where $a,b \in \{1,2\}$ denotes their degree. For example, if $i \neq j$ and $k \neq l$ then we have
$$
\begin{tikzpicture}[scale=.75][>=stealth]
\draw (0,0) -- (0,2)[->][very thick] ;
\filldraw [blue] (0,.66) circle (3pt);
\draw (.5,.5) node {$\hdots$};
\draw (1,0)--(1,2)[->] [very thick];
\filldraw [blue] (1,1.33) circle (3pt);
\draw (2.5,1) node {$=$};
\draw (3.5,1) node {$-$};
\draw [shift={+(2,0)}](2,0) --(2,2)[->][very thick] ;
\filldraw [shift={+(2,0)}][blue] (2,1.33) circle (3pt);
\draw [shift={+(2,0)}](2.5,.5) node {$\hdots$};
\filldraw [shift={+(2,0)}][blue] (3,.66) circle (3pt);
\draw [shift={+(2,0)}](3,0) -- (3,2)[->][very thick] ;

\draw (0,-.25) node {$i$};
\draw (1,-.25) node {$k$};
\draw (0,2.25) node {$j$};
\draw (1,2.25) node {$l$};

\draw (4,-.25) node {$i$};
\draw (5,-.25) node {$k$};
\draw (4,2.25) node {$j$};
\draw (5,2.25) node {$l$};
\end{tikzpicture},
$$
since each of these dots has degree one.
The above relation is technically not a local relation, since there may be any number of vertical strands in between the dotted strands. The remaining relations listed below are all local.

The relation which governs the composition of dots on the same strand is 
\begin{equation}\label{eq:rel0'}
\begin{tikzpicture}[scale=.75][>=stealth]
\draw (0,0) -- (0,2) [->][very thick];
\draw (0,-.25) node {$i$};
\filldraw [blue] (0,.66) circle (3pt);
\draw (-.25,1) node {$j$};
\filldraw [blue] (0,1.33) circle (3pt);
\draw (0,2.25) node {$k$};
\draw (.5,1) node {$=$};
\draw (1.33,1) node {$\delta_{ik}\epsilon_{ij}$} ;
\draw (2,0) -- (2,2) [->][very thick];
\draw (2,-.25) node {$i$};
\filldraw [blue] (2,1) circle (3pt);
\draw (2,2.25) node {$i$};
\end{tikzpicture}
\end{equation}

Next, for any $i,j,k \in \I$ we have 
\begin{equation}\label{eq:rel1'}
\begin{tikzpicture}[scale=.75][>=stealth]
\draw (0,0) .. controls (1,1) .. (0,2)[->][very thick] ;
\draw (1,0) .. controls (0,1) .. (1,2)[->] [very thick];
\draw (1.5,1) node {=};
\draw (2,0) --(2,2)[->][very thick] ;
\draw (3,0) -- (3,2)[->][very thick] ;
\draw (0,-.25) node {$i$};
\draw (1,-.25) node {$j$};
\draw (2,-.25) node {$i$};
\draw (3,-.25) node {$j$};

\draw [shift={+(5,0)}](0,0) -- (2,2)[->][very thick];
\draw [shift={+(5,0)}](2,0) -- (0,2)[->][very thick];
\draw [shift={+(5,0)}](1,0) .. controls (0,1) .. (1,2)[->][very thick];
\draw [shift={+(5,0)}](0,-.25) node {$i$};
\draw [shift={+(5,0)}](1,-.25) node {$j$};
\draw [shift={+(5,0)}](2,-.25) node {$k$};
\draw [shift={+(5,0)}](2.5,1) node {=};
\draw [shift={+(5,0)}](3,0) -- (5,2)[->][very thick];
\draw [shift={+(5,0)}](5,0) -- (3,2)[->][very thick];
\draw [shift={+(5,0)}](4,0) .. controls (5,1) .. (4,2)[->][very thick];
\draw [shift={+(5,0)}](3,-.25) node {$i$};
\draw [shift={+(5,0)}](4,-.25) node {$j$};
\draw [shift={+(5,0)}](5,-.25) node {$k$};
\end{tikzpicture}
\end{equation}

\begin{equation}\label{eq:rel2'}
\begin{tikzpicture}[scale=.75][>=stealth]
\draw [shift={+(2,1)}](0,0) arc (180:360:0.5cm) [very thick];
\draw [shift={+(2,1)}][->](1,0) arc (0:180:0.5cm) [very thick];
\filldraw [shift={+(2,1)}] [blue](1,0) circle (3pt);
\draw [shift={+(2,1)}](1.50,0) node{$= \mathbf{1}$};
\draw [shift={+(2,1)}](.5,-.75) node {$i$};
\draw [shift={+(2,1)}](.5,.75) node {$i$};

\draw  [shift={+(6,1)}](0,0) .. controls (0,.5) and (.7,.5) .. (.9,0) [very thick];
\draw  [shift={+(6,1)}](0,0) .. controls (0,-.5) and (.7,-.5) .. (.9,0) [very thick];
\draw  [shift={+(6,1)}](1,-1) .. controls (1,-.5) .. (.9,0) [very thick];
\draw  [shift={+(6,1)}](.9,0) .. controls (1,.5) .. (1,1) [->] [very thick];
\draw  [shift={+(6,1)}](1.5,0) node {$=$};
\draw  [shift={+(6,1)}](2,0) node {$0.$};
\draw (7,-0.25) node {$i$};
\end{tikzpicture}
\end{equation}

Finally, if $i \ne j$ then 
\begin{equation}\label{eq:rel3'}
\begin{tikzpicture}[scale=.75][>=stealth]
\draw (0,0) .. controls (1,1) .. (0,2)[<-][very thick];
\draw (1,0) .. controls (0,1) .. (1,2)[->] [very thick];
\draw (0,-.25) node {$i$};
\draw (1,-.25) node {$j$};
\draw (1.5,1) node {=};
\draw (2,0) --(2,2)[<-][very thick];
\draw (3,0) -- (3,2)[->][very thick];
\draw (2,-.25) node {$i$};
\draw (3,-.25) node {$j$};
\draw (3.75,1) node {$-\ \epsilon_{ij}$};

\draw (4,1.75) arc (180:360:.5) [very thick];
\draw (4,2) -- (4,1.75) [very thick];
\draw (5,2) -- (5,1.75) [very thick][<-];
\draw (5,.25) arc (0:180:.5) [very thick];
\filldraw [blue] (4.5,1.25) circle (3pt);
\filldraw [blue] (4.5,0.75) circle (3pt);
\draw (5,0) -- (5,.25) [very thick];
\draw (4,0) -- (4,.25) [very thick][<-];
\draw (4,-.25) node {$i$};
\draw (5,-.25) node {$j$};
\draw (4,2.25) node {$i$};
\draw (5,2.25) node {$j$};
\end{tikzpicture}
\end{equation}
while 
\begin{equation}\label{eq:rel4'}
\begin{tikzpicture}[scale=.75][>=stealth]
\draw (0,0) .. controls (1,1) .. (0,2)[<-][very thick];
\draw (1,0) .. controls (0,1) .. (1,2)[->] [very thick];
\draw (0,-.25) node {$i$};
\draw (1,-.25) node {$i$};
\draw (1.5,1) node {=};
\draw (2,0) --(2,2)[<-][very thick];
\draw (3,0) -- (3,2)[->][very thick];
\draw (2,-.25) node {$i$};
\draw (3,-.25) node {$i$};
\draw (3.5,1) node {$-$};
\draw (4,1.75) arc (180:360:.5) [very thick];
\draw (4,2) -- (4,1.75) [very thick];
\draw (5,2) -- (5,1.75) [very thick][<-];
\draw (5,.25) arc (0:180:.5) [very thick];
\filldraw [blue] (4.5,1.25) circle (3pt);
\draw (5,0) -- (5,.25) [very thick];
\draw (4,0) -- (4,.25) [very thick][<-];
\draw (4,-.25) node {$i$};
\draw (5,-.25) node {$i$};
\draw (4,2.25) node {$i$};
\draw (5,2.25) node {$i$};
\draw (5.5,1) node {$-$};
\filldraw [blue] (6.5,0.75) circle (3pt);
\draw (6,1.75) arc (180:360:.5) [very thick];
\draw (6,2) -- (6,1.75) [very thick];
\draw (7,2) -- (7,1.75) [very thick][<-];
\draw (7,.25) arc (0:180:.5) [very thick];
\draw (7,0) -- (7,.25) [very thick];
\draw (6,0) -- (6,.25) [very thick][<-];
\draw (6,-.25) node {$i$};
\draw (7,-.25) node {$i$};
\draw (6,2.25) node {$i$};
\draw (7,2.25) node {$i$};
\end{tikzpicture}
\end{equation}

Notice all the graphical relations are compatible with the grading assigned to generators. Relation \ref{eq:rel1'} above implies that there is a natural map $\k[S_n] \longrightarrow \End(\P_i^n)$, which will be important in the next section. 

\subsection{Idempotent completeness}

Since the underlying 2-category in a 2-representation of $\h$ is assumed to be idempotent complete, any idempotent 2-morphism $e$ of a 1-morphism $\A$ gives rise to a direct sum decomposition $\A \cong \A e \oplus \A (1-e)$.  

For example, since $\k[S_n] \rightarrow \End(\P_i^n)$, each idempotent in $\k[S_n]$ gives rise to a direct summand of $\P_i^n$.  We let $\P_i^{\mu}$ be the 1-morphism of $\K$ corresponding to a minimal idempotent of $\k[S_n]$ associated to the partition $\mu$ of $n$.  More explicitly, fix a labeling $T$ of the boxes in a Young diagram of $\mu$ with the numbers $1,\hdots,n$. Corresponding to $T$ there are the subgroups $S_{row(T)}$ and $S_{col(T)}$ of $S_n$ preserving the rows and columns respectively. These subgroups have associated Young symmetrizers
$$a_T = \sum_{g\in S_{row(T)}} g, \ \ \ \sum_{g\in S_{col(T)}} (-1)^{l(g)}g, $$
where $l(g)$ is the length of the permutation $g$.  

We set $c_T = \frac{1}{n_\mu}a_T b_T$, an idempotent in $\k[S_n]$. Here the scalar $n_\mu$ is defined as the cardinality of the set $\{(s_1,t_1,s_2,t_2) : s_i\in row(T), \ t_i\in col(T), \ \ (-1)^{l(t_1)} = (-1)^{l(t_2)}, \ \ s_1t_1s_2t_2 = 1\}$.
The idempotent $c_T$ is used to construct the irreducible $S_n$ representation $\k[S_n]c_T$ associated to the partition $\mu$. Note that $\k[S_n]c_T \cong \k[S_n]c_{T'}$ if $T$ and $T'$ are different fillings of the same partition.  Exchanging the roles of row and column, one can also use instead the idempotent $\tilde{c}_T = \frac{1}{n_\mu} b_Ta_T$ to construct the same irreducible representation (see Chapter 7 of \cite{F} for a discussion of the constructions of irreducible symmetric group representations from tableaux).

In our case, $c_T$ and $\tilde{c}_T$ also define 1-morphisms $\P_i^nc_T, \P_i\tilde{c}_T$ in $\K$. Just as for representations of the symmetric group, these 1-morphisms do not depend on the choice of labeling $T$ but only on the partition $\mu$ in the sense that $\P_i^nc_T \cong \P_i^nc_{T'}$ if $T$ and $T'$ are different labelings of $\mu$; moreover, both of these one-morphisms are isomorphic to $\P_i\tilde{c}_T$. We will abuse notation slightly and write $\P_i^\mu$ for $\P_i^nc_T$ for an arbitrary choice of $T$. The 1-morphism $\P_i^{\mu}$ will generally be drawn as
$$
\begin{tikzpicture}[scale=.75][>=stealth]
\shadedraw[gray]  (0,0) rectangle (4,.5);
\draw (2,.25) node {$(\mu)_i$};
\draw (2,.5) -- (2,1) [->][very thick];
\end{tikzpicture}
$$
When the strand label $i$ is understood, which is often the case in this paper, it will not be drawn. We define $\Q_i^{\mu} := \Q_i^n c_T$ similarly and draw this 1-morphism with downward pointing arrows. In a few cases we will need to emphasize the choice of the idempotent $c_T$, in which case we draw
$$
\begin{tikzpicture}[scale=.75][>=stealth]
\shadedraw[gray]  (-4.5,.5) rectangle (-1.5,1);
\draw (-3,.75) node {$(\mu)$};
\draw (-3,1) -- (-3,2) [->][very thick];
\draw (-.5,.75) node {$= \frac{1}{n_\mu}$};
\shadedraw[gray]  (.5,0) rectangle (3.5,.5);
\draw (2,.25) node {$b_T$};
\draw (2,.5) -- (2,1) [->][very thick];
\shadedraw[gray]  (.5,1) rectangle (3.5,1.5);
\draw (2,1.25) node {$a_T$};
\draw (2,1.5) -- (2,2) [->][very thick];
\end{tikzpicture}.
$$
Of particular importance are the elements $\P_i^{(n)}, \P_i^{(1^n)}, \Q_i^{(n)}$ and $\Q_i^{(1^n)}$ corresponding to the partitions $\mu = (n)$ and $\mu = (1^n) = (1,1,\hdots,1)$ (i.e. corresponding to the trivial and sign representations of $S_n$). To emphasize the difference between $\P_i^{(n)}$ and $\P_i^{(1^n)}$ in future calculations we will shade the box for $\P_i^{(1^n)}$ while leaving the box for $\P_i^{(n)}$ unshaded (this shading of the boxes is for visual convenience only). 

The $\P$s and the $\Q$s are biadjoint to each other in $\K$. Precisely, if $\mu$ is an arbitrary partition then the left and right adjoints are given by
$$(\P_i^{\mu})_R \cong \Q_i^{\mu} \la -|\mu| \ra \text{ and } (\P_i^{\mu})_L \cong \Q_i^{\mu} \la |\mu| \ra$$
where $|\mu|$ denotes the size of the partition. 

\subsubsection{Integrability}

A 2-representation of $\h$ is said to be \emph{integrable} if the object $\1_n$ is zero for $n \ll 0$. It is \emph{weakly integrable} if it is the direct sum of (possibly infinitely many) integrable 2-representations. 

In this paper, all 2-representations are assumed to be weakly integrable. An example of such a 2-representation of $\h$ was constructed in \cite{CLi1} on categories of coherent sheaves on Hilbert schemes of points on the surface $\widehat{\bA^2/\Gamma}$.

\subsubsection{A symmetry of 2-representations}\label{sec:psi}

A 2-representation of $\h$ admits a covariant involution $\Psi: \K \rightarrow \K$. It is defined as the identity on objects and 1-morphisms and also the identity on cups, caps, and dots while acting as multiplication by $-1$ on any crossing. This means that $\Psi$ takes the idempotent 2-morphism $c_{(1^n)}$ to $c_{(n)}$ and vice versa. Subsequently
$$ \Psi(\P_i^{(n)}) = \P_i^{(1^n)}, \ \ \Psi(\P_i^{(1^n)}) = \P_i^{(n)}, \ \ \Psi(\Q_i^{(n)}) = \Q_i^{(1^n)}, \ \ \Psi(\Q_i^{(1^n)}) = \Q_i^{(n)} $$
while more generally $\Psi(\P_i^{(\l)}) = \P_i^{(\l^t)}$ and $\Psi(\Q_i^{(\l)}) = \Q_i^{(\l^t)}$. Thus $\Psi$ categorifies the involution $\psi$ from Section \ref{sec:qhei}.

\subsubsection{Induced relations among 1-morphisms in $\K$}

The relations among 2-morphisms in a 2-representation of $\h$ imply certain isomorphisms between 1-morphisms. We recall some of these relations below. 

\begin{prop}\label{prop:rels1}\cite{CLi1} 
We have the following direct sum decompositions of 1-morphisms in $\H$:
\begin{enumerate}
\item $\P_i^{(m)}\P_j^{(n)} \cong 
\begin{cases} 
\P_i^{(n)}\P_i^{(m)} \cong \oplus_{k=0}^{\min(n,m)} \P_i^{(n+m-k,k)} \text{ if } i=j \\
\P_i^{(n)}\P_j^{(m)} \text{ if } i \ne j 
\end{cases}$
\item $\P_i^{(m)} \P_j^{(1^n)} \cong 
\begin{cases} 
\P_i^{(m,1^n)} \oplus \P_i^{(m+1,1^{n-1})} \text{ if } i=j \\
\P_j^{(1^n)} \P_i^{(m)} \text{ if } i \ne j
\end{cases}$
\item $\Q_j^{(n)} \P_i^{(m)} \cong 
\begin{cases}
\oplus_{k \ge 0} \P_i^{(m-k)} \Q_i^{(n-k)} \otimes_\k V_k \text{ if } i=j \\
\P_i^{(m)} \Q_j^{(n)} \oplus \P_i^{(m-1)} \Q_j^{(n-1)} \text{ if } \la i,j \ra = -1 \\
\P_i^{(m)} \Q_j^{(n)} \text{ if } \la i,j \ra = 0 
\end{cases}$ 
\item $\Q_j^{(1^n)} \P_i^{(m)} \cong 
\begin{cases} 
\P_i^{(m)} \Q_i^{(1^n)} \oplus \P_i^{(m-1)} \Q_i^{(1^{n-1})} \otimes_\k V_1 \oplus \P_i^{(m-2)} \Q_i^{(1^{n-2})} \text{ if } i=j \\
\oplus_{k \ge 0} \P_i^{(m-k)} \Q_j^{(1^{n-k})} \text{ if } \la i,j \ra = -1 \\
\P_i^{(m)} \Q_j^{(1^n)} \text{ if } \la i,j \ra = 0
\end{cases}$ 
\end{enumerate}
In each case the direct summands on the right hand side are indecomposable 1-morphisms in $\K$.
\end{prop}

\subsection{Technical lemmas}

We now discuss a few technical lemmas dealing with 2-representations of $\h$. These will be used later but can be skipped on a first reading of the paper. 

\begin{lemma}\label{lem:PQrels2}
For an arbitrary partition $\l$ we have
$$\Q_i^{(\l)} \P_i \cong \P_i \Q_i^{(\l)} \bigoplus_{\l' \subset \l} \Q_i^{(\l')} \otimes_\k V_1 \text{ and } \Q_i \P_i^{(\l)} \cong \P_i^{(\l)} \Q_i \bigoplus_{\l' \subset \l} \P_i^{(\l')} \otimes_\k V_1$$
where the sums are over all $\l' \subset \l$ with $|\l'| = |\l|-1$. 
\end{lemma}
\begin{proof}
We prove only the first relation (the second one is proved similarly). First, using the fact that $\P_i\Q_i^{(\l)}$ is indecomposable and that $\Q_i \P_i \1_n \cong \P_i \Q_i \1_n \oplus \1_n \otimes_\k V_1$, it follows by induction that 
$$\Q_i^{(\l)} \P_i \cong \P_i \Q_i^{(\l)} \bigoplus_{\mu \subset \l} \Q_i^{(\mu)} \otimes_\k V_{\mu, \l}$$ 
for some graded vector spaces $V_{\mu, \l}$. We will prove by induction on $|\l|$ that $V_{\mu,\l} = V_1$ if $\mu \subset \l$ and that 
$V_{\mu,\l} = 0$ otherwise. 

Note that if $\Q_i^{(\mu)} \la l \ra$ is a summand of $\Q_i^{(\l)} \P_i$ then we must have
$$\Hom(\Q_i^{(\l)} \P_i, \Q_i^{(\mu)} \la l \ra) \ne 0 \text{ and } 
\Hom(\Q_i^{(\mu)} \la l \ra, \Q_i^{(\l)} \P_i) \ne 0.$$
Now, by adjunction
\begin{eqnarray*}
\Hom(\Q_i^{(\l)} \P_i, \Q_i^{(\mu)} \la l \ra) 
&\cong& \Hom(\Q_i^{(\l)}, \Q_i^{(\mu)} \Q_i \la l+1 \ra) \\
&\cong& \Hom(\Q_i^{(\l)}, \oplus_{\mu \subset \mu'} \Q_i^{(\mu')} \la l+1 \ra) 
\end{eqnarray*}
and similarly
$$\Hom(\Q_i^{(\mu)} \la l \ra, \Q_i^{(\l)} \P_i) \cong \Hom( \oplus_{\mu \subset \mu'} \Q_i^{(\mu')}, \Q_i^{(\l)} \la -l+1 \ra).$$
One of these two morphism spaces is zero unless $l=-1$ or $l=1$ and $\mu \subset \l$ in which case one of these is one dimensional. Thus $V_{\mu,\l} = 0$ unless $\mu \subset \l$ in which case $V_{\mu,\l} \subset V_1$. 

It remains to show that $V_{\mu,\l}$ actually equals $V_1$. We do this by counting dimensions. Take $\nu \subset \l$ with $|\l| = |\nu|+1$. Then by induction
$$\Q_i^{(\nu)} \Q_i \P_i \cong \P_i \Q_i^{(\nu)} \Q_i \oplus \Q_i^{(\nu)} \otimes_\k V_1 \oplus_{\nu' \subset \nu} \Q_i^{(\nu')} \Q_i \otimes_\k V_1.$$
On the other hand, this equals 
$$\oplus_{\nu \subset \l'} \Q_i^{(\l')} \P_i \cong \oplus_{\nu \subset \l'} \P_i \Q_i^{(\l')} \oplus_{\nu, \l'' \subset \l'} \Q_i^{(\l'')} \otimes_\k V_{\l'',\l'}.$$  
Comparing summands involving only $\Q_i$'s one can check that indeed $\dim(V_{\mu,\l'}) = 2$ for any $\mu \subset \l'$. The result follows. 
\end{proof}

\begin{lemma}\label{lem:differentials}
Suppose $\l,\l',\mu$ and $\mu'$ are partitions such that $|\l| > |\l'|$ and $|\mu|>|\mu'|$. Then $\dim \Hom( \P_i^{(\l)} \Q_i^{(\mu)}, \P_i^{(\l')} \Q_i^{(\mu')} \la 1 \ra) \le 1$ with equality if and only if $\l' \subset \l$ and $\mu' \subset \mu$ with $|\l| = |\l'|+1$ and $|\mu|=|\mu'|+1$. In this case this space is spanned by the diagram given by a single cap. 

Likewise, if $|\l|<|\l'|$ and $|\mu|<|\mu'|$ then $\dim \Hom( \P_i^{(\l)} \Q_i^{(\mu)}, \P_i^{(\l')} \Q_i^{(\mu')} \la 1 \ra) \le 1$ with equality if and only if $\l' \supset \l$ and $\mu' \supset \mu$ with $|\l| = |\l'|-1$ and $|\mu|=|\mu'|-1$. In this case this space is spanned by the diagram given by a single cup.
\end{lemma}
\begin{proof}
We prove only the first assertion (as the second one follows similarly). If $|\l|-|\l'| \ge 2$ then the space of degree one maps is zero since any map requires at least two caps and thus has degree at least two. So from now on we assume that $|\l| = |\l'|+1$ and hence $|\mu|=|\mu'|+1$. 

First we show by induction on $|\mu|$ that if $\mu' \not\subset \mu$ then the space of maps is zero. Choose any $\l'' \subset \l'$ with $|\l'| = |\l''|+1$. Since $\Q_i \Q_i^{(\l'')}$ contains $\Q_i^{(\l')}$ as a direct summand it suffices to show that 
$$\Hom(\P_i^{(\l)} \Q_i^{(\mu)}, \P_i \P_i^{(\l'')} \Q_i^{(\mu')} \la 1 \ra) = 0.$$ 
By adjunction and Lemma \ref{lem:PQrels2} we have 
\begin{eqnarray*}
&& \Hom(\P_i^{(\l)} \Q_i^{(\mu)}, \P_i \P_i^{(\l'')} \Q_i^{(\mu')} \la 1 \ra) \\
&\cong& \Hom(\Q_i \P_i^{(\l)} \Q_i^{(\mu)}, \P_i^{(\l'')} \Q_i^{(\mu')} ) \\
&\cong& \Hom(\P_i^{(\l)} \Q_i \Q_i^{(\mu)}, \P_i^{(\l'')} \Q_i^{(\mu')} ) \bigoplus_{\l''' \subset \l'} \Hom(\P_i^{(\l''')} \Q_i^{(\mu)}, \P_i^{(\l'')} \Q_i^{(\mu')} \otimes_\k V_1) \\
&\cong& \bigoplus_{\mu'' \subset \mu} \Hom(\P_i^{(\l)} \Q_i^{(\mu'')}, \P_i^{(\l'')} \Q_i^{(\mu')} ) \bigoplus_{\l''' \subset \l'} \Hom(\P_i^{(\l''')} \Q_i^{(\mu)}, \P_i^{(\l'')} \Q_i^{(\mu')} \otimes_\k V_1).
\end{eqnarray*}
The terms in the first sum vanish since $\mu''$ and $\mu'$ are never equal, while those in the second sum vanish by induction. Likewise, one can show that the space of maps is also zero if $\l' \not\subset \l$. 

Now suppose $\l' \subset \l$ and $\mu' \subset \mu$. It follows from degree consideration that any map $\P_i^{(\l)} \Q_i^{(\mu)} \rightarrow \P_i^{(\l')} \Q_i^{(\mu')} \la 1 \ra$ consists of a single cap composed with a diagram without local minima or local maxima. Thus we have a diagram like the following
$$
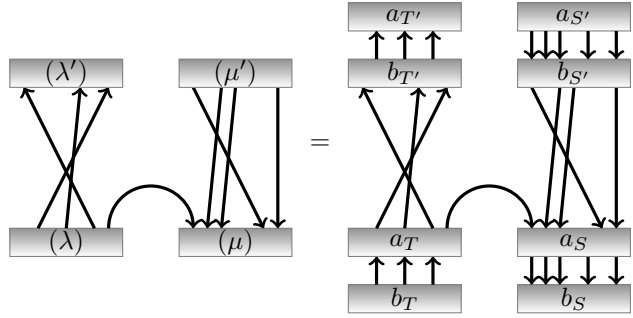
\begin{tikzpicture}[scale=.75][>=stealth]
\shadedraw[gray]  (0,0) rectangle (2,.5);
\draw (1,.25) node {$(\l)$};
\shadedraw[gray]  (3,0) rectangle (5,.5);
\draw (4,.25) node {$(\mu)$};
\draw (1,.5) -- (1.25,3) [->][very thick];
\draw (.5,.5) -- (1.75,3) [->][very thick];
\draw (1.5,.5) -- (.25,3) [->][very thick];

\draw (3.75,.5) -- (4,3) [<-][very thick];
\draw (3.5,.5) -- (3.75,3) [<-][very thick];
\draw (4.5,.5) -- (3.25,3) [<-][very thick];
\draw (4.75,.5) -- (4.75,3) [<-][very thick];

\shadedraw[gray]  (0,3) rectangle (2,3.5);
\draw (1,3.25) node {$(\l')$};
\shadedraw[gray]  (3,3) rectangle (5,3.5);
\draw (4,3.25) node {$(\mu')$};

\draw[<-](3.25,.5) arc (0:180:.75cm) [very thick];

\draw (5.5,2) node {$=$};

\shadedraw[gray]  [shift = {+(6,-1)}](0,0) rectangle (2,.5);
\draw [shift = {+(6,-1)}](1,.25) node {$b_T$};

\draw [shift = {+(6,0)}](1.5,-.5) -- (1.5,0) [->][very thick];
\draw [shift = {+(6,0)}](.5,-.5) -- (.5,0) [->][very thick];
\draw [shift = {+(6,0)}](1,-.5) -- (1,0) [->][very thick];

\shadedraw[gray]  [shift = {+(6,-1)}](3,0) rectangle (5,.5);
\draw [shift = {+(6,-1)}](4,.25) node {$b_S$};

\draw [shift = {+(9,0)}](1.75,-.5) -- (1.75,0) [<-][very thick];
\draw [shift = {+(9,0)}](.75,-.5) -- (.75,0) [<-][very thick];
\draw [shift = {+(9,0)}](1.25,-.5) -- (1.25,0) [<-][very thick];
\draw [shift = {+(9,0)}](.25,-.5) -- (.25,0) [<-][very thick];
\draw [shift = {+(9,0)}](.5,-.5) -- (.5,0) [<-][very thick];

\shadedraw[gray]  [shift = {+(6,0)}](0,0) rectangle (2,.5);
\draw [shift = {+(6,0)}](1,.25) node {$a_T$};
\shadedraw[gray]  [shift = {+(6,0)}](3,0) rectangle (5,.5);
\draw [shift = {+(6,0)}](4,.25) node {$a_S$};
\draw [shift = {+(6,0)}](1,.5) -- (1.25,3) [->][very thick];
\draw [shift = {+(6,0)}](.5,.5) -- (1.75,3) [->][very thick];
\draw [shift = {+(6,0)}](1.5,.5) -- (.25,3) [->][very thick];

\draw [shift = {+(6,0)}](3.75,.5) -- (4,3) [<-][very thick];
\draw [shift = {+(6,0)}](3.5,.5) -- (3.75,3) [<-][very thick];
\draw [shift = {+(6,0)}](4.5,.5) -- (3.25,3) [<-][very thick];
\draw [shift = {+(6,0)}](4.75,.5) -- (4.75,3) [<-][very thick];

\shadedraw[gray]  [shift = {+(6,0)}](0,3) rectangle (2,3.5);
\draw [shift = {+(6,0)}](1,3.25) node {$b_{T'}$};
\shadedraw[gray] [shift = {+(6,0)}] (3,3) rectangle (5,3.5);
\draw [shift = {+(6,0)}](4,3.25) node {$b_{S'}$};

\draw[shift = {+(6,0)}][<-](3.25,.5) arc (0:180:.75cm) [very thick];

\shadedraw[gray]  [shift = {+(6,4)}](0,0) rectangle (2,.5);
\draw [shift = {+(6,4)}](1,.25) node {$a_{T'}$};
\shadedraw[gray]  [shift = {+(9,4)}](0,0) rectangle (2,.5);
\draw [shift = {+(9,4)}](1,.25) node {$a_{S'}$};

\draw [shift = {+(6,4)}](1.5,-.5) -- (1.5,0) [->][very thick];
\draw [shift = {+(6,4)}](.5,-.5) -- (.5,0) [->][very thick];
\draw [shift = {+(6,4)}](1,-.5) -- (1,0) [->][very thick];

\draw [shift = {+(9,4)}](1.75,-.5) -- (1.75,0) [<-][very thick];
\draw [shift = {+(9,4)}](.75,-.5) -- (.75,0) [<-][very thick];
\draw [shift = {+(9,4)}](1.25,-.5) -- (1.25,0) [<-][very thick];
\draw [shift = {+(9,4)}](.25,-.5) -- (.25,0) [<-][very thick];
\draw [shift = {+(9,4)}](.5,-.5) -- (.5,0) [<-][very thick];
\end{tikzpicture}
$$
After possibly replacing the filling $T'$ by the filling $gT'$ for some permutation $g$, we may assume that the upward pointing strands do not intersect. Similarly, replacing $S'$ by $hS'$ for some permutation $h$ we can assume the downward strands do not intersect. Thus we must show that the diagram
 
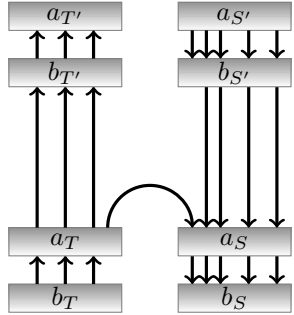
\begin{equation}\label{eq:span}
 \begin{tikzpicture}[scale=.75]
 \shadedraw[gray]  [shift = {+(6,-1)}](0,0) rectangle (2,.5);
\draw [shift = {+(6,-1)}](1,.25) node {$b_T$};

\draw [shift = {+(6,0)}](1.5,-.5) -- (1.5,0) [->][very thick];
\draw [shift = {+(6,0)}](.5,-.5) -- (.5,0) [->][very thick];
\draw [shift = {+(6,0)}](1,-.5) -- (1,0) [->][very thick];

\shadedraw[gray]  [shift = {+(6,-1)}](3,0) rectangle (5,.5);
\draw [shift = {+(6,-1)}](4,.25) node {$b_S$};

\draw [shift = {+(9,0)}](1.75,-.5) -- (1.75,0) [<-][very thick];
\draw [shift = {+(9,0)}](.75,-.5) -- (.75,0) [<-][very thick];
\draw [shift = {+(9,0)}](1.25,-.5) -- (1.25,0) [<-][very thick];
\draw [shift = {+(9,0)}](.25,-.5) -- (.25,0) [<-][very thick];

\shadedraw[gray]  [shift = {+(6,0)}](0,0) rectangle (2,.5);
\draw [shift = {+(6,0)}](1,.25) node {$a_T$};
\shadedraw[gray]  [shift = {+(6,0)}](3,0) rectangle (5,.5);
\draw [shift = {+(6,0)}](4,.25) node {$a_S$};
\draw [shift = {+(6,0)}](1,.5) -- (1,3) [->][very thick];
\draw [shift = {+(6,0)}](.5,.5) -- (.5,3) [->][very thick];
\draw [shift = {+(6,0)}](1.5,.5) -- (1.5,3) [->][very thick];

\draw [shift = {+(9,0)}](1.75,.5) -- (1.75,3) [<-][very thick];
\draw [shift = {+(9,0)}](.75,.5) -- (.75,3) [<-][very thick];
\draw [shift = {+(9,0)}](1.25,.5) -- (1.25,3) [<-][very thick];
\draw [shift = {+(9,0)}](.5,.5) -- (.5,3) [<-][very thick];

\draw [shift = {+(9,0)}](.5,-.5) -- (.5,0) [<-][very thick];

\shadedraw[gray]  [shift = {+(6,0)}](0,3) rectangle (2,3.5);
\draw [shift = {+(6,0)}](1,3.25) node {$b_{T'}$};
\shadedraw[gray] [shift = {+(6,0)}] (3,3) rectangle (5,3.5);
\draw [shift = {+(6,0)}](4,3.25) node {$b_{S'}$};

\draw[shift = {+(6,0)}][<-](3.25,.5) arc (0:180:.75cm) [very thick];

\shadedraw[gray]  [shift = {+(6,4)}](0,0) rectangle (2,.5);
\draw [shift = {+(6,4)}](1,.25) node {$a_{T'}$};
\shadedraw[gray]  [shift = {+(9,4)}](0,0) rectangle (2,.5);
\draw [shift = {+(9,4)}](1,.25) node {$a_{S'}$};

\draw [shift = {+(6,4)}](1.5,-.5) -- (1.5,0) [->][very thick];
\draw [shift = {+(6,4)}](.5,-.5) -- (.5,0) [->][very thick];
\draw [shift = {+(6,4)}](1,-.5) -- (1,0) [->][very thick];

\draw [shift = {+(9,4)}](1.75,-.5) -- (1.75,0) [<-][very thick];
\draw [shift = {+(9,4)}](.75,-.5) -- (.75,0) [<-][very thick];
\draw [shift = {+(9,4)}](1.25,-.5) -- (1.25,0) [<-][very thick];
\draw [shift = {+(9,4)}](.25,-.5) -- (.25,0) [<-][very thick];
\draw [shift = {+(9,4)}](.5,-.5) -- (.5,0) [<-][very thick];

\end{tikzpicture}
\end{equation}
which spans the space of maps. 

It remains to show that $\dim \Hom(\P_i^{(\l)} \Q_i^{(\mu)}, \P_i^{(\l')} \Q_i^{(\mu')} \la 1 \ra) = 1$. Now 
\begin{eqnarray*}
&& \Hom(\P_i \P_i^{(\l')} \Q_i^{(\mu)}, \P_i^{(\l')} \Q_i^{(\mu')} \la 1 \ra) \\
&\cong& \Hom(\P_i^{(\l')} \Q_i^{(\mu)}, \Q_i \P_i^{(\l')} \Q_i^{(\mu')}) \\ 
&\cong& \Hom(\P_i^{(\l')} \Q_i^{(\mu)}, \P_i^{(\l')} \Q_i \Q_i^{(\mu')}) \bigoplus_{\l'' \subset \l'} \Hom(\P_i^{(\l')} \Q_i^{(\mu)}, \P_i^{(\l'')} \Q_i^{(\mu')} \otimes_\k V_1) 
\end{eqnarray*}
Now the left hand term is isomorphic to $\k$ while, by induction, the right hand term is isomorphic to $\k^\ell$ where $\ell$ is the number of $\l'' \subset \l'$. Thus 
$$\dim \Hom(\P_i \P_i^{(\l')} \Q_i^{(\mu)}, \P_i^{(\l')} \Q_i^{(\mu')} \la 1 \ra) = \ell+1.$$
On the other hand, $\P_i \P_i^{(\l')} \cong \oplus_{\l' \subset \nu} \P_i^{(\nu)}$. Since the number of such partitions $\nu$ is $\ell+1$ it follows that 
$$\dim \Hom(\P_i^{(\nu)} \Q_i^{(\mu)}, \P_i^{(\l')} \Q_i^{(\mu')} \la 1 \ra) = 1$$
for any $\nu$ containing $\l'$. The result follows since we can take $\nu = \l$.
\end{proof}

\begin{lemma}\label{lem:onedimhoms}
If $i \ne j$ then $\dim \Hom(\P_i^{(a)} \P_j^{(b)} \Q_i^{(1^c)} \Q_j^{(1^d)}, \P_i^{(a')} \P_j^{({b'})} \Q_i^{(1^{c'})} \Q_j^{(1^{d'})} \la 1 \ra) \le 1$. 
\end{lemma}
\begin{proof}
Let $D$ be a diagram depicting a 2-morphism in 
$$\Hom(\P_i^{(a)} \P_j^{(b)} \Q_i^{(1^c)} \Q_j^{(1^d)}, \P_i^{(a')} \P_j^{({b'})} \Q_i^{(1^{c'})} \Q_j^{(1^{d'})} \la 1 \ra).$$  
Because all $P$'s occur to the left of all $Q$'s one can simplify $D$ so that it has no right-pointing cups or left-pointing caps. One can also get rid of all degree zero crossings. The remaining map is made up of dots, cups and caps all of which have positive degree. Of these only the the $ij$ dot, the $ji$ dot, the right cap and the left cup have degree one. It follows that $D$ is made up precisely of one such map and hence the $\Hom$ space is zero except in the following cases:
\begin{enumerate}
\item $(a',b',c',d') = (a\pm 1,b\mp 1,c,d)$,
\item $(a',b',c',d') = (a ,b,c\pm 1,d \mp 1)$
\item $(a',b',c',d') = (a \pm 1,b ,c \pm 1,d)$
\item $(a',b',c',d') = (a,b \pm 1,c,d \pm 1)$
\end{enumerate}
\end{proof}

The proof above shows, for example, that $\Hom(\P_i^{(a)} \P_j^{(b)} \Q_i^{(1^c)} \Q_j^{(1^d)}, \P_i^{(a-1)} \P_j^{({b+1})} \Q_i^{(1^{c})} \Q_j^{(1^d)} \la 1 \ra)$ is spanned by the diagram
\begin{equation}\label{eq:nonzeromap}
\begin{tikzpicture}[scale=.75]
\draw (-.25,1) rectangle (1.75,1.5);
\draw (.75,1.25) node {$(a)_i$};
\filldraw[gray] (2.25,1) rectangle (4.25,1.5);
\draw (3.25,1.25) node {$(b)_j$};

\filldraw [gray](4.4,1) rectangle (6.1,1.5);
\draw (5.25,1.25) node {$(1^{c})_i$};
\draw (6.5,1) rectangle (8.5,1.5);
\draw (7.5,1.25) node {$(1^d)_j$};

\draw [shift= {+(0,3)}] (-.25,1) rectangle (1.75,1.5);
\draw [shift= {+(0,3)}] (.75,1.25) node {$(a-1)_i$};
\filldraw[gray] [shift= {+(0,3)}] (2.25,1) rectangle (4.25,1.5);
\draw [shift= {+(0,3)}] (3.25,1.25) node {$({b+1})_j$};

\filldraw [gray][shift= {+(0,3)}] (4.4,1) rectangle (6.1,1.5);
\draw [shift= {+(0,3)}] (5.25,1.25) node {$(1^{c})_i$};
\draw [shift= {+(0,3)}] (6.5,1) rectangle (8.5,1.5);
\draw [shift= {+(0,3)}] (7.5,1.25) node {$(1^d)_j$};

\draw (.5,1.5) -- (.5,4) [->][very thick];
\draw (3.5,1.5) -- (3.5,4) [->][very thick];
\draw (5.25,1.5) -- (5.25,4) [<-][very thick];
\draw (7.5,1.5) -- (7.5,4) [<-][very thick];

\draw (1,1.5) -- (3,4) [->][very thick];
\filldraw [blue] (2,2.75) circle (3pt);
\end{tikzpicture}
\end{equation}
where the dot is an $i-j$ dot.  A basic question is how to check that the diagram above is a nonzero 2-morphism. There are two ways to do this. One way is to check directly in some 2-representation of $\H$ that this diagram is represents a nonzero 2-morphism. A representation of $\H$ which is faithful on 2-morphisms was defined in \cite{CLi1}, so in principle one can check that the above diagram becomes a nonzero map in that representation.

A second proof that the above diagram is nonzero proceeds by closing off the diagram, simplifying using the graphical relations, and then showing that the resulting diagram is nonzero. Since we use this technique several times in future sections we now explain in this example how this procedure works. First note that it suffices to show that the diagram
 $$
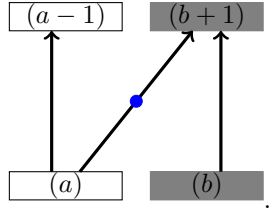
\begin{tikzpicture}[scale=.75]
\draw (-.25,1) rectangle (1.75,1.5);
\draw (.75,1.25) node {$(a)$};
\filldraw[gray] (2.25,1) rectangle (4.25,1.5);
\draw (3.25,1.25) node {$(b)$};

\draw [shift= {+(0,3)}] (-.25,1) rectangle (1.75,1.5);
\draw [shift= {+(0,3)}] (.75,1.25) node {$(a-1)$};
\filldraw[gray] [shift= {+(0,3)}] (2.25,1) rectangle (4.25,1.5);
\draw [shift= {+(0,3)}] (3.25,1.25) node {$({b+1})$};

\draw (.5,1.5) -- (.5,4) [->][very thick];
\draw (3.5,1.5) -- (3.5,4) [->][very thick];
\draw (1,1.5) -- (3,4) [->][very thick];
\filldraw [blue] (2,2.75) circle (3pt);

\end{tikzpicture}.
$$
is nonzero. To do this we close it off as shown in figure \ref{fig:close} below to get an endomorphism of the identity. Note that that there are two $i-j$ dots in the diagram (the ones on the two diagonal line segments) while the remaining ones are either $i-i$ or $j-j$ dots.  
\begin{equation}\label{fig:close} 
\begin{tikzpicture}[scale=.75][>=stealth]

\draw[->](-3.25,0) arc (180:360:0.75cm) [very thick];
\draw[->](-3.5,0) arc (180:360:1cm) [very thick];

\draw (2,.75) node {$\hdots$};

\draw[->](-4.5,0) arc (180:360:2cm) [very thick];
\draw[->](-4.75,0) arc (180:360:2.25cm) [very thick];

\draw[<-] (-4.75,0)--(-6.25,1.5)[very thick];
\draw[<-] (-6.5,0)--(-6.5,1.5)[very thick];
\draw[<-] (-7.5,0)--(-7.5,1.5)[very thick];

\filldraw[gray]  [shift={+(3,0)}](-.3,0) rectangle (-1.85,.5);
\draw  (-.2,0) rectangle (-1.75,.5);
\draw  (-.95,.25) node {$(a)$};
\draw  (2,.25) node {$({b})$};

\draw[->] (1.5,.5)--(1.5,1)[very thick];
\draw[->] (2.5,.5)--(2.5,1)[very thick];

\draw[->] (-.5,.5)--(-.5,1)[very thick];
\draw[->] (-.25,.5)--(1.25,1)[very thick];
\draw[->] (-1.5,.5)--(-1.5,1)[very thick];

\filldraw [blue] (.5,.75) circle (3pt);

\draw[shift={+(-6,0)}][<-] (1.5,0)--(1.5,1.5)[very thick];
\draw[shift={+(-6,0)}][<-] (2.5,0)--(2.5,1.5)[very thick];
\draw[shift={+(-6,0)}][<-] (2.75,0)--(2.75,1.5)[very thick];
\draw [shift={+(-3,0)}](-1,.75) node {$\hdots$};

\filldraw[gray]  [shift={+(3,0)}](-.3,1) rectangle (-1.85,1.5);
\draw  (-.2,1) rectangle (-1.75,1.5);
\draw  (-.95,1.25) node {$(a-1)$};
\draw  (2,1.25) node {$({b+1})$};

\draw (-1,.75) node {$\hdots$};

\draw[->][shift={+(-3,0)}](1.25,1.5) arc (0:180:0.75cm) [very thick];
\draw[->][shift={+(-3,0)}](1.5,1.5) arc (0:180:1cm) [very thick];
\draw[->][shift={+(-3,0)}](2.5,1.5) arc (0:180:2cm) [very thick];

\draw[->](-6.5,0) arc (180:360:4cm) [very thick];
\draw[->](-7.5,0) arc (180:360:5cm) [very thick];
\draw[<-](-6.5,1.5) arc (180:0:4cm) [very thick];
\draw[<-](-6.25,1.5) arc (180:0:3.75cm) [very thick];
\draw[<-](-7.5,1.5) arc (180:0:5cm) [very thick];

\draw (2,.75) node {$\hdots$};
\draw[->](-4.5,0) arc (180:360:2cm) [very thick];
\filldraw [blue] (-5.5,.75) circle (3pt);

\filldraw [blue] (-2.5,-1) circle (3pt);
\filldraw [blue] (-2.5,-2) circle (3pt);
\filldraw [blue] [shift={+(-3,0)}](.5,-.75) circle (3pt);

\filldraw[shift={+(0,-3)}] [blue] (-2.5,-1) circle (3pt);
\filldraw [shift={+(0,-3)}][blue] (-2.5,-2) circle (3pt);
\end{tikzpicture}
\end{equation}
This procedure can be thought of as defining a linear map
$$\Hom(\P_i^{(a)} \P_j^{(b)} \1_n, \P_i^{(a-1)} \P_j^{({b+1})} \1_n \la 1 \ra) \longrightarrow \Hom(\1_n,\1_n).$$
We now explain why this defines a nonzero multiple of the identity endomorphism of $\1_n$. First, the idempotent labeled $(a-1)$ can be absorbed into the idempotent $(a)$ and the idempotent $(b)$ can be absorbed into the idempotent $(b+1)$. Now, expanding the remaining idempotents $(a)$ and $(b+1)$ explicitly as a sum of permutations we see that almost all of the summands contain either two degree 2 dots on the same strand (giving zero) or a left-twist curl (which is also zero). Some of the remaining terms may have double crossings between strands but these double crossings can all be removed using the graphical relations. 

The only remaining diagrams are a collection of disjoint counterclockwise circles with a degree $2$ dot on them. Each such circle is equal to the identity and hence can be erased. This leaves us with the empty diagram, which is the identity endomorphism of $\1_n$ and is therefore nonzero.

\begin{lemma}\label{lem:square}
Suppose $\partial_1$ and $\partial_2$ are nonzero 2-morphisms
$$\partial_1: \P_i^{(a)} \P_j^{(b)} \Q_i^{(1^c)} \Q_j^{(1^d)} \rightarrow \P_i^{(a')} \P_j^{({b'})} \Q_i^{(1^{c'})} \Q_j^{(1^{d'})} \la 1 \ra
$$
$$\partial_2 : \P_i^{(a)} \P_j^{(b)} \Q_i^{(1^c)} \Q_j^{(1^d)} \rightarrow \P_i^{(a'')} \P_j^{({b''})} \Q_i^{(1^{c''})} \Q_j^{(1^{d''})} \la 1 \ra.
$$
Then there exist maps $\partial_1', \partial_2'$ that form a commutative square
\begin{equation*}
\begin{CD}
\P_i^{(a)} \P_j^{(b)} \Q_i^{(1^c)} \Q_j^{(1^d)} @>\partial_1>> \P_i^{(a')} \P_j^{({b'})} \Q_i^{(1^{c'})} \Q_j^{(1^{d'})} \la 1 \ra \\
@VV\partial_2 V @VV\partial_2'V \\
\P_i^{(a'')} \P_j^{({b''})} \Q_i^{(1^{c''})} \Q_j^{(1^{d''})} \la 1 \ra @>\partial_1'>> \P_i^{(a''')} \P_j^{({b'''})} \Q_i^{(1^{c'''})} \Q_j^{(1^{d'''})} \la 2 \ra
\end{CD}
\end{equation*}
so that the compositions $\partial_2' \circ \partial_1 = \partial_1' \circ \partial_2$ are nonzero. 
\end{lemma}
\begin{proof}
This follows from Lemma \ref{lem:onedimhoms}. For example, suppose $\partial_1$ is the diagram consisting of an $i-j$ dot and $\partial_2$ is the diagram consisting of a right cap labeled $i$. Then $\partial_2'$ is defined as a right cap while $\partial_1'$ is now an $i-j$ dot.  

In the composition $\partial_2' \circ \partial_1$ the dot lies below the right cap while in $\partial_1' \circ \partial_2$ the right cap lies below the dot. These diagrams are equivalent just because of the isotopy relation which allows one to slide the portion of the diagram containing the dot past the portion containg the right cup. Similar equivalences can be checked for any other pair of nonzero 2-morphisms $\partial_1, \partial_2$.

Finally, one can check that these compositions are nonzero just like we proved that the map (\ref{eq:nonzeromap}) is nonzero. 
\end{proof}

\subsection{Complexes in a 2-representation of $\h$}\label{sec:Kom(H)}

Since the underlying 2-category $\K$ of a 2-representation of $\h$ is a graded, additive, $\k$-linear 2-category, its homotopy category $\Kom(\K)$ is a graded, additive, $\k$-linear triangulated 2-category. The objects of $\Kom(\K)$ are the objects of $\K$, the 1-morphisms are (possibly unbounded) complexes of 1-morphisms of $\K$ and 2-morphisms are chain maps up to homotopy.  

Many such complexes arise naturally in categorification of quantum affine algebras. We give an example of such a complex below. This complex will be relevant later. 

Fix $i \in \I$. We define complexes
\begin{eqnarray}
\label{eq:Pcpx}
\P_i^{[1^k]} &:=& 
\left[ \P_i^{(k)} \la -2(k-1) \ra \rightarrow \dots \rightarrow \P_i^{(3,1^{k-3})} \la -4 \ra \rightarrow \P_i^{(2,1^{k-2})} \la -2 \ra \rightarrow \P_i^{(1^k)} \right] [1] \la -1 \ra \\
\label{eq:Qcpx}
\Q_i^{[1^k]} &:=& 
\left[ \Q_i^{(1^k)} \rightarrow \Q_i^{(2,1^{k-2})} \la 2 \ra \rightarrow \Q_i^{(3,1^{k-3})} \la 4 \ra \rightarrow \dots \rightarrow \Q_i^{(k)} \la 2(k-1) \ra \right] [-1] \la 1 \ra 
\end{eqnarray}
where the right hand term in (\ref{eq:Pcpx}) and the left hand term in (\ref{eq:Qcpx}) are in cohomological degree zero. The differential in the above complexes is (essentially) defined by a single dot. More precisely, consider a filling of the hook partition $(a,1^{k-a})$ with numbers $1,\hdots,a$ in the row of the hook and $a,a+1,\hdots,k$ in the column (so that the unique box in the first row and first column is filled with $a$). Then the idempotent defining $\Q_i^{(a,k-a)}$ is the product of the trivial idempotent of $\C[S_a] \subset \C[S_k]$ (embedded as permutations which fix $a+1,\hdots,k$) and the sign idempotent of $\C[S_{k-a+1}]$ (embedded as permutations which fix $1,\hdots,a-1$). Diagrammatically we have
$$
\begin{tikzpicture}[scale=.75][>=stealth]
\shadedraw[gray]  (-5,.5) rectangle (-1,1);
\draw (-3,.75) node {$(a,k-a)$};
\draw (-3,1) -- (-3,2) [<-][very thick];
\draw (-.5,.75) node {$=$};
\filldraw[gray]  (3,0) rectangle (7,.5);
\draw (5,.25) node {$(1^{k-a+1})$};
\draw (5,.5) -- (5,2) [<-][very thick];
\draw (3.5,.5) -- (3.5,1) [<-][very thick];
\draw (3.5,1.5) -- (3.5,2) [<-][very thick];
\draw[gray]  (0,1) rectangle (4,1.5);
\draw (2,1.25) node {$(a)$};
\draw (2,1.5) -- (2,2) [<-][very thick];
\draw (2,0) -- (2,1) [<-][very thick];
\end{tikzpicture},
$$
where all strands are labeled by $i$. The differential $\Q_i^{(a,1^{k-a})} \rightarrow \Q_i^{(a+1,1^{k-a-1})}$ in $\Q_i^{[1^k]}$ is defined by the diagram
$$
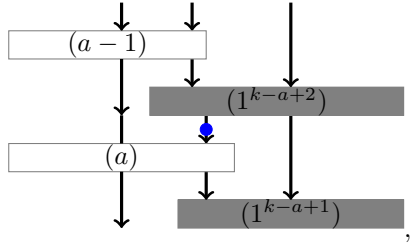
\begin{tikzpicture}[scale=.75][>=stealth]
\filldraw[gray]  (3,0) rectangle (7,.5);
\draw (5,.25) node {$(1^{k-a+1})$};
\draw (5,.5) -- (5,2) [<-][very thick];
\draw (3.5,.5) -- (3.5,1) [<-][very thick];
\draw (3.5,1.5) -- (3.5,2) [<-][very thick];
\draw[gray]  (0,1) rectangle (4,1.5);
\draw (2,1.25) node {$(a)$};
\draw (2,1.5) -- (2,2) [very thick];
\draw (2,0) -- (2,1) [<-][very thick];

\filldraw [shift={+(0,2)}] [gray]  (2.5,0) rectangle (7,.5);
\draw [shift={+(0,2)}] (4.75,.25) node {$(1^{k-a+2})$};
\draw [shift={+(0,2)}](5,.5) -- (5,2) [<-][very thick];
\draw [shift={+(0,2)}](3.25,.5) -- (3.25,1) [<-][very thick];
\draw [shift={+(0,2)}](3.25,1.5) -- (3.25,2) [<-][very thick];
\draw[shift={+(0,2)}][gray]  (0,1) rectangle (3.5,1.5);
\draw [shift={+(0,2)}](1.75,1.25) node {$(a-1)$};
\draw [shift={+(0,2)}](2,1.5) -- (2,2) [<-][very thick];
\draw [shift={+(0,2)}](2,0) -- (2,1) [<-][very thick];

\filldraw [blue] (3.5,1.75) circle (3pt);
\end{tikzpicture},
$$
Notice that there is one degree two dot on the middle downward strand. To see that this defines a differential (i.e. $d^2=0$) note that the picture for $d^2$ is 
$$
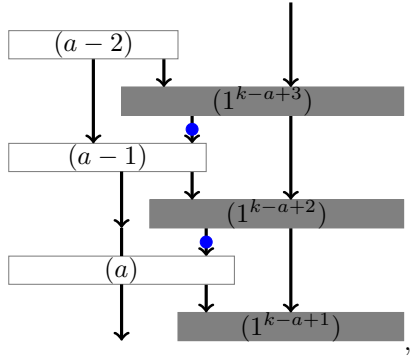
\begin{tikzpicture}[scale=.75][>=stealth]
\filldraw[gray]  (3,0) rectangle (7,.5);
\draw (5,.25) node {$(1^{k-a+1})$};
\draw (5,.5) -- (5,2) [<-][very thick];
\draw (3.5,.5) -- (3.5,1) [<-][very thick];
\draw (3.5,1.5) -- (3.5,2) [<-][very thick];
\draw[gray]  (0,1) rectangle (4,1.5);
\draw (2,1.25) node {$(a)$};
\draw (2,1.5) -- (2,2) [very thick];
\draw (2,0) -- (2,1) [<-][very thick];

\filldraw [shift={+(0,2)}] [gray]  (2.5,0) rectangle (7,.5);
\draw [shift={+(0,2)}] (4.75,.25) node {$(1^{k-a+2})$};
\draw [shift={+(0,2)}](5,.5) -- (5,2) [<-][very thick];
\draw [shift={+(0,2)}](3.25,.5) -- (3.25,1) [<-][very thick];
\draw [shift={+(0,2)}](3.25,1.5) -- (3.25,2) [<-][very thick];
\draw[shift={+(0,2)}][gray]  (0,1) rectangle (3.5,1.5);
\draw [shift={+(0,2)}](1.75,1.25) node {$(a-1)$};
\draw [shift={+(0,2)}](1.5,1.5) -- (1.5,3) [<-][very thick];
\draw [shift={+(0,2)}](2,0) -- (2,1) [<-][very thick];

\filldraw [blue] (3.5,1.75) circle (3pt);

\filldraw [shift={+(0,4)}] [gray]  (2,0) rectangle (7,.5);
\draw [shift={+(0,4)}] (4.5,.25) node {$(1^{k-a+3})$};
\draw [shift={+(0,4)}](5,.5) -- (5,2) [<-][very thick];
\draw [shift={+(0,4)}](2.75,.5) -- (2.75,1) [<-][very thick];
\draw[shift={+(0,4)}][gray]  (0,1) rectangle (3,1.5);
\draw [shift={+(0,4)}](1.5,1.25) node {$(a-2)$};

\filldraw [shift={+(0,2)}][blue] (3.25,1.75) circle (3pt);

\end{tikzpicture},
$$
Expanding the middle idempotents $(1^{k-a+2})$ and $(a-1)$ explicitly as a sum of permutations we write the above as a sum of diagrams where there are three (at least) three strands connecting the top $(1^{k-a+3})$ box and the bottom $(a)$ box. Each of these three strands contains $0,1$ or $2$ dots but it is easy to check that for any such configuration one gets zero because of the idempotents $(1^{k-a+2})$ and $(a-1)$. Therefore all the expansion terms are zero and it follows that $d^2=0$.

Finally, we would like to check that $d \ne 0$. To see this consider the closure of the map with a dot on each strand except for the strand already containing a dot (meaning that if we erased the idempotents in the differential we would have a collection of counterclockwise circles each of which has a dot). Expanding all the idempotents and using that a dot squares to zero one sees that all diagrams vanish except for the ones containing $k$ counterclockwise circles, each with one dot. Since each such circle is the identity, it follows this map is nonzero and hence that $d \ne 0$. 

It turns out that these complexes (and their images under $\Psi$) lift the following elements in $\h$
\begin{eqnarray*}
Q_i^{[1^n]} := \sum_{m=0}^n (-q)^m [m] Q_i^{(1^{n-m})} Q_i^{(m)} &\text{ and }&
Q_i^{[n]} := \sum_{m=0}^n (-q)^m [m] Q_i^{(n-m)} Q_i^{(1^m)} \\
P_i^{[1^n]} := \sum_{m=0}^n (-q)^{-m} [m] P_i^{(1^{n-m})} P_i^{(m)} &\text{ and }&
P_i^{[n]} := \sum_{m=0}^n (-q)^{-m} [m] P_i^{(n-m)} P_i^{(1^m)} 
\end{eqnarray*}
These elements show up naturally in the definition of a quantum affine algebra in section \ref{sec:quantumaffine}. Note that $P_i^{[1]} = -q^{-1} P_i$ and $Q_i^{[1]} = -q^{-1} Q_i$.

\begin{lemma}\label{lem:decatPQ} 
At the level of Grothendieck groups we have $[\P_i^{[1^n]} \1_\l] = P_i^{[1_n]} 1_\l$ and $[\Q_i^{[1^n]} \1_\l] = Q_i^{[1^n]} 1_\l$ and likewise if we replace $[1^n]$ by $[n]$. 
\end{lemma}
\begin{proof}
We prove that $[\P_i^{[1^n]} \1_\l] = P_i^{[1_n]} 1_\l$ (the other equalities follow similarly). First recall that $P_i^{(1^{n-m})} P_i^{(m)} = P_i^{(m,1^{n-m})} + P_i^{(m+1,1^{n-m-1})}$. So we get 
\begin{eqnarray*}
P_i^{[1^n]} 
&=& \sum_{m=1}^n (-q)^{-m} [m] \left( P_i^{(m,1^{n-m})} + P_i^{(m+1,1^{n-m-1})} \right) \\ 
&=& \sum_{m=1}^n P_i^{(m,1^{n-m})} \left( (-q)^{-m+1}[m-1] + (-q)^{-m}[m] \right) \\
&=& \sum_{m=1}^n (-1)^m q^{-2m+1} P_i^{(m,1^{n-m})} \\
&=& -q^{-1} \sum_{m=1}^n (-1)^{m-1} q^{-2(m-1)} P_i^{(m,1^{n-m})} \\
&=& [\P_i^{[1^n]}]
\end{eqnarray*}
where the last equality follows from the definition of $\P_i^{[1^n]}$. 
\end{proof}


\section{Quantum affine algebras and $(\hg,\theta)$ actions}\label{sec:affinelies}
\subsection{The idempotent vertex presentation}\label{sec:quantumaffine}

The quantum Heisenberg algebra $\h$ plays an important role in infinite dimensional representation theory in part because it occurs as a subalgebra of an associated quantum affine algebra. In this section we will define an idempotent version of this quantum affine Lie algebra, which we denote by $\dU_q(\hg)$. In general this definition involves a parameter $c$ known as the level. Since we are primarily concerned with the basic representation, which is level one, we have set $c=1$ in the definitions below.  For the general definition involving an arbitrary level, see \cite{CLi2}. 

We define the level one quantum affine algebra $\dU_q(\hg)$ to be the $\k(q)$ algebra generated by
$$E_{i,r} 1_\l, 1_\l F_{i,r}, P_i^{(n)} 1_\l, 1_\l Q_i^{(n)}, P_i^{(1^n)} 1_\l, 1_\l Q_i^{(1^n)} \text{ where } i \in \I, r \in \Z \text{ and } n \in \N.$$
The relations in $\dU_q(\hg)$ are as follows.

\begin{enumerate}
\item  $\{1_\l: \l \in \hX \}$ are mutually orthogonal idempotents, moreover 
\begin{eqnarray*}
& & E_{i,r} 1_\l = 1_\mu E_{i,r} 1_\l = 1_\mu E_{i,r} \\
& & F_{i,-r} 1_\mu = 1_\l F_{i,-r} 1_\mu = 1_\l F_{i,-r} 
\end{eqnarray*}
where $\mu = \l+ \alpha_i + r \delta$ 

\item We have 
\begin{eqnarray*}
P_i^{(n)} 1_\l = 1_\mu P_i^{(n)} 1_\l = 1_\mu P_i^{(n)} &\text{ and }& P_i^{(1^n)} 1_\l = 1_\mu P_i^{(1^n)} 1_\l = 1_\mu P_i^{(1^n)} \\
Q_i^{(n)} 1_\mu = 1_\l Q_i^{(n)} 1_\mu = 1_\l Q_i^{(n)} &\text{ and }& Q_i^{(1^n)} 1_\mu = 1_\l Q_i^{(1^n)} 1_\mu = 1_\l Q_i^{(1^n)} 
\end{eqnarray*}
where $\mu = \l + n \delta$. 

\item The subalgebra generated by $P$s and $Q$s is isomorphic to the quantum Heisenberg algebra $\h$. 

\item We have 
\begin{eqnarray*}
\label{eq:q'E1}
[Q_i^{[1^{a+1}]}, E_{i,b}] 1_\l &=&
\begin{cases}
q^2 Q_i^{[1^a]} E_{i,b+1} 1_\l - q^{-2} E_{i,b+1} Q_i^{[1^a]} 1_\l \text{ if } a > 0 \\
[2]E_{i,b+1} 1_\l \text{ if } a = 0.
\end{cases} \\
\label{eq:q'E2}
q^{-1} [Q_i^{[1^{a+1}]}, F_{i,b}] 1_\l &=&
\begin{cases}
q^{-2} Q_i^{[1^a]} F_{i,b+1} 1_\l - q^2 F_{i,b+1} Q_i^{[1^a]}  1_\l \text{ if } a > 0 \\
- [2]F_{i,b+1}1_\l \text{ if } a = 0.
\end{cases} \\
\label{eq:q'E3}
q [P_i^{[1^{a+1}]}, E_{i,b+1}]1_\l &=&
\begin{cases}
q^2 E_{i,b} P_i^{[1^{a}]} 1_\l - q^{-2} P_i^{[1^{a}]} E_{i,b}1_\l \text{ if } a > 0 \\
[2] E_{i,b}1_\l \text{ if } a = 0
\end{cases} \\
\label{eq:q'E4}
[P_i^{[1^{a+1}]}, F_{i,b+1}]1_\l &=&
\begin{cases}
q^{-2} F_{i,b} P_i^{[1^{a}]} 1_\l - q^2 P_i^{[1^{a}]} F_{i,b}1_\l \text{ if } a > 0 \\
- [2] F_{i,b}1_\l \text{ if } a = 0.
\end{cases}
\end{eqnarray*}
while if $\la i, j \ra = -1$ we have
\begin{eqnarray*}
\label{eq:q'E5}
[Q_j^{[1^{a+1}]}, E_{i,b}] 1_\l &=&
\begin{cases}
qE_{i,b+1} Q_j^{[1^a]}1_\l - q^{-1} Q_j^{[1^a]} E_{i,b+1}1_\l \text{ if } a > 0 \\
E_{i,b+1}1_\l \text{ if } a = 0.
\end{cases} \\
\label{eq:q'E6}
q^{-1} [Q_j^{[1^{a+1}]}, F_{i,b}]1_\l &=&
\begin{cases}
q^{-1} F_{i,b+1} Q_j^{[1^a]}1_\l - q Q_j^{[1^a]} F_{i,b+1}1_\l \text{ if } a > 0 \\
- F_{i,b+1}1_\l \text{ if } a = 0
\end{cases} \\
\label{eq:q'E7}
q [P_j^{[1^{a+1}]}, E_{i,b+1}]1_\l &=&
\begin{cases}
q^{-1}E_{i,b} P_j^{[1^{a}]}1_\l - q P_j^{[1^{a}]} E_{i,b}1_\l \text{ if } a > 0 \\
E_{i,b}1_\l \text{ if } a = 0
\end{cases} \\
\label{eq:q'E8}
[P_j^{[1^{a+1}]}, F_{i,b+1}]1_\l &=&
\begin{cases}
qF_{i,b} P_j^{[1^{a}]}1_\l - q^{-1} P_j^{[1^{a}]} F_{i,b}1_\l \text{ if } a > 0 \\
- F_{i,b}1_\l \text{ if } a = 0.
\end{cases}
\end{eqnarray*}
If $\la i,j \ra = 0$ then $P_j^{[1^a]} 1_\l$ and $Q_j^{[1^{a}]} 1_\l$ commute with both $E_{i,b} 1_\l$ and $F_{i,b} 1_\l$. 

\item We have
$$[E_{i,a}, F_{i,b}] 1_\l =
\begin{cases}
q^{-b} q^{\la \l, i \ra} Q_i^{[1^{a+b}]} 1_\l \text{ if } a+b > 0 \\
q^{-a} q^{- \la \l, i \ra} P_i^{[1^{-a-b}]} 1_\l \text{ if } a+b < 0 \\
[\la \l,i \ra + a] 1_\l \text{ if } a+b=0.

\end{cases}$$
while if $i \ne j$ then $[E_{i,a}, F_{j,b}] 1_\l = 0$.

\item For any $m,n \in \Z$ we have
\begin{eqnarray*}
E_{i,m}E_{i,n-1} 1_\l + E_{i,n}E_{i,m-1} 1_\l &=& q^2 \left( E_{i,m-1}E_{i,n} 1_\l + E_{i,n-1}E_{i,m} 1_\l \right) \\
F_{i,n-1}F_{i,m} 1_\l + F_{i,m-1}F_{i,n} 1_\l &=& q^2 \left( F_{i,n} F_{i,m-1} 1_\l + F_{i,m} F_{i,n-1} 1_\l \right).
\end{eqnarray*}

\item For any $m,n \in \Z$, if $\la i, j \ra = -1$ we have
\begin{eqnarray*}
E_{i,m}E_{j,n+1} 1_\l - q E_{j,n+1} E_{i,m} 1_\l &=& E_{j,n} E_{i,m+1} 1_\l - q E_{i,m+1} E_{j,n} 1_\l \\
F_{i,m+1}F_{j,n} 1_\l - q F_{j,n} F_{i,m+1} 1_\l &=& F_{j,n+1} F_{i,m} 1_\l - q F_{i,m} F_{j,n+1} 1_\l 
\end{eqnarray*}
while if $\la i, j \ra = 0$ then
$$E_{i,m}E_{j,n} 1_\l = E_{j,n}E_{i,m} 1_\l \text{ and } F_{i,m}F_{j,n} 1_\l = F_{j,n}F_{i,m} 1_\l.$$

\item If $\la i, j \ra = -1$ then
$$\sum_{\sigma \in S_2} \left( E_{j,n}E_{i,m_{\sigma(1)}}E_{i,m_{\sigma(2)}} 1_\l +  E_{i, m_{\sigma(1)}}E_{i,m_{\sigma(2)}}E_{j,n} 1_\l \right) = \sum_{\sigma \in S_2} [2]E_{i,m_{\sigma(1)}}E_{j,n}E_{i,m_{\sigma(2)}} 1_\l$$
and similarly if we replace all $E$s by $F$s.

\end{enumerate}

This definition is an idempotent modified version of Drinfeld's new realization of the quantum affine algebra. The modification is not entirely trivial. For example, we no longer have the generators $q^h$, $q^{\pm d}$ or $q^{\pm c/2}$ which are standard in the Drinfeld presentation. Our presentation also differs from Drinfeld's by a renormalization which includes certain sign changes. This presentation is described in the last section of \cite{CLi2} where we show that it is equivalent to Drinfeld's new realization.

The fact that $\dU_q(\hg)$ is idempotent means that we can think of it as a 1-category. The objects are labeled by the weights of $\hg$ while elements like $E_i 1_\l$ are 1-morphisms between $\l$ and $\l+\alpha_i$. 

\subsubsection{The basic representation of $U_q(\hg)$}\label{sec:vbasic}

The basic representation $V_{\Lambda_0}$ of the affine Lie algebra $\hg$ is characterized by the fact that it is irreducible and that there exists a vector $v \in V_{\Lambda_0}$ such that
$$ (\g\otimes \C[t])\cdot v = 0 \text{ and } c\cdot v = v. $$
It follows from this that the central element $c$ acts by the identity on all of $V_{\Lambda_0}$, so that $V_{\Lambda_0}$ is a level one irreducible representation. This representation deforms to give the basic representation of the quantum affine algebra $\dU_q(\hg)$.

An explicit construction of $V_{\Lambda_0}$ was given by Frenkel-Kac \cite{FK} and Segal \cite{S} in the case of the affine Lie algebra ($q=1$). The construction, which is known as the homogeneous realization of $V_{\Lambda_0}$, was extended to the quantum affine algebra in \cite{FJ}. It begins by restricting $V_{\Lambda_0}$ to the homogeneous Heisenberg subalgebra $\h \subset \g$. Each vector $u \in V_{\Lambda_0}$ such that $\h^+\cdot u = 0$ generates a copy of the Fock space $\mathcal{F}$. The space of such vectors has a basis given by $\widehat{ W}\cdot v$, namely the affine Weyl group orbit of the highest weight vector $v \in V_{\Lambda_0}(\Lambda_0)$. This orbit can be identified with the root lattice $Y$ of $\g$. Thus, as an $\h$ module, $V_{\Lambda_0}$ decomposes as
$$ {V_{\Lambda_0}}_{|_{\h}} \cong \mathcal{F}\otimes_{\k(q)} \k(q)[Y].$$
For $\alpha = w \cdot \Lambda_0 \in X$, the summand $\mathcal{F}(n)\otimes \alpha$ is the weight space in $V_{\Lambda_0}$ of weight $w\cdot \Lambda_0 - n \delta$.  This bijection between weight spaces of the basic representation and pairs $(\alpha,n) \in Y\times \N$ will be used in Section \ref{sec:vertexops}. The Frenkel-Kac-Segal construction, as described in \cite{FJ}, then uses vertex operators to extend the action of $\h$ on $V_{\Lambda_0}$ to the entire quantum affine algebra.

\subsection{A stripped down presentation}\label{sec:simple}

We now define a simplified version of the quantum affine algebra $\dU_q(\hg)$. The advantage of this definition is that it is much simpler to check in practice. Moreover, one can show that in the case of an integrable representation it is equivalent to an action of the quantum affine algebra in its Kac-Moody presentation. 

This presentation is generated by $E_{i,m} \1_\l, \1_\l F_{i,-m}$ where $i \in I$ and $m \in \{0,1\}$ together with the the following relations. 
\begin{enumerate}
\item  $\{1_\l: \l \in \hX \}$ are mutually orthogonal idempotents, moreover 
\begin{eqnarray*}
& & E_{i,m} 1_\l = 1_\mu E_{i,m} 1_\l = 1_\mu E_{i,m} \\
& & F_{i,-m} 1_\mu = 1_\l F_{i,-m} 1_\mu = 1_\l F_{i,-m} 
\end{eqnarray*}
where $\mu = \l+ \alpha_i + m \delta$. 

\item We have $[E_{i,m}, F_{i,-m}] 1_\l = [\l_i + m] 1_\l$ while if $i \ne j$ then $[E_{i,m}, F_{j,-n}] 1_\l = 0$.

\item We have $E_{i,1} \E_i 1_\l = q^2 E_i E_{i,1} 1_\l$ and $F_{i,-1} F_i 1_\l = q^2 F_i F_{i,-1} 1_\l$.

\item If $\la i, j \ra = -1$ we have
\begin{eqnarray*}
E_{i}E_{j,1} 1_\l - q E_{j,1} E_{i} 1_\l &=& E_{j} E_{i,1} 1_\l - q E_{i,1} E_{j} 1_\l \\
F_{i}F_{j,-1} 1_\l - q F_{j,-1} F_{i} 1_\l &=& F_{j} F_{i,-1} 1_\l - q F_{i,-1} F_{j} 1_\l 
\end{eqnarray*}
while if $\la i, j \ra = 0$ then
$$E_{i,m}E_{j,n} 1_\l = E_{j,n}E_{i,m} 1_\l \text{ and } F_{i,-m}F_{j,-n} 1_\l = F_{j,-n}F_{i,-m} 1_\l.$$
\end{enumerate}

Here are some somewhat surprising things to notice about this definition.
\begin{itemize}
\item There is no Heisenberg algebra explicitly present. However, one can recover $P_i$ as the commutator of $E_i$ and $F_{i,-1}$ (and similarly for $Q_i$). 
\item Although we only have $E_{i,m}$ and $F_{i,-m}$ explicitly defined when $m=0,1$, the missing generators can be obtained from these. 
\item There is no Serre relation because it is a formal consequence in an integrable representation. 
\end{itemize}

\subsection{$(\hg,\theta)$ actions}\label{sec:g-action}

For a Kac-Moody Lie algebra $\g$ there exist several notions of a categorical action of $\g$ \cite{KL3,R,CK,Cau}. Although the definitions in these papers are somewhat different, they are all closely related. More precisely, we will use the concept of a $(\g,\theta)$ action from \cite{Cau}. This action imitates the simplified quantum affine algebra action from section \ref{sec:simple} and hence has the advantage that it is much easier to check. Nevertheless, it still carries all the rich structure such as an action of the KLR algebras (as discussed in section \ref{sec:cataction} of the introduction). 

However, in this paper $\hg$ and its $q$-deformation $\dU_q(\hg)$ appear not in their Kac-Moody presentation, but rather in their loop presentation. Subsequently we extend in a simple way the definition from \cite{Cau} to a $(\hg,\theta)$ action where $\hg$ is a quantum affine algebra in its loop presentation. 

Since the basic representation has level one, we consider here only the notion of a level one action of $\hg$. Higher level representations are also natural and will be discussed in future work.

A level one $(\hg,\theta)$ action consists of a graded, triangulated, $\k$-linear idempotent complete 2-category $\K$ where:
\begin{itemize}
\item 0-morphisms (objects) are denoted $\D(\l)$ and are indexed by weights $\l \in \hX$,
\item 1-morphisms include $\E_{i,m} \1_\l = \1_{\l+\alpha_i-m\delta} \E_{i,m}$ and $\F_{i,-m} \1_{\l+\alpha_i-m\delta} = \1_\l \F_{i,-m}$ where $m=0,1$ and $\1_\l$ is the identity 1-morphism of $\l$.
\item 2-morphisms include, for each $\l \in \hX$, a linear map $\hY_\k \rightarrow \End^2(\1_\l)$.
\end{itemize}
\begin{Remark} 
We will abuse notation and denote by $\theta \in \End^2(\1_\l)$ the image of $\theta \in \hY_\k$ under the linear map above.
\end{Remark}

On this data we impose the following conditions.
\begin{enumerate}
\item \label{co:hom1} $\Hom_{\K}(\1_\l, \1_\l \la l \ra)$ is zero if $l < 0$ and one-dimensional if $l=0$ and $\1_\l \ne 0$. Moreover, the space of maps between any two 1-morphisms is finite dimensional.
\item $\E_i$ and $\F_i$ are left and right adjoints of each other up to specified shifts. More precisely
\begin{enumerate}
\item $(\E_{i,m} \1_\l)_R \cong \1_\l \F_{i,-m} \la \l_i+1+m \ra$, and
\item $(\E_{i,m} \1_\l)_L \cong \1_\l \F_{i,-m} \la -\l_i-1-m \ra$.
\end{enumerate}
\item \label{co:EF} We have
\begin{align*}
\F_{i,-m} \E_{i,m} \1_{\l} &\cong \E_{i,m} \F_{i,-m} \1_{\l} \oplus \1_\l \otimes_\k V_{- \l_i - m - 1}  \text{ if } \l_i + m \le 0, \\
\E_{i,m} \F_{i,-m} \1_{\l} &\cong \F_{i,-m} \E_{i,m} \1_{\l} \oplus \1_\l \otimes_\k V_{\l_i + m - 1} \text{ if } \l_i + m \ge 0.
\end{align*}

\item \label{co:EiFj} If $i \ne j$ then $\E_{i,m}$ and $\F_{j,-n}$ commute for $m,n \in \{0,1\}$. 

\item \label{co:theta} If $\l_i+m \ge 0$ then map $(I \theta I) \in \End^2(\E_{i,m} \1_\l \F_{i,-m})$ induces an isomorphism between $\l_i+m+1$ (resp. zero) of the $\l_i+m+2$ summands $\1_{\l+\alpha_i-m\delta}$ when $\la \theta, \alpha_i \ra \ne 0$ (resp. $\la \theta, \alpha_i \ra = 0$). If $\l_i+m \le 0$ then the analogous result holds for $(I \theta I) \in \End^2(\F_{i,-m} \1_\l \E_{i,m})$.

\item For $i \in \I$ we have
\begin{eqnarray*}
\E_{i} \E_{i,1} \1_\l \cong \E_{i,1} \E_{i} \la -2 \ra \1_\l \\
\F_{i,-1} \F_{i} \1_\l \cong \F_{i} \F_{i,-1} \la 2 \ra \1_\l.
\end{eqnarray*}

\item \label{co:triangle}
From the relations above it follows that if $\la i,j \ra = -1$ then
$$\Hom(\E_{i} \E_{j,1} \1_\l, \E_{j,1} \E_{i} \la 1 \ra \1_\l) \text{ and } 
\Hom(\E_{j} \E_{i,1} \1_\l, \E_{i,1} \E_{j} \la 1 \ra \1_\l)$$
are both one-dimensional. If $\alpha$ and $\beta$ are 2-morphisms which span these spaces, then we require the relation $\Cone(\alpha) \cong \Cone(\beta)$. 

\item \label{co:vanish1} If $\alpha = \alpha_i$ or $\alpha = \alpha_i + \alpha_j$ for some $i,j \in I$ with $\la i,j \ra = -1$ then $\1_{\l+r \alpha} = 0$ for $r \gg 0$ or $r \ll 0$.

\item \label{co:vanish2} If $\delta = \alpha_i + \alpha_j + \alpha_k$ with $i,j,k \in I$ a triangle in the Dynking diagram or $\delta = \alpha_i + \alpha_j + \alpha_k + \alpha_\l$ with $i,j,k,l \in I$ forming a square then $\1_{\l+r \delta} = 0$ for $r \gg 0$ and $\la \l, \delta \ra > 0$ if $\1_\l \ne 0$.

\item \label{co:new} Suppose $i \ne j \in I$ and $\l \in \hX$. If $\1_{\l+\alpha_i}$ and $\1_{\l+\alpha_j}$ are nonzero then $\1_{\l}$ and $\1_{\l+\alpha_i+\alpha_j}$ are also nonzero.
\end{enumerate}

\begin{Remark}
One of the new features in the definition above is the appearance of distinguished triangles in condition (\ref{co:triangle}).  
\end{Remark}

\subsection{Toroidal modifications}\label{sec:toroidal}

All definitions above can be extended from the affine case to the toroidal case as follows. First recall that the Heisenberg 2-category $\H^\G$ from \cite{CLi1} has 1-morphisms $\P_i$ and $\Q_i$ for $i \in \hI$ a node in the affine Dynkin diagram (not just the finite one); thus 2-morphisms in the 2category $\H^\G$ also include generators $\P_0$ and $\Q_0$, where $0 \in \hI$ is the affine node. The 2-category $\H^\G$ categorifies the quantum toroidal Heisenberg algebra, which is a subalgebra of the idempotent quantum toroidal algebra $\dU_q(\widehat{\hg})$.  

The definition of a 2-representation of $\widehat{\h}$ and $\widehat{\hg}$ is then essentially the same as for $\h$ and $\hg$, taking into account the existence of new 1-morphisms $\P_0$, $\Q_0$, $\E_{0,b}$, $\F_{0,b}$, etc. The categorified vertex operators of this paper then produce the basic 2-representation of $\widehat{\hg}$ from the Fock space 2-representations of $\widehat{\h}$.  

In particular, the structure of the basic representation of $\dU_q(\widehat{\hg})$ is very similar to that of the basic representation of $\dU_q(\hg)$. For example, the underlying vector space for the basic toroidal representation is
$$ V_{\Lambda_0} := \mathcal{\widehat{F}}\otimes_{\k(q)} \k(q)[\widehat{Y}]$$
where $\widehat{\sF}$ is the Fock space representation of the quantum toroidal Heisenberg algebra and $\hY$ is the affine root lattice, \cite{FJW}. In Section \ref{sec:applications} we will give algebraic and geometric categorifications of the basic representation of $\dU_q(\widehat{\hg})$. 

All the proofs to follow extend without extra work from the affine case to the toroidal case, with one exception. The definition of a 2-representation of $\widehat{\widehat{\sl_2}}$ should be changed slightly, essentially because the Dynkin diagram of $\widehat{\sl_2}$ is not simply-laced. The appropriate definition of Heisenberg category for this case is given in the appendix of \cite{CLi1}, and after using that definition, the results in the current paper then carry over.


\section{Categorical vertex operators and main results}
\subsection{Categorified vertex operators}\label{sec:vertexops}

Let $\K$ denote the underlying 2-category of a 2-representation of $\h$.
For any $i \in \I$ we define the following complexes in the homotopy category $\Kom(\K)$. We let 
\begin{eqnarray}
\label{eq:cpxC1}
\fC^-_i(k) &:=& \left[ \dots \rightarrow \P_i^{(-k+l)} \Q_i^{(1^l)} \la -l \ra \rightarrow \dots \rightarrow \P_i^{(-k+1)} \Q_i \la -1 \ra \rightarrow \P_i^{(-k)} \right] \la k \ra [-k]  \\
\label{eq:cpxC2}
\fC^-_i(k) &:=& 
\left[ \dots \rightarrow \P_i^{(l)} \Q_i^{(1^{k+l})} \la -l \ra \rightarrow \dots \rightarrow \P_i \Q_i^{(1^{k+1})} \la -1 \ra \rightarrow \Q_i^{(1^{k})} \right] 
\end{eqnarray}
depending on whether $k \le 0$ or $k \ge 0$ respectively. Here the right most term is in cohomological degree zero and the minus signs indicate that the complex is unbounded below. Likewise, we let
\begin{eqnarray}
\label{eq:cpxC3}
\fC^+_i(k) &:=&
\left[ \Q_i^{(-k)} \rightarrow \P_i \Q_i^{(-k+1)} \la 1 \ra \rightarrow \dots \rightarrow \P_i^{(1^l)} \Q_i^{(-k+l)} \la l \ra \rightarrow \dots \right] \la -k \ra [k] \\
\label{eq:cpxC4}
\fC^+_i(k) &:=& 
\left[ \P_i^{(1^{k})} \rightarrow \P_i^{(1^{k+1})} \Q_i \la 1 \ra \rightarrow \dots \rightarrow \P_i^{(1^{k+l})} \Q_i^{(l)} \la l \ra \rightarrow \dots \right] 
\end{eqnarray}
depending on whether $k \le 0$ or $k \ge 0$. Here the left most term is in cohomological degree zero and the plus signs indicate that the complex is unbounded above. The differentials in these complexes are either given by a cap or a cup. For example, the differentials
\begin{eqnarray*}
\P_i^{(l)} \Q_i^{(1^{k+l})} \la -l \ra &\longrightarrow& \P_i^{(l-1)} \Q_i^{(1^{k+l-1})} \la -l+1 \ra  \text{ and } \\
\P_i^{(1^{k+l})} \Q_i^{(l)} \la l \ra &\longrightarrow& \P_i^{(1^{k+l+1})} \Q_i^{(l+1)} \la l+1 \ra
\end{eqnarray*}
in the definition of $\fC_i^-(k)$ and $\fC^+_i(k)$ when $k \ge 0$ are given by 
$$
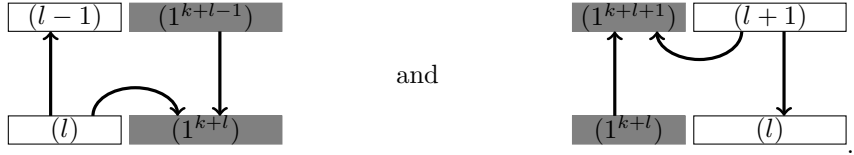
\begin{tikzpicture}[scale=.75][>=stealth]

\draw (2.25,1) rectangle (4.25,1.5);
\draw (3.25,1.25) node {$(l)$};
\filldraw[gray] (4.4,1) rectangle (7.1,1.5);
\draw (5.75,1.25) node {$(1^{k+l})$};

\draw [shift = {+(0,2)}](2.25,1) rectangle (4.25,1.5);
\draw [shift = {+(0,2)}](3.25,1.25) node {$(l-1)$};
\filldraw[gray] [shift = {+(0,2)}](4.4,1) rectangle (7.1,1.5);
\draw [shift = {+(0,2)}](5.75,1.25) node {$(1^{k+l-1})$};

\draw (3,1.5) -- (3,3) [->][very thick];
\draw (6,1.5) -- (6,3) [<-][very thick];
\draw (5.25,1.5) arc (0:180:.75cm and .5cm)[<-][very thick];

\draw (9.5,2.25) node {and};

\filldraw[gray] [shift = {+(10,0)}](2.25,1) rectangle (4.25,1.5);
\draw [shift = {+(10,0)}](3.25,1.25) node {$(1^{k+l})$};
\draw[shift = {+(10,0)}] (4.4,1) rectangle (7.1,1.5);
\draw [shift = {+(10,0)}](5.75,1.25) node {$(l)$};

\filldraw[gray][shift = {+(10,0)}] [shift = {+(0,2)}](2.25,1) rectangle (4.25,1.5);
\draw [shift = {+(10,0)}][shift = {+(0,2)}](3.25,1.25) node {$(1^{k+l+1})$};
\draw [shift = {+(10,0)}][shift = {+(0,2)}](4.4,1) rectangle (7.1,1.5);
\draw [shift = {+(10,0)}][shift = {+(0,2)}](5.75,1.25) node {$(l+1)$};

\draw [shift = {+(10,0)}](3,1.5) -- (3,3) [->][very thick];
\draw [shift = {+(10,0)}](6,1.5) -- (6,3) [<-][very thick];
\draw [shift = {+(10,0)}](3.75,3) arc (180:360:.75cm and .5cm)[<-][very thick];
\end{tikzpicture}.
$$
\begin{lemma} The maps defined above define a differential (i.e. they square to zero). 
\end{lemma}
\begin{proof}
We consider $\fC_i^-$. Applying the differential twice gives the following
$$
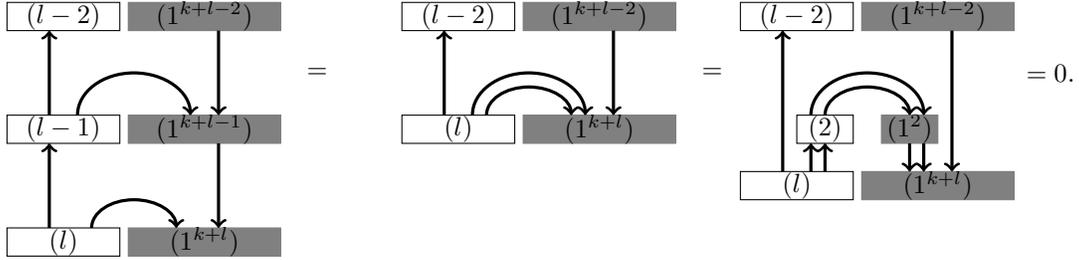
\begin{tikzpicture}[scale=.75][>=stealth]

\draw  [shift = {+(-7,0)}](2.25,1) rectangle (4.25,1.5);
\draw  [shift = {+(-7,0)}](3.25,1.25) node {$(l-1)$};
\filldraw[gray] [shift = {+(-7,0)}] (4.4,1) rectangle (7.1,1.5);
\draw [shift = {+(-7,0)}] (5.75,1.25) node {$(1^{k+l-1})$};

\draw  [shift = {+(-7,0)}][shift = {+(0,2)}](2.25,1) rectangle (4.25,1.5);
\draw  [shift = {+(-7,0)}][shift = {+(0,2)}](3.25,1.25) node {$(l-2)$};
\filldraw[gray]  [shift = {+(-7,0)}][shift = {+(0,2)}](4.4,1) rectangle (7.1,1.5);
\draw  [shift = {+(-7,0)}][shift = {+(0,2)}](5.75,1.25) node {$(1^{k+l-2})$};

\draw  [shift = {+(-7,0)}](3,1.5) -- (3,3) [->][very thick];
\draw  [shift = {+(-7,0)}](6,1.5) -- (6,3) [<-][very thick];
\draw  [shift = {+(-7,0)}](5.5,1.5) arc (0:180:1cm and .75cm)[<-][very thick];

\draw  [shift = {+(-7,0)}](7.75,2.25) node {$=$};

\draw  [shift = {+(-7,-2)}](2.25,1) rectangle (4.25,1.5);
\draw  [shift = {+(-7,-2)}](3.25,1.25) node {$(l)$};
\filldraw[gray] [shift = {+(-7,-2)}] (4.4,1) rectangle (7.1,1.5);
\draw [shift = {+(-7,-2)}] (5.75,1.25) node {$(1^{k+l})$};

\draw  [shift = {+(-7,-2)}](3,1.5) -- (3,3) [->][very thick];
\draw  [shift = {+(-7,-2)}](6,1.5) -- (6,3) [<-][very thick];
\draw  [shift = {+(-7,-2)}](5.25,1.5) arc (0:180:.75cm and .5cm)[<-][very thick];

\draw (2.25,1) rectangle (4.25,1.5);
\draw (3.25,1.25) node {$(l)$};
\filldraw[gray] (4.4,1) rectangle (7.1,1.5);
\draw (5.75,1.25) node {$(1^{k+l})$};

\draw [shift = {+(0,2)}](2.25,1) rectangle (4.25,1.5);
\draw [shift = {+(0,2)}](3.25,1.25) node {$(l-2)$};
\filldraw[gray] [shift = {+(0,2)}](4.4,1) rectangle (7.1,1.5);
\draw [shift = {+(0,2)}](5.75,1.25) node {$(1^{k+l-2})$};

\draw (3,1.5) -- (3,3) [->][very thick];
\draw (6,1.5) -- (6,3) [<-][very thick];
\draw (5.25,1.5) arc (0:180:.75cm and .5cm)[<-][very thick];
\draw (5.5,1.5) arc (0:180:1cm and .75cm)[<-][very thick];

\draw (7.75,2.25) node {$=$};

\draw [shift = {+(6,0)}](3.25,1) rectangle (4.25,1.5);
\draw  [shift = {+(6,0)}](3.75,1.25) node {$(2)$};
\filldraw[gray]  [shift = {+(6,0)}](4.75,1) rectangle (5.75,1.5);
\draw  [shift = {+(6,0)}](5.25,1.25) node {$(1^{2})$};

\draw  [shift = {+(6,0)}][shift = {+(0,2)}](2.25,1) rectangle (4.25,1.5);
\draw  [shift = {+(6,0)}][shift = {+(0,2)}](3.25,1.25) node {$(l-2)$};
\filldraw[gray]  [shift = {+(6,0)}][shift = {+(0,2)}](4.4,1) rectangle (7.1,1.5);
\draw  [shift = {+(6,0)}][shift = {+(0,2)}](5.75,1.25) node {$(1^{k+l-2})$};

\draw  [shift = {+(6,0)}](3,.5) -- (3,3) [->][very thick];
\draw  [shift = {+(6,0)}](6,.5) -- (6,3) [<-][very thick];
\draw  [shift = {+(6,0)}](5.25,1.5) arc (0:180:.75cm and .5cm)[<-][very thick];
\draw  [shift = {+(6,0)}](5.5,1.5) arc (0:180:1cm and .75cm)[<-][very thick];
\draw [shift = {+(6,-1)}](2.25,1) rectangle (4.25,1.5);
\draw [shift = {+(6,-1)}](3.25,1.25) node {$(l)$};
\filldraw[gray] [shift = {+(6,-1)}](4.4,1) rectangle (7.1,1.5);
\draw [shift = {+(6,-1)}](5.75,1.25) node {$(1^{k+l})$};
\draw  [shift = {+(6,-1)}](5.25,1.5) -- (5.25,2) [<-][very thick];
\draw  [shift = {+(6,-1)}](5.5,1.5) -- (5.5,2) [<-][very thick];
\draw  [shift = {+(6,-1)}](3.5,1.5) -- (3.5,2) [->][very thick];
\draw  [shift = {+(6,-1)}](3.75,1.5) -- (3.75,2) [->][very thick];

\draw[shift = {+(6,0)}] (7.75,2.25) node {$=0.$};
\end{tikzpicture}
$$
The first and second equalities above follows from the fact for $s \geq 0$, if we denote by $c_{(s)}$ and $c_{(1^s)}$ the idempotents in $\k[S_s]$ corresponding to the trivial and sign representation, then after embedding $\k[S_r]$ into $\k[S_{r+s}]$ in the natural way we have
$$c_{(r)}c_{(r+s)} = c_{(r+s)} = c_{(r+s)}c_{(r)} \text{ and }
c_{(1^r)}c_{(1^{r+s})} = c_{(1^{r+s})} = c_{(1^{r+s})}c_{(1^r)}.$$
Now last equality in the proof of the Lemma follows from the fact that $c_{(2)} c_{(1^2)} = 0 \in \k[S_2]$.  Thus $\fC_i^-$ is a chain complex. The proof that $\fC_i^+$ is a chain complex is the same. 
\end{proof}

\subsection{Defining a $(\hg,\theta)$ action}\label{sec:defs}

Suppose $\K$ is an integrable 2-representation of $\h$. Since the object labeled $n$ is zero for $n \ll 0$ we can relabel the objects of $\K$ so that $\1_n = 0$ if $n < 0$. We will assume this from now on. 

We now define a 2-representation of $\hg$ by describing the objects, 1-morphisms and 2-morphisms. If $\l \in \hX$ does not occur in the basic representation $V_{\Lambda_0}$ of $\hg$ then we define $\D(\l):=0$. On the other hand, any nonzero weight space of $V_{\Lambda_0}$ corresponds to a weight $\l \in \hX$ of the form $\l = w \cdot \Lambda_0 - n \delta$ where $w \in W_{\hg}$ and $n \in \N$ (this $w$ and $n$ are uniquely determined by $\l$). In this case we define $\D(\l) := \D(n)$. Since the $W_{\hg}$ orbit of $\Lambda_0$ is in bijection with the root lattice $Y$ of $\g$, it follows that nonzero objects of are in bijection with pairs $(\alpha, n)$ where $\alpha \in Y$ and $n \in \N$. 
 
Next we define 1-morphisms $\E_{i,m} \1_\l$ and $\1_\l \F_{i,m}$ as 1-morphisms in $\Kom(\K)$. We set
\begin{equation}\label{eq:Emaps1}
\E_{i,m} \1_\l \mapsto \fC_i^-(\la \l, \alpha_i \ra+1+m) \1_n \ \ \text{ and } \ \ 
\1_\l \F_{i,m} \mapsto \1_n \fC_i^+(\la \l, \alpha_i \ra+1-m).
\end{equation}

To define $\hY_\k \rightarrow \End^2(\1_\l)$ consider, for each $i$, the following 2-morphism in $\K$ 
\begin{equation}\label{eq:thetamap}
\begin{tikzpicture}[scale=.75][>=stealth]
\draw [shift={+(2,1)}](-.25,0) arc (180:360:0.5cm) [very thick];
\draw [shift={+(2,1)}][<-](.75,0) arc (0:180:0.5cm) [very thick];
\draw [shift={+(2.8,1)}](1.50,0) node{$: \1_n \rightarrow \1_n \la 2 \ra$};
\draw [shift={+(1.2,1.75)}](.25,-.75) node {$i$}; 
\end{tikzpicture}
\end{equation}
We define $\hY_\k \rightarrow \End^2(\1_\l)$ by taking $\alpha_i$ to the map above, $\delta$ to zero and extending linearly. 

\subsection{The main results}\label{sec:mainthms}

The main work in this paper is proving that the functors defined above induce a $(\hg,\theta)$ action. Below we summarize all the relations that we check directly.  

\begin{Theorem}\label{thm:main1} In the homotopy 2-category $\Kom(\K)$ we have the following relations. 
\begin{itemize}
\item {[Proposition \ref{prop:adjointrels}]}
For $i \in I$ and $m \in \Z$ we have
\begin{enumerate}
\item $(\E_{i,m} \1_\l)_R \cong \1_\l \F_{i,-m} \la \l_i + 1 + m \ra$
\item $(\E_{i,m} \1_\l)_L \cong \1_\l \F_{i,-m} \la - \l_i - 1 - m \ra$
\end{enumerate}
\item {[Proposition \ref{prop:sl2comm}]} 
For $i \in \I$ we have 
\begin{align*}
\F_{i,-m} \E_{i,m} \1_{\l} &\cong \E_{i,m} \F_{i,-m} \1_{\l} \oplus \1_\l \otimes_\k V_{- \la \l, \alpha_i \ra - m -1}  \text{ if } \la \l, \alpha_i \ra + m \le 0 \\
\E_{i,m} \F_{i,-m} \1_{\l} &\cong \F_{i,-m} \E_{i,m} \1_{\l} \oplus \1_\l \otimes_\k V_{\la \l, \alpha_i \ra + m -1} \text{ if } \la \l, \alpha_i \ra + m \ge 0.
\end{align*}

\item {[Proposition \ref{prop:theta}]} Under the two isomorphisms above, for $\theta \in Y_\k$ the maps
\begin{align*}
I \theta I: \F_{i,-m} \1_\l \E_{i,m} \rightarrow \F_{i,-m} \1_\l \E_{i,m} \la 2 \ra  &\text{ if } \l_i + m \le 2 \text{ and } \\
I \theta I: \E_{i,m} \1_\l \F_{i,-m} \rightarrow  \E_{i,m} \1_\l \F_{i,-m} \la 2 \ra &\text{ if } \l_i + m \ge -2
\end{align*}
induce an isomorphism between all (resp. none of the) summands $\1_\l \la \cdot \ra$ of the same degree on either side if $\la \theta, \alpha_i \ra \ne 0$ (resp. $\la \theta, \alpha_i \ra = 0$).  

\item {[Proposition \ref{prop:EiFjcomm}]} If $i \ne j \in \I$ then $\E_{i,m}$ commutes with $\F_{j,-n}$. 

\item {[Proposition \ref{prop:Escommute}]}
For $i \in \I$ we have $\E_{i} \E_{i,1} \1_\l \cong \E_{i,1} \E_{i} \la -2 \ra \1_\l$. 

\item {[Proposition \ref{prop:cones}]}
If $\la i,j \ra = -1$ then there exist unique (up to a multiple) nonzero maps 
$$\alpha \in \Hom(\E_{i} \E_{j,1} \1_\l, \E_{j,1} \E_{i} \la 1 \ra \1_\l) \text{ and } 
\beta \in \Hom(\E_{j} \E_{i,1} \1_\l, \E_{i,1} \E_{j} \la 1 \ra \1_\l)$$
and $\Cone(\alpha) \cong \Cone(\beta)$. Meanwhile, if $\la i,j \ra = 0$ then $\E_{i,m}$ and $\E_{j,n}$ commute. 

\end{itemize}
\end{Theorem}

\begin{cor}\label{cor:main1}
Maps (\ref{eq:Emaps1}) and (\ref{eq:thetamap}) define a $(\hg,\theta)$ action. 
\end{cor}
\begin{proof}
All but the last three conditions of having a $(\hg,\theta)$ action are checked in Theorem \ref{thm:main1}. The last three conditions are just statements about the basic representation of $\hg$ which are easy to check.
\end{proof}

\subsection{Divided powers}\label{sec:powers}

One of the implications of a $(\hg,\theta)$ action is that one has divided powers. In other words, there exist $\E_{i,m}^{(r)} \1_\l$ and $\1_\l \F_{i,-m}^{(r)}$ in $\Kom(\K)$ such that 
$$\E_{i,m}^r \1_\l \cong \bigoplus_{[n]!} \E_{i,m}^{(r)} \1_\l \ \ \text{ and } \ \ \1_\l \F_{i,-m}^r \cong \bigoplus_{[n]!} \1_\l \F_{i,-m}^{(r)}.$$
Although we know such complexes exist it is not clear what they are explicitly. We conjecture the following. 

If $k := -(\l_i+r+m) \ge 0$ then 
\begin{equation*}
\E_{i,m}^{(r)} \1_\l \cong 
\left[ \dots \rightarrow \bigoplus_{w(\mu) \le r, |\mu|=l} \P_i^{(k^r, \mu^t)} \Q_i^{(\mu)} \la -l \ra \rightarrow \dots \rightarrow \P_i^{(k^r,1)} \Q_i \la -1 \ra \rightarrow \P_i^{(k^r)} \right] \la -kr - \binom{r}{2} \ra [kr + \binom{r}{2}] 
\end{equation*}
where the sum is over all partitions $\mu$ of size $|\mu|$ which fit in a box of width $r$ ($w(\mu)$ denotes the width of $\mu$). 

Similarly, if $k := \l_i + r + m \ge 0$ then 
\begin{equation*}
\E_{i,m}^{(r)} \1_\l \cong 
\left[ \dots \rightarrow \bigoplus_{w(\mu) \le r, |\mu|=l} \P_i^{(\mu^t)} \Q_i^{(r^k, \mu)} \la -l \ra \rightarrow \dots \rightarrow \P_i \Q_i^{(r^k,1)} \la -1 \ra \rightarrow \Q_i^{(r^k)} \right] \la - \binom{r}{2} \ra [\binom{r}{2}].
\end{equation*}
The differentials are defined using caps as above. For example, in the first complex above there is a map 
$$\P_i^{(k^r, \mu^t)} \Q_i^{(\mu)} \la -l \ra \rightarrow \P_i^{(k^r, \nu^t)} \Q_i^{(\nu)} \la -l+1 \ra$$
if and only if $\mu$ is obtained from $\nu$ by adding a box. Moreover, in this case, this map is unique up to scalar and given by a cap. The scalar multiples need to be chosen so as to satisfy $\partial^2=0$, but it is not hard to see (using the same arguments as in Lemma \ref{lem:EEtcpx}) that any two such choices of scalars yield homotopic complexes.

Similarly, the divided powers $\1_\l \F_{i,-m}^{(r)}$ are defined as the appropriately shifted right (or equivalently left) adjoints of the complexes above. The shifts above are to ensure that 
\begin{enumerate}
\item $(\E^{(r)}_{i,m} \1_\l)_R \cong \1_\l \F^{(r)}_{i,-m} \la r(\l_i + r + m) \ra$
\item $(\E^{(r)}_{i,m} \1_\l)_L \cong \1_\l \F^{(r)}_{i,-m} \la -r(\l_i + r + m) \ra$
\end{enumerate}
We have checked the conjecture above in the case $r=2$. 

\subsection{Further relations}

One may wonder if any of the relations in the extended definition from section \ref{sec:quantumaffine} have categorical analogues in this case. It turns out that most of them do, as we now explain. 

In what follows we define the 1-morphisms $\P_i^{[1^n]} \1_\l$ and $\Q_i^{[1^n]} \1_\l$ as in (\ref{eq:Pcpx}) and (\ref{eq:Qcpx}) from section \ref{sec:Kom(H)}. The same techniques used to prove Theorem \ref{thm:main1} can be used to check the following relations in $\Kom(\K)$. Although we have checked all these relations we do not include a proof here because it is not strictly necessary and also to save space. 

\begin{enumerate}
\item $(\P_{i}^{[1^n]} \1_\l)_R \cong \Q_{i}^{[1^n]} \la -n \ra$ and $(\P_{i}^{[1^n]} \1_\l)_L \cong \Q_{i}^{[1^n]} \la n \ra$.

\item If $m+n \ne 0$ and $i \in \I$ then there exist distinguished triangles
\begin{align*}
\Q_i^{[1^{m+n}]} \1_\l \la \l_i - n \ra \rightarrow \E_{i,m} \F_{i,n} \1_\l \rightarrow \F_{i,n} \E_{i,m} \1_\l &\text{ if } m+n > 0, \\
\F_{i,n} \E_{i,m} \1_\l \rightarrow \E_{i,m} \F_{i,n} \1_\l \rightarrow \P_{i}^{[1^{-m-n}]} \1_\l \la - \l_i - m \ra &\text{ if } m+n < 0.
\end{align*}

\item For $i \in \I$ and $m \in \Z$ we have the distinguished triangles
\begin{align*}
& \E_{i,m-1} \otimes_\k V_1 \la -1 \ra \1_\l \rightarrow \P_i^{[1]} \E_{i,m} \1_\l \rightarrow \E_{i,m} \P_i^{[1]} \1_\l \\
& \E_{i,m} \Q_i^{[1]} \1_\l \rightarrow \Q_i^{[1]} \E_{i,m} \1_\l \rightarrow \E_{i,m+1} \otimes_\k V_1 \1_\l \\
& \Q_i^{[1]} \F_{i,m} \1_\l \rightarrow \F_{i,m} \Q_i^{[1]} \1_\l \rightarrow  \F_{i,m+1} \otimes_\k V_1 \la 1 \ra \1_\l \\
& \F_{i,m-1} \otimes_\k V_1 \1_\l \rightarrow \F_{i,m} \P_i^{[1]} \1_\l \rightarrow \P_i^{[1]} \F_{i,m} \1_\l.
\end{align*}

\item If $\la i, j \ra = -1$ then there exist distinguished triangles
\begin{align*}
& \E_{i,m-1} [1] \la -1 \ra \1_\l \rightarrow \E_{i,m} \P_j^{[1]} \1_\l \rightarrow \P_j^{[1]} \E_{i,m} \1_\l \\
& \E_{i,m+1} \1_\l \rightarrow \Q_j^{[1]} \E_{i,m} \1_\l \rightarrow \E_{i,m} \Q_j^{[1]} \1_\l \\
& \F_{i,m} \Q_j^{[1]} \1_\l \rightarrow \Q_j^{[1]} \F_{i,m} \1_\l \rightarrow \F_{i,m+1} [-1] \la 1 \ra \1_\l \\
& \P_j^{[1]} \F_{i,m} \1_\l \rightarrow \F_{i,m} \P_j^{[1]} \1_\l \rightarrow \F_{i,m-1} \1_\l 
\end{align*}
while if $\la i, j \ra = 0$ then $\E_{i,m}$ and $\F_{i,m}$ commute with $\P_j^{[1]}$ and $\Q_j^{[1]}$. 
\end{enumerate}

\begin{Remark} 
For any distinguished triangle above it is possible to find an example where the triangle does not split.
\end{Remark}


\section{The $\sl_2$ commutator relation}

In this section we prove the $\sl_2$ commutator relation which appear in Theorem \ref{thm:main1}. Since we are only dealing with $\sl_2$ we abbreviate $\P_i$ by $\P$, $\Q_i$ by $\Q$ while $\l$ is now just the integer $\l_i$. 

The following is a key tool we will repeatedly use to simplify complexes of 1-morphisms in the homotopy category of $\H$ (it is a minor generalization of a lemma of Bar-Natan).  

\begin{lemma}[Gaussian elimination]\label{lem:cancel}
Let $ X, Y, Z, W, U, V$ be six objects in an additive category and consider a complex 
\begin{equation}\label{eq:6.5}
\dots \rightarrow U \xrightarrow{u} X \oplus Y \xrightarrow{f} Z \oplus W \xrightarrow{v} V \rightarrow \dots
\end{equation}
where $f = \left( \begin{matrix} A & B \\ C & D \end{matrix} \right)$ and $u,v$ are arbitrary morphisms. If $D: Y \rightarrow W$ is an isomorphism, then (\ref{eq:6.5}) is homotopic to a complex 
\begin{eqnarray}\label{eq:7}
\dots \rightarrow U \xrightarrow{u} X \xrightarrow{A-BD^{-1}C} Z \xrightarrow{v|_Z} V \rightarrow \dots
\end{eqnarray}
\end{lemma}
\begin{proof}
The key is the following commutative diagram: 
\begin{equation*}
\xymatrix{
\dots \ar[r] & U \ar[r]^-{u} \ar[d]^{\id} & X \oplus Y \ar[r]^f \ar[d]^{\alpha} & Z \oplus W \ar[d]^{\beta} \ar[r]^-{v} & V \ar[r] \ar[d]^{\id} & \dots \\
\dots \ar[r] & U \ar[r]^-{\alpha u} & X \oplus Y \ar[r]^g & Z \oplus W \ar[r]^-{v \beta^{-1}} & V \ar[r] & \dots 
}
\end{equation*}
where 
$$g = \left( \begin{matrix} A-BD^{-1}C & 0 \\ 0 & D \end{matrix} \right) \hspace{.5cm} \alpha = \left( \begin{matrix} 1 & 0 \\ D^{-1}C & 1 \end{matrix} \right) \hspace{.5cm} \beta = \left( \begin{matrix} 1 & -BD^{-1} \\ 0 & 1 \end{matrix} \right).$$
The vertical map of complexes is a homotopy equivalence and it is straightforward to check that the bottom row is homotopic to (\ref{eq:7}). 
\end{proof}

\begin{prop}\label{prop:adjointrels} 
The left and right adjoints of $\E_{i,m}^{(r)} \1_\l$ are given by
$$(\E_{i,m} \1_\l)_R \cong \1_\l \F_{i,-m} \la \l+1+m \ra \text{ and } (\E_{i,m} \1_\l)_L \cong  \1_\l \F_{i,-m} \la -\l-1-m \ra.$$
\end{prop}
\begin{proof}
This is a consequence of the fact that $\P_R \cong \Q \la -1 \ra$ and $\P_L \cong \Q \la 1 \ra$. 
\end{proof}

The following is the main result of this section. 

\begin{prop}\label{prop:sl2comm} We have 
\begin{eqnarray}
\label{eq:comm1}
\F_{i,-m} \E_{i,m} \1_{\l} &\cong& \E_{i,m} \F_{i,-m} \1_{\l} \oplus \1_\l \otimes_\k V_{- \l - m -1}  \text{ if } \l + m \le 0 \\
\label{eq:comm2}
\E_{i,m} \F_{i,-m} \1_{\l} &\cong& \F_{i,-m} \E_{i,m} \1_{\l} \oplus \1_\l \otimes_\k V_{\l + m - 1} \text{ if } \l + m \ge 0.
\end{eqnarray}
\end{prop}

We will only prove (\ref{eq:comm2}) as (\ref{eq:comm1}) follows similarly. Because of the way $\E_{i,m}$ is defined it suffices to prove the case $m=0$. To do this we identify in the following two propositions explicit expressions for $\F_i \E_i \1_\l$ and $\E_i \F_i \1_\l$ which we then compare. For simplicity we will write $\E$ and $\F$ for $\E_i$ and $\F_i$. 

\begin{prop}\label{prop:FEcomp}
The composition $\F \E \1_\l$ is homotopic to the complex
\begin{equation}\label{eq:FE}
\dots \rightarrow \bigoplus_{i \ge 0,-l} \P^{(i+1, 1^{\l+i+l})} \Q^{(l+i+1, 1^{\l+i})} \la l \ra \rightarrow \bigoplus_{i \ge 0,-l-1} \P^{(i+1, 1^{\l+i+l+1})} \Q^{(l+i+2, 1^{\l+i})} \la l+1 \ra \rightarrow \dots
\end{equation}
where the nonzero part of the differential maps the summand $\P^{(i+1, 1^{\l+i+l})} \Q^{(l+i+1, 1^{\l+i})} \la l \ra$ to the two terms 
$$\P^{(i+1,1^{\l+i+l+1})} \Q^{(l+i+2,1^{\l + i})} \la l+1 \ra \text{ and } \P^{(i,1^{\l + i + l})} \Q^{(l+i+1,1^{\l+i-1})} \la l+1 \ra$$
using a cup and cap as in figures (\ref{fig:1}) and (\ref{fig:2}) below. 
\end{prop}
\begin{equation}\label{fig:1}
\begin{tikzpicture}[scale=.75][>=stealth]
\shadedraw[gray]  (0,0) rectangle (4,.5);
\draw (2,.25) node {$(i+1,1^{\l + i+ l})$};
\shadedraw[gray]  (5,0) rectangle (9,.5);
\draw (7,.25) node {$(l+i+1,1^{\l + i})$};
\draw (1,.5) -- (1,1) [->][very thick];
\draw (3,.5) -- (3,1) [->][very thick];
\draw (6,.5) -- (6,1) [<-][very thick];
\draw (8,.5) -- (8,1) [<-][very thick];

\draw (-.25,1) rectangle (1.75,1.5);
\draw (.75,1.25) node {$(i)$};
\filldraw [gray] (2.25,1) rectangle (4.25,1.5);
\draw (3.25,1.25) node {$(1^{\l+i+l+1})$};
\draw (4.4,1) rectangle (7.1,1.5);
\draw (5.75,1.25) node {$(l+i)$};
\filldraw[gray] (7.5,1) rectangle (9.5,1.5);
\draw (8.5,1.25) node {$(1^{\l + i+1})$};

\draw [shift = {+(0,2)}](-.25,1) rectangle (1.75,1.5);
\draw [shift = {+(0,2)}](.75,1.25) node {$(i)$};
\filldraw[gray] [shift = {+(0,2)}](2.25,1) rectangle (4.25,1.5);
\draw [shift = {+(0,2)}](3.25,1.25) node {$(1^{\l + i + l + 2})$};
\draw [shift = {+(0,2)}](4.4,1) rectangle (7.1,1.5);
\draw [shift = {+(0,2)}](5.75,1.25) node {$(l+i+1)$};
\filldraw[gray] [shift = {+(0,2)}](7.5,1) rectangle (9.5,1.5);
\draw [shift = {+(0,2)}](8.5,1.25) node {$(1^{\l + i + 1})$};

\draw (1,1.5) -- (1,3) [->][very thick];
\draw (3,1.5) -- (3,3) [->][very thick];
\draw (6,1.5) -- (6,3) [<-][very thick];
\draw (8,1.5) -- (8,3) [<-][very thick];
\draw (3.75,3) arc (180:360:.75cm and .5cm)[<-][very thick];

\shadedraw[gray] [shift = {+(0,4)}](0,0) rectangle (4,.5);
\draw [shift = {+(0,4)}](2,.25) node {$(i+1,1^{\l + i+ l+1})$};
\shadedraw[gray]  [shift = {+(0,4)}](5,0) rectangle (9,.5);
\draw [shift = {+(0,4)}](7,.25) node {$(l+i+2,1^{\l +i})$};
\draw [shift = {+(0,3)}](1,.5) -- (1,1) [->][very thick];
\draw [shift = {+(0,3)}](3,.5) -- (3,1) [->][very thick];
\draw [shift = {+(0,3)}](6,.5) -- (6,1) [<-][very thick];
\draw [shift = {+(0,3)}](8,.5) -- (8,1) [<-][very thick];

\end{tikzpicture}
\end{equation}

\begin{equation}\label{fig:2}
\begin{tikzpicture}[scale=.75][>=stealth]
\shadedraw[gray]  (0,0) rectangle (4,.5);
\draw (2,.25) node {$(i+1,1^{\l + i+ l})$};
\shadedraw[gray]  (5,0) rectangle (9,.5);
\draw (7,.25) node {$(l+i+1,1^{\l + i})$};
\draw (1,.5) -- (1,1) [->][very thick];
\draw (3,.5) -- (3,1) [->][very thick];
\draw (6,.5) -- (6,1) [<-][very thick];
\draw (8,.5) -- (8,1) [<-][very thick];

\filldraw [gray] (-.25,1) rectangle (1.75,1.5);
\draw (.75,1.25) node {$(1^{\l+i+l+1})$};
\draw (2.25,1) rectangle (4.25,1.5);
\draw (3.25,1.25) node {$(i)$};
\filldraw[gray] (4.4,1) rectangle (7.1,1.5);
\draw (5.75,1.25) node {$(1^{\l + i + 1})$};
\draw (7.5,1) rectangle (9.5,1.5);
\draw (8.5,1.25) node {$(l+i)$};

\filldraw[gray] [shift = {+(0,2)}](-.25,1) rectangle (1.75,1.5);
\draw [shift = {+(0,2)}](.75,1.25) node {$(1^{\l+i+l+1})$};
\draw [shift = {+(0,2)}](2.25,1) rectangle (4.25,1.5);
\draw [shift = {+(0,2)}](3.25,1.25) node {$(i-1)$};
\filldraw[gray] [shift = {+(0,2)}](4.4,1) rectangle (7.1,1.5);
\draw [shift = {+(0,2)}](5.75,1.25) node {$(1^{\l+i})$};
\draw [shift = {+(0,2)}](7.5,1) rectangle (9.5,1.5);
\draw [shift = {+(0,2)}](8.5,1.25) node {$(l+i)$};

\draw (1,1.5) -- (1,3) [->][very thick];
\draw (3,1.5) -- (3,3) [->][very thick];
\draw (6,1.5) -- (6,3) [<-][very thick];
\draw (8,1.5) -- (8,3) [<-][very thick];
\draw (5.25,1.5) arc (0:180:.75cm and .5cm)[<-][very thick];

\shadedraw[gray] [shift = {+(0,4)}](0,0) rectangle (4,.5);
\draw [shift = {+(0,4)}](2,.25) node {$(i,1^{\l + i+ l})$};
\shadedraw[gray]  [shift = {+(0,4)}](5,0) rectangle (9,.5);
\draw [shift = {+(0,4)}](7,.25) node {$(l+i+1,1^{\l +i-1})$};
\draw [shift = {+(0,3)}](1,.5) -- (1,1) [->][very thick];
\draw [shift = {+(0,3)}](3,.5) -- (3,1) [->][very thick];
\draw [shift = {+(0,3)}](6,.5) -- (6,1) [<-][very thick];
\draw [shift = {+(0,3)}](8,.5) -- (8,1) [<-][very thick];
\end{tikzpicture}
\end{equation}
\begin{proof}
The composition $\F \E \1_\l$ is equal to the complex
\begin{equation}\label{eq:1}
\dots \rightarrow \bigoplus_{i \ge 0, -l} \P^{(1^{\l+1+i+l})} \Q^{(l+i)} \P^{(i)} \Q^{(1^{\l+1+i})} \la l \ra \rightarrow \dots
\end{equation}
where the terms occurs in cohomological degree $l$. One can decompose the $\Q^{(l+i)} \P^{(i)}$ part of the expression to obtain 
\begin{equation}\label{eq:2}
\dots \rightarrow \bigoplus_{i \ge 0,-l \ge j-i} \P^{(1^{\l+1+i+l})} \P^{(i-j)} \Q^{(l+i-j)} \Q^{(1^{\l+1+i})} \otimes_\k V_j \la l \ra \rightarrow \dots
\end{equation}
where $V_j$ is the graded vector space in (\ref{eq:V_j}). Now, by Proposition \ref{prop:rels1}, each term above breaks up into four terms. Let us consider the term $\P^{(i_0-j_0, 1^{\l+i_0+l_0+1})} \Q^{(l_0+i_0-j_0, 1^{\l+1+i_0})}$ for some fixed $(i_0,j_0,l_0)$.  This indecomposable 1-morphism 
occurs four times, namely, when 
\begin{equation}\label{eq:temp}
(i,j,l) \text{ equals } (i_0,j_0,l_0), (i_0+1,j_0+2,l_0), (i_0+1,j_0+1,l_0-1) \text{ or } (i_0,j_0+1,l_0+1).
\end{equation}
When $-1 \le j_0 < \min(i_0,i_0+l)$, these terms taken together form a subcomplex 
\begin{equation}\label{eq:6}
\P^{(i_0-j_0, 1^{\l+1+i_0+l_0})} \Q^{(l_0+i_0-j_0, 1^{\l+1+i_0})} \la l_0 \ra \otimes_\k \left( V_{j_0+1} \la -1 \ra \rightarrow V_{j_0} \oplus V_{j_0+2} \rightarrow V_{j_0+1} \la 1 \ra \right)
\end{equation}
We empose the restriction $-1 \le j_0 < \min(i_0,i_0+l_0)$ because otherwise the term in (\ref{eq:2}) disappears for at least one of the choices of $(i,j,l)$ in (\ref{eq:temp}). In Lemma \ref{lem:1} below, we show that the complex in equation (\ref{eq:6}) exact. Thus we can cancel out such terms using the cancellation Lemma \ref{lem:cancel}. 

If $j_0 > \min(i_0,i_0+l_0)$ then the term $\P^{(i_0-j_0, 1^{\l+i_0+l_0+1})} \Q^{(\l+i_0-j_0, 1^{\l+1+i_0})}$ vanishes because (\ref{eq:2}) disappears for all choices of $(i,j,l)$ in (\ref{eq:temp}). So it remains to study the case $j_0 = \min(i_0,i_0+l_0)$. For convenience let us assume $l_0 \ge 0$ so that $j_0 = i_0$ (the case $l_0 \le 0$ is similar). 

{\bf Case $j_0 = i_0$, $|l_0| \ge 2$.} Here we are interested in terms of the form $\P^{(1^{\l+1+i_0+l_0})} \Q^{(l_0,1^{\l+1+i_0})}$ when $l_0 \ge 2$. Looking again at (\ref{eq:2}), such terms occur when 
$$(i,j,l) \text{ equals } (i_0,i_0,l_0), (i_0,i_0-1,l_0-1), (i_0+1,i_0+1,l_0-1) \text{ or } (i_0+1,i_0,l_0-2).$$ 
So the resulting complex looks like 
$$\P^{(1^{\l+1+i_0+l_0})} \Q^{(l_0,1^{\l+1+i_0})} \la l_0-1 \ra \otimes_\k \left(V_{i_0} \la -1 \ra \rightarrow V_{i_0+1} \oplus V_{i_0-1} \rightarrow V_{i_0} \la 1 \ra \right).$$
The same type of argument used in Lemma \ref{lem:1} works to show that it is exact. 

{\bf Case $j_0 = i_0$, $l_0=0,1$.} Here we are looking at terms of the form $\P^{(1^a)} \Q^{(1^a)}$. Such terms occur in (\ref{eq:2}) when 
$$(i,j,l) \text{ equals } (a-\l,a-\l-1,-1), (a-\l,a-\l,0), (a-\l-1,a-\l-2,0) \text{ or } (a-\l-1,a-\l-1,1)$$ 
and again we get an exact complex. 

{\bf Case $j_0=-2$.} So it seems like every complex is exact, but we missed one case, namely terms of the form $\P^{(i+1, 1^{\l+i+l})} \Q^{(l+i+1, 1^{\l+i})}$ which occur when $j_0=-2$. Then three of the four terms in the complex in (\ref{eq:6}) are zero and we are left with just one term $\P^{(i+1, 1^{\l+i+l})} \Q^{(l+i+1, 1^{\l+i})} \la l \ra$ in degree $l$. Such indecomposable terms only occur once and hence cannot cancel out with anything else. Putting all these terms together we obtain a complex like in (\ref{eq:FE}). 

Finally, we need to determine the differential 
$$\P^{(i+1, 1^{\l+i+l})} \Q^{(l+i+1, 1^{\l+i})} \la l \ra \rightarrow \P^{(i+1,1^{\l+i+l+1})} \Q^{(l+i+2,1^{\l+i})} \la l+1 \ra.$$ 
For degree reasons it is not hard to see that the cancellation above does not alter the original differential between such terms. This differential is given by the following composition 
$$
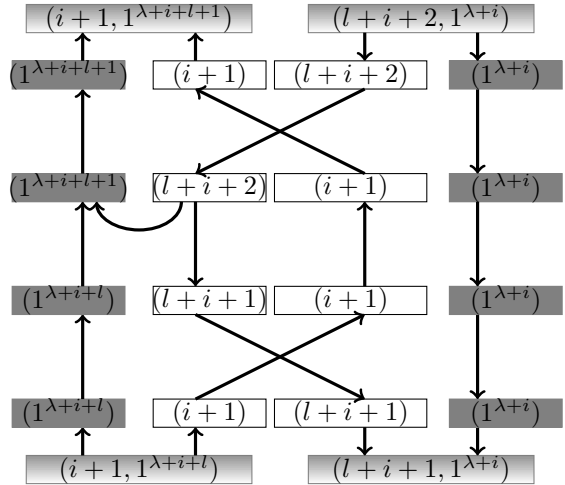
\begin{tikzpicture}[scale=.75][>=stealth]
\shadedraw[gray]  (0,0) rectangle (4,.5);
\draw (2,.25) node {$(i+1,1^{\l + i+ l})$};
\shadedraw[gray]  (5,0) rectangle (9,.5);
\draw (7,.25) node {$(l+i+1,1^{\l + i})$};
\draw (1,.5) -- (1,1) [->][very thick];
\draw (3,.5) -- (3,1) [->][very thick];
\draw (6,.5) -- (6,1) [<-][very thick];
\draw (8,.5) -- (8,1) [<-][very thick];

\filldraw [gray](-.25,1) rectangle (1.75,1.5);
\draw (.75,1.25) node {$(1^{\l+i+l})$};
\draw (2.25,1) rectangle (4.25,1.5);
\draw (3.25,1.25) node {$(i+1)$};
\draw (4.4,1) rectangle (7.1,1.5);
\draw (5.75,1.25) node {$(l+i+1)$};
\filldraw[gray] (7.5,1) rectangle (9.5,1.5);
\draw (8.5,1.25) node {$(1^{\l + i})$};

\draw [shift = {+(0,2)}](1,-.5) -- (1,1) [->][very thick];
\draw [shift = {+(0,2)}](3,-.5) -- (6,1) [->][very thick];
\draw [shift = {+(0,2)}](6,-.5) -- (3,1) [<-][very thick];
\draw [shift = {+(0,2)}](8,-.5) -- (8,1) [<-][very thick];

\filldraw [gray][shift = {+(0,2)}](-.25,1) rectangle (1.75,1.5);
\draw [shift = {+(0,2)}](.75,1.25) node {$(1^{\l+i+l})$};
\draw [shift = {+(0,2)}](2.25,1) rectangle (4.25,1.5);
\draw [shift = {+(0,2)}](3.25,1.25) node {$(l+i+1)$};
\draw [shift = {+(0,2)}](4.4,1) rectangle (7.1,1.5);
\draw [shift = {+(0,2)}](5.75,1.25) node {$(i+1)$};
\filldraw[gray] [shift = {+(0,2)}](7.5,1) rectangle (9.5,1.5);
\draw [shift = {+(0,2)}](8.5,1.25) node {$(1^{\l + i})$};

\filldraw [gray][shift = {+(0,4)}](-.25,1) rectangle (1.75,1.5);
\draw [shift = {+(0,4)}](.75,1.25) node {$(1^{\l+i+l+1})$};
\draw [shift = {+(0,4)}](2.25,1) rectangle (4.25,1.5);
\draw [shift = {+(0,4)}](3.25,1.25) node {$(l+i+2)$};
\draw [shift = {+(0,4)}](4.4,1) rectangle (7.1,1.5);
\draw [shift = {+(0,4)}](5.75,1.25) node {$(i+1)$};
\filldraw[gray] [shift = {+(0,4)}](7.5,1) rectangle (9.5,1.5);
\draw [shift = {+(0,4)}](8.5,1.25) node {$(1^{\l + i})$};

\draw (1.25,5) arc (180:360:.75cm and .5cm)[<-][very thick];
\draw [shift = {+(0,4)}](1,-.5) -- (1,1) [->][very thick];
\draw [shift = {+(0,4)}](3,-.5) -- (3,1) [<-][very thick];
\draw [shift = {+(0,4)}](6,-.5) -- (6,1) [->][very thick];
\draw [shift = {+(0,4)}](8,-.5) -- (8,1) [<-][very thick];

\draw [shift = {+(0,6)}](1,-.5) -- (1,1) [->][very thick];
\draw [shift = {+(0,6)}](3,-.5) -- (6,1) [<-][very thick];
\draw [shift = {+(0,6)}](6,-.5) -- (3,1) [->][very thick];
\draw [shift = {+(0,6)}](8,-.5) -- (8,1) [<-][very thick];

\filldraw [gray][shift = {+(0,6)}](-.25,1) rectangle (1.75,1.5);
\draw [shift = {+(0,6)}](.75,1.25) node {$(1^{\l+i+l+1})$};
\draw [shift = {+(0,6)}](2.25,1) rectangle (4.25,1.5);
\draw [shift = {+(0,6)}](3.25,1.25) node {$(i+1)$};
\draw [shift = {+(0,6)}](4.4,1) rectangle (7.1,1.5);
\draw [shift = {+(0,6)}](5.75,1.25) node {$(l+i+2)$};
\filldraw[gray] [shift = {+(0,6)}](7.5,1) rectangle (9.5,1.5);
\draw [shift = {+(0,6)}](8.5,1.25) node {$(1^{\l + i})$};

\shadedraw[gray] [shift = {+(0,8)}](0,0) rectangle (4,.5);
\draw [shift = {+(0,8)}](2,.25) node {$(i+1,1^{\l + i+ l+1})$};
\shadedraw[gray]  [shift = {+(0,8)}](5,0) rectangle (9,.5);
\draw [shift = {+(0,8)}](7,.25) node {$(l+i+2,1^{\l +i})$};
\draw [shift = {+(0,7)}](1,.5) -- (1,1) [->][very thick];
\draw [shift = {+(0,7)}](3,.5) -- (3,1) [->][very thick];
\draw [shift = {+(0,7)}](6,.5) -- (6,1) [<-][very thick];
\draw [shift = {+(0,7)}](8,.5) -- (8,1) [<-][very thick];
\end{tikzpicture}
$$
where the cup in the middle of the diagram comes from the differential in the complex for $\F$. It is not hard to check that this diagram simplifies to give the map in (\ref{fig:1}). Likewise, one has such a diagram which simplifies to give the map in (\ref{fig:2}). Finally, all the other possible differential have to be zero by Lemma \ref{lem:differentials}. Thus, in the end, we get the complex in (\ref{eq:FE}). 
\end{proof}

\begin{lemma}\label{lem:1}
The complex in (\ref{eq:6}) is exact. 
\end{lemma}
\begin{proof}
We need to show that the complex 
\begin{equation}\label{eq:6'}
\P^{(i_0-j_0, 1^{\l+1+i_0+l_0})} \Q^{(l+i_0-j_0, 1^{\l+1+i_0})} \la l_0 \ra \otimes_\k \left( V_{j_0+1} \la -1 \ra \rightarrow V_{j_0} \oplus V_{j_0+2} \rightarrow V_{j_0+1} \la 1 \ra \right)
\end{equation}
is exact. To do this we first show that the map $\alpha: V_{j_0+1}\la -1 \ra \rightarrow V_{j_0+2}$ at the bottom left of the complex is injective. This map is given as a sum of the following compositions: 

$$
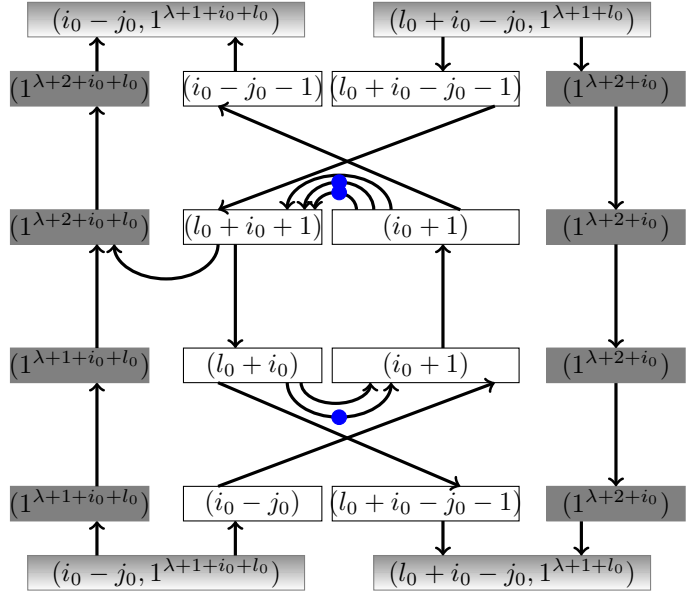
\begin{tikzpicture}[scale=.92][>=stealth]
\shadedraw[gray]  (0,0) rectangle (4,.5);
\draw (2,.25) node {$(i_0-j_0,1^{\l + 1 + i_0 + l_0})$};
\shadedraw[gray]  (5,0) rectangle (9,.5);
\draw (7,.25) node {$(l_0+i_0-j_0,1^{\l + 1 + l_0})$};
\draw (1,.5) -- (1,1) [->][very thick];
\draw (3,.5) -- (3,1) [->][very thick];
\draw (6,.5) -- (6,1) [<-][very thick];
\draw (8,.5) -- (8,1) [<-][very thick];

\filldraw[gray] (-.25,1) rectangle (1.75,1.5);
\draw (.75,1.25) node {$(1^{\l + 1 + i_0 + l_0})$};
\draw (2.25,1) rectangle (4.25,1.5);
\draw (3.25,1.25) node {$(i_0-j_0)$};
\draw (4.4,1) rectangle (7.1,1.5);
\draw (5.75,1.25) node {$(l_0+i_0-j_0-1)$};
\filldraw[gray] (7.5,1) rectangle (9.5,1.5);
\draw (8.5,1.25) node {$(1^{\l + 2 + i_0})$};

\draw (1,1.5) -- (1,3)[->][very thick];
\draw (8.5,1.5) -- (8.5,3)[<-][very thick];
\draw (2.75,1.5) -- (6.75,3) [->] [very thick];
\draw (6.25,1.5) -- (2.75,3) [<-] [very thick];
\draw (3.75,3) arc (180:360:.75cm and .5cm)[->][very thick];
\draw (3.95,3) arc (180:360:.5cm and .3cm)[->][very thick];
\filldraw [blue](4.5,2.5) circle (3pt);

\filldraw[gray] [shift = {+(0,2)}](-.25,1) rectangle (1.75,1.5);
\draw [shift = {+(0,2)}](.75,1.25) node {$(1^{\l + 1 + i_0 + l_0})$};
\draw [shift = {+(0,2)}](2.25,1) rectangle (4.25,1.5);
\draw [shift = {+(0,2)}](3.25,1.25) node {$(l_0+i_0)$};
\draw [shift = {+(0,2)}](4.4,1) rectangle (7.1,1.5);
\draw [shift = {+(0,2)}](5.75,1.25) node {$(i_0+1)$};
\filldraw[gray] [shift = {+(0,2)}](7.5,1) rectangle (9.5,1.5);
\draw [shift = {+(0,2)}](8.5,1.25) node {$(1^{\l + 2 + i_0})$};

\filldraw[gray] [shift = {+(0,4)}](-.25,1) rectangle (1.75,1.5);
\draw [shift = {+(0,4)}](.75,1.25) node {$(1^{\l + 2 + i_0 + l_0})$};
\draw [shift = {+(0,4)}](2.25,1) rectangle (4.25,1.5);
\draw [shift = {+(0,4)}](3.25,1.25) node {$(l_0+i_0+1)$};
\draw [shift = {+(0,4)}](4.4,1) rectangle (7.1,1.5);
\draw [shift = {+(0,4)}](5.75,1.25) node {$(i_0+1)$};
\filldraw[gray] [shift = {+(0,4)}](7.5,1) rectangle (9.5,1.5);
\draw [shift = {+(0,4)}](8.5,1.25) node {$(1^{\l + 2 + i_0})$};

\draw [shift = {+(0,2)}](1,1.5) -- (1,3)[->][very thick];
\draw [shift = {+(0,2)}](8.5,1.5) -- (8.5,3)[<-][very thick];
\draw  [shift = {+(0,4)}](3,-.5) -- (3,1) [<-][very thick];
\draw  [shift = {+(0,4)}](6,-.5) -- (6,1) [->][very thick];
\draw (1.25,5) arc (180:360:.75cm and .5cm)[<-][very thick];

\draw [shift = {+(4,5)}](1.25,.5) arc (0:180:.75cm and .5cm)[->][very thick];
\draw [shift = {+(4,5)}](1,.5) arc (0:180:.5cm and .4cm)[->][very thick];
\draw [shift = {+(4,5)}](.75,.5) arc (0:180:.3cm and .25cm)[->][very thick];
\draw  [shift = {+(0,4)}](1,1.5) -- (1,3)[->][very thick];
\draw  [shift = {+(0,4)}](8.5,1.5) -- (8.5,3)[<-][very thick];
\draw  [shift = {+(0,4)}](2.75,1.5) -- (6.75,3) [<-] [very thick];
\draw  [shift = {+(0,4)}](6.25,1.5) -- (2.75,3) [->] [very thick];
\filldraw [blue](4.5,5.9) circle (3pt);
\filldraw [blue](4.5,5.75) circle (3pt);

\filldraw[gray] [shift = {+(0,6)}](-.25,1) rectangle (1.75,1.5);
\draw [shift = {+(0,6)}](.75,1.25) node {$(1^{\l + 2 + i_0 + l_0})$};
\draw [shift = {+(0,6)}](2.25,1) rectangle (4.25,1.5);
\draw [shift = {+(0,6)}](3.25,1.25) node {$(i_0-j_0-1)$};
\draw [shift = {+(0,6)}](4.4,1) rectangle (7.1,1.5);
\draw [shift = {+(0,6)}](5.75,1.25) node {$(l_0+i_0-j_0-1)$};
\filldraw[gray] [shift = {+(0,6)}](7.5,1) rectangle (9.5,1.5);
\draw [shift = {+(0,6)}](8.5,1.25) node {$(1^{\l + 2 + i_0})$};

\shadedraw[gray] [shift = {+(0,8)}](0,0) rectangle (4,.5);
\draw [shift = {+(0,8)}](2,.25) node {$(i_0-j_0,1^{\l + 1 + i_0 + l_0})$};
\shadedraw[gray]  [shift = {+(0,8)}](5,0) rectangle (9,.5);
\draw [shift = {+(0,8)}](7,.25) node {$(l_0+i_0-j_0,1^{\l + 1 + l_0})$};
\draw [shift = {+(0,7)}](1,.5) -- (1,1) [->][very thick];
\draw [shift = {+(0,7)}](3,.5) -- (3,1) [->][very thick];
\draw [shift = {+(0,7)}](6,.5) -- (6,1) [<-][very thick];
\draw [shift = {+(0,7)}](8,.5) -- (8,1) [<-][very thick];
\end{tikzpicture}
$$
Here there are $j_0+1$ right pointing cups in the bottom part of the above picture, one left pointing cup in the middle, and $j_0 + 2$ left pointing caps in the top part. The sum is taken over a basis of the space of solid dots on the right pointing cups (this space is $V_{j_0+1}$ given by $0, 1, \dots, j_0+1$ dots) and a basis for the space of solid dots on the left cups (this space is $V_{j_0+2}$ given by $0,1, \dots, j_0+2$ dots).

If there are $k$ dots on the right pointing cups then the diagram evaluates to zero unless the left pointing caps have either $j_0+1-k$ or $j_0+2-k$ dots. This is because of the following three facts:
\begin{itemize}
\item two dots on the same strand is zero 
\item a counter-clockwise circle with no dots is also zero
\item a counter-clockwise loop on an upward strand is zero.
\end{itemize} 
Now let us look at the two cases when the diagram is nonzero.

{\bf Case 1.} In the first case, when the total number of dots is $j_0+1$ the diagram simplifies to 
$$
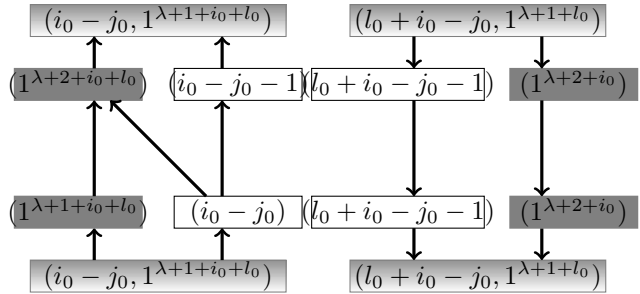
\begin{tikzpicture}[scale=.85][>=stealth]
\shadedraw[gray]  (0,0) rectangle (4,.5);
\draw (2,.25) node {$(i_0-j_0,1^{\l + 1 + i_0 + l_0})$};
\shadedraw[gray]  (5,0) rectangle (9,.5);
\draw (7,.25) node {$(l_0+i_0-j_0,1^{\l + 1 + l_0})$};
\draw (1,.5) -- (1,1) [->][very thick];
\draw (3,.5) -- (3,1) [->][very thick];
\draw (6,.5) -- (6,1) [<-][very thick];
\draw (8,.5) -- (8,1) [<-][very thick];

\filldraw[gray] (-.25,1) rectangle (1.75,1.5);
\draw (.75,1.25) node {$(1^{\l + 1 + i_0 + l_0})$};
\draw (2.25,1) rectangle (4.25,1.5);
\draw (3.25,1.25) node {$(i_0-j_0)$};
\draw (4.4,1) rectangle (7.1,1.5);
\draw (5.75,1.25) node {$(l_0+i_0-j_0-1)$};
\filldraw[gray] (7.5,1) rectangle (9.5,1.5);
\draw (8.5,1.25) node {$(1^{\l + 2 + i_0})$};

\filldraw[gray] [shift = {+(0,2)}](-.25,1) rectangle (1.75,1.5);
\draw [shift = {+(0,2)}](.75,1.25) node {$(1^{\l + 2 + i_0 + l_0})$};
\draw [shift = {+(0,2)}](2.25,1) rectangle (4.25,1.5);
\draw [shift = {+(0,2)}](3.25,1.25) node {$(i_0-j_0-1)$};
\draw [shift = {+(0,2)}](4.4,1) rectangle (7.1,1.5);
\draw [shift = {+(0,2)}](5.75,1.25) node {$(l_0+i_0-j_0-1)$};
\filldraw[gray] [shift = {+(0,2)}](7.5,1) rectangle (9.5,1.5);
\draw [shift = {+(0,2)}](8.5,1.25) node {$(1^{\l + 2 + i_0})$};

\draw (1,1.5) -- (1,3) [->][very thick];
\draw (2.75,1.5) -- (1.25,3) [->][very thick];
\draw (3,1.5) -- (3,3) [->][very thick];
\draw (6,1.5) -- (6,3) [<-][very thick];
\draw (8,1.5) -- (8,3) [<-][very thick];

\shadedraw[gray] [shift = {+(0,4)}](0,0) rectangle (4,.5);
\draw [shift = {+(0,4)}](2,.25) node {$(i_0-j_0,1^{\l + 1 + i_0 + l_0})$};
\shadedraw[gray]  [shift = {+(0,4)}](5,0) rectangle (9,.5);
\draw [shift = {+(0,4)}](7,.25) node {$(l_0+i_0-j_0,1^{\l + 1 + l_0})$};
\draw [shift = {+(0,3)}](1,.5) -- (1,1) [->][very thick];
\draw [shift = {+(0,3)}](3,.5) -- (3,1) [->][very thick];
\draw [shift = {+(0,3)}](6,.5) -- (6,1) [<-][very thick];
\draw [shift = {+(0,3)}](8,.5) -- (8,1) [<-][very thick];
\end{tikzpicture}.
$$
In this simplification we use the fact that a counter-clockwise circle with a degree two dot is equal to the identity and hence can be erased. Let us call the composition in the above diagram $f_1$. Now, the left part of the diagram, the part which involves only $\P$s, is the composition  
$$\P^{(a,1^b)} \rightarrow \P^{(a)} \P^{(1^b)} \rightarrow \P^{(a-1)} \P \P^{(1^b)} \rightarrow \P^{(a-1)} \P^{(1^{b+1})} \rightarrow \P^{(a,1^b)}$$
where $a = i_0-j_0$ and $b = \l+1+i_0+l_0$.  It is an exercise in the representation theory of the symmetric group that this composition is a nonzero multiple of the identity. Thus $f_1$ is (a nonzero multiple of) the identity.

{\bf Case 2.} In the second case, when the total number of dots is $j_0+2$ the diagram simplifies to almost the same thing, namely:
$$
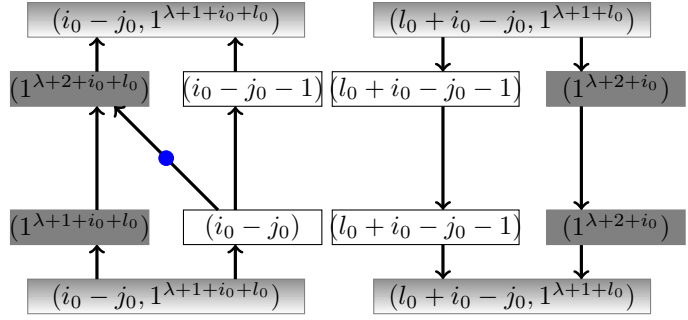
\begin{tikzpicture}[scale=.92][>=stealth]
\shadedraw[gray]  (0,0) rectangle (4,.5);
\draw (2,.25) node {$(i_0-j_0,1^{\l + 1 + i_0 + l_0})$};
\shadedraw[gray]  (5,0) rectangle (9,.5);
\draw (7,.25) node {$(l_0+i_0-j_0,1^{\l + 1 + l_0})$};
\draw (1,.5) -- (1,1) [->][very thick];
\draw (3,.5) -- (3,1) [->][very thick];
\draw (6,.5) -- (6,1) [<-][very thick];
\draw (8,.5) -- (8,1) [<-][very thick];

\filldraw[gray] (-.25,1) rectangle (1.75,1.5);
\draw (.75,1.25) node {$(1^{\l + 1 + i_0 + l_0})$};
\draw (2.25,1) rectangle (4.25,1.5);
\draw (3.25,1.25) node {$(i_0-j_0)$};
\draw (4.4,1) rectangle (7.1,1.5);
\draw (5.75,1.25) node {$(l_0+i_0-j_0-1)$};
\filldraw[gray] (7.5,1) rectangle (9.5,1.5);
\draw (8.5,1.25) node {$(1^{\l + 2 + i_0})$};

\filldraw[gray] [shift = {+(0,2)}](-.25,1) rectangle (1.75,1.5);
\draw [shift = {+(0,2)}](.75,1.25) node {$(1^{\l + 2 + i_0 + l_0})$};
\draw [shift = {+(0,2)}](2.25,1) rectangle (4.25,1.5);
\draw [shift = {+(0,2)}](3.25,1.25) node {$(i_0-j_0-1)$};
\draw [shift = {+(0,2)}](4.4,1) rectangle (7.1,1.5);
\draw [shift = {+(0,2)}](5.75,1.25) node {$(l_0+i_0-j_0-1)$};
\filldraw[gray] [shift = {+(0,2)}](7.5,1) rectangle (9.5,1.5);
\draw [shift = {+(0,2)}](8.5,1.25) node {$(1^{\l + 2 + i_0})$};

\draw (1,1.5) -- (1,3) [->][very thick];
\draw (2.75,1.5) -- (1.25,3) [->][very thick];
\filldraw [blue] (2,2.25) circle (3pt);
\draw (3,1.5) -- (3,3) [->][very thick];
\draw (6,1.5) -- (6,3) [<-][very thick];
\draw (8,1.5) -- (8,3) [<-][very thick];

\shadedraw[gray] [shift = {+(0,4)}](0,0) rectangle (4,.5);
\draw [shift = {+(0,4)}](2,.25) node {$(i_0-j_0,1^{\l + 1 + i_0 + l_0})$};
\shadedraw[gray]  [shift = {+(0,4)}](5,0) rectangle (9,.5);
\draw [shift = {+(0,4)}](7,.25) node {$(l_0+i_0-j_0,1^{\l + 1 + l_0})$};
\draw [shift = {+(0,3)}](1,.5) -- (1,1) [->][very thick];
\draw [shift = {+(0,3)}](3,.5) -- (3,1) [->][very thick];
\draw [shift = {+(0,3)}](6,.5) -- (6,1) [<-][very thick];
\draw [shift = {+(0,3)}](8,.5) -- (8,1) [<-][very thick];
\end{tikzpicture}
$$
Let us denote this map by $f_2$. 

In terms of $f_1 = \id$ and $f_2$ the matrix for $\alpha: V_{j_0+1} \rightarrow V_{j_0+2}$ is
$$\alpha = 
\left[ \begin{matrix}  
0 & \hdots & 0 & f_2 \\
0 & \hdots & f_2 & \id \\ 
\hdots & \hdots & \hdots & \hdots \\ 
f_2 & \id  & 0 & \hdots \\ 
\id & 0 &\hdots & 0.
\end{matrix} \right]$$
Note that there are $j_0+2$ rows and $j_0+1$ columns in the above matrix. Since this matrix has rank $j_0+1$, it follows that $\alpha: V_{j_0+1} \la -1\ra\rightarrow V_{j_0+2}$ is injective.  

An almost identical analysis shows that the second map $\beta: V_{j_0+2} \rightarrow V_{j_0+1} \la 1 \ra$ in (\ref{eq:6'}) is surjective. Thus the complex in (\ref{eq:6'}) is exact, since the first map is injective, the second is surjective and the dimension of the middle term is the sum of the dimensions of the right and left terms.
\end{proof}

\begin{prop}\label{prop:EFcomp}
The composition $\E \F \1_\l$ is homotopic to the complex in (\ref{eq:FE}) direct sum $\1_\l \otimes_\k V_{\l-1}$. 
\end{prop}
\begin{proof}
The composition $\E \F \1_\l$ is given by 
\begin{equation}\label{eq:3} 
\dots \rightarrow \bigoplus_{i \ge 0} \P^{(i)} \Q^{(1^{\l-1+i})} \P^{(1^{\l-1+i+l})} \Q^{(i+l)} \la l \ra \rightarrow \dots 
\end{equation}
which we can rewrite as 
\begin{equation}\label{eq:4}
\dots \rightarrow \bigoplus_{i,j \ge 0} \P^{(i)} \P^{(1^{\l-1+i+l-j})} \Q^{(1^{\l-1+i-j})} \Q^{(i+l)} \otimes_\k V_j \la l \ra \rightarrow \dots
\end{equation}
Now consider the term $\P^{(i_0,1^{\l-1+i_0+l_0-j_0})} \Q^{(i_0+l_0,1^{\l-1+i_0-j_0})}$. Such a term occurs when 
$$(i,j,k) \text{ equals } (i_0,j_0,l_0), (i_0-1,j_0-2,l_0), (i_0,j_0-1,l_0-1) \text{ or } (i_0-1,j_0-1,l_0+1).$$ 
Thus we end up with a complex 
\begin{equation}\label{eq:5}
\P^{(i_0,1^{\l-1+i_0+l_0-j_0})} \Q^{(i_0+l_0,1^{\l-1+i_0-j_0})} \la l_0 \ra \otimes_k \left( V_{j_0-1} \la -1 \ra \rightarrow V_{j_0} \oplus V_{j_0-2} \rightarrow V_{j_0-1} \la 1 \ra \right).
\end{equation}
The same argument as in Lemma \ref{lem:1} shows that this complex is exact with the following possible exceptions: $i_0 = 0$ or $j_0 = 0$. 

{\bf Case $i_0 = 0$, $j_0 > 0$, $l_0 \ge 2$.} Here we have terms of the form $\P^{(1^{\l-1+l_0-j_0})} \Q^{(l_0,1^{\l-1-j_0})}$. If $l_0 \ge 2$ then one can check again by looking at (\ref{eq:4}) that there are four cases when such a term occurs, namely when 
$$(i,j,l) \text{ equals } (0,j_0,l_0), (0,j_0-1,l_0-1), (0,j_0+1,l_0-1) \text{ or } (0,j_0,l_0-2).$$ 
Thus we end up with a complex 
$$\P^{(1^{\l-1+l_0-j_0})} \Q^{(l_0,1^{\l-1-j_0})} \la l_0-1 \ra \otimes_\k \left( V_{j_0} \la -1 \ra \rightarrow V_{j_0-1} \oplus V_{j_0+1} \rightarrow V_{j_0} \la 1 \ra \right)$$
and one can prove like before that if $j_0 > 0$ then it is exact. Likewise, one obtains an exact sequence if $l_0 \le -2$. 

{\bf Case $i_0=0$, $j_0 > 0$, $l_0 = -1,0,1$.} This time we end up with terms of the form $\P^{(1^a)} \Q^{(1^a)}$ for some $a$. Such terms occur when 
$$(i,j,l) \text{ equals } (0,\l-1-a,0), (1,\l+1-a,0), (0,\l-a,1) \text{ or } (1,\l-a,-1).$$ 
So again, if $a > 0$ we obtain a complex 
$$\P^{(1^a)} \Q^{(1^a)} \otimes_\k \left( V_{\l-a} \la -1 \ra \rightarrow V_{\l-1-a} \oplus V_{\l+1-a} \rightarrow V_{\l-a} \la 1 \ra \right)$$
which, by the same type of argument as in Lemma \ref{lem:1}, is exact. The reason $a=0$ is special is that we require that $j_0 \le \min(\l-1+i, \l-1+i+l)$. So if $a=0$ this condition is violated for three of the terms in the complex above and we end up with only $\1_\l \otimes_\k V_{\l-1}$. 

{\bf Case $j_0=0$.} Finally, if $j_0=0$ then three of the four terms in the complex (\ref{eq:5}) become zero and we end up with 
$$\P^{(i_0,1^{\l-1+i_0+l_0})} \Q^{(i_0+l_0,1^{\l-1+i_0})} \la l_0 \ra \otimes_k \left( 0 \rightarrow V_{0} \oplus 0 \rightarrow 0 \right) \cong \P^{(i_0,1^{\l-1+i_0+l_0})} \Q^{(i_0+l_0,1^{\l-1+i_0})} \la l_0 \ra.$$ 
Putting these terms together leaves us with a complex 
$$\dots \rightarrow \bigoplus_{i \ge 1,-l+1} \P^{(i,1^{\l-1+i+l})} \Q^{(i+l,1^{\l-1+i})} \la l \ra \rightarrow \dots$$
Notice that after replacing $i$ by $i+1$ the terms are the same terms as those in the complex (\ref{eq:FE}). Tracing through the differentials as before it is not hard to see that they are the same as those in (\ref{fig:1}) and (\ref{fig:2}), at least up to a nonzero multiple. Although these nonzero multiples may differ, one can show that the particular choice of multiples does not matter since any two such complexes must be homotopic. 

Finally, one can check that the term $\1_\l \otimes_\k V_{\l-1}$ is a direct summand (i.e. there are no differentials into it or out of it). In fact, by adjunction one can check that 
$$\Hom(\P^{(i+1,1^{\l-1+i})} \Q^{(i,1^{\l+i})} \la -1 \ra, \1_\l \otimes_\k V_{\l-1}) = 0 = \Hom(\1_\l \otimes_\k V_{\l-1}, \P^{(i,1^{\l+i})} \Q^{(i+1,1^{\l-1+i})} \la 1 \ra)$$
if $i \ge 1$. For example, the left hand space is equal to 
\begin{eqnarray*}
\Hom(\Q^{(i,1^{\l+i})} \la -1 \ra, (\P^{(i+1,1^{\l-1+i})})_R \otimes_\k V_{\l-1}) 
&=& \Hom(\Q^{(i,1^{\l+i})}, \Q^{(i+1,1^{\l-1+i})} \otimes_\k V_{\l-1} \la -\l-2i+1 \ra) \\
&=& \oplus_{j=0}^{\l-1} \Hom(\Q^{(i,1^{\l+i})}, \Q^{(i+1,1^{\l-1+i})} \la -2i-2j \ra)
\end{eqnarray*}
which is zero since $2i+2j > 0$ if $j \ge 0$ and $i \ge 1$. 

It follows that $\E \F \1_\l$ is a direct sum of the complex in (\ref{eq:FE}) and $\1_\l \otimes_\k V_{\l-1}$ which is what we needed to prove. 
\end{proof}

\begin{Remark} If $\l \le 0$ the argument is the same. In that case one shows that $\E\F \1_\l$ is the complex 
$$\dots \rightarrow \bigoplus_{i \ge 0,-l} \P^{(-\l+i+2,1^{i+l-1})} \Q^{(-\l+i+l+2, 1^{i-1})} \la l \ra \rightarrow \bigoplus_{i \ge 0,-l-1} \P^{(-\l+i+2,1^{i+l})} \Q^{(-\l+i+l+3, 1^{i-1})} \la l+1 \ra \rightarrow \dots$$
where the differentials are given by cups and caps like the ones in (\ref{fig:1}) and (\ref{fig:2}). 
\end{Remark}

\section{The map $\theta$}

For $\theta \in Y_\k$ we have the induced map 
\begin{equation}\label{eq:theta}
I \theta I : \Q_i^{(n)} \1_{m+n} \P_i^{(n)} \longrightarrow \Q_i^{(n)} \1_{m+n} \P_i^{(n)} \la 2 \ra \cong  \Q_i^{(n)} \P_i^{(n)} \1_m \la 2 \ra.
\end{equation}
Now recall that $\Q_i^{(n)} \P_i^{(n)} \1_m \cong \oplus_{k=0}^n \P_i^{(n-k)} \Q_i^{(n-k)} \otimes_\k V_k \1_m$. 

\begin{lemma}\label{lem:theta}
If $\theta \in \End^2(\1_{m+n})$ is the image of $\alpha_i \in Y_\k$ then the map $I \theta I \in \End^2(\Q_i^{(n)} \1_{m+n} \P_i^{(n)})$ induces an isomorphism between $n$ summands of the form $\1_m \la \cdot \ra$ on either side (in other words between all summands $\1_m \la \cdot \ra$ of the same degree on either side).  

On the other hand, if $\la \theta, \alpha_i \ra = 0$ then $I \theta I \in \End^2(\Q_i^{(n)} \1_{m+n} \P_i^{(n)})$ does not induce any isomorphism between any summands $\1_m \la \cdot \ra$. 
\end{lemma}
\begin{proof}
Suppose $\theta$ is the image of $\alpha_i$. Note that we have a canonical inclusion and projection
$$\iota: \1_m \la n \ra \rightarrow \Q_i^{(n)} \P_i^{(n)} \1_m, \ \  \pi: \Q_i^{(n)} \P_i^{(n)} \1_m \rightarrow \1_m \la n \ra$$
given by a cup and cap. The result follows if we can show that the composition $\pi (I \theta I)^n \iota: \1_m \rightarrow \1_m$ is (some nonzero multiple of) the identity. 

Diagrammatically, $\pi (I \theta I)^n \iota$ is given by the picture
$$
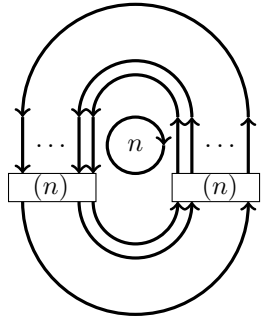
\begin{tikzpicture}[scale=.75][>=stealth]
\draw[<-](1,1) arc (0:360:0.5cm) [very thick];
\draw[->](-.25,0) arc (180:360:0.75cm) [very thick];
\draw[->](-.5,0) arc (180:360:1cm) [very thick];

\draw (2,1) node {$\hdots$};

\draw[->](-1.5,0) arc (180:360:2cm) [very thick];

\draw  [shift={+(3,0)}](-.3,0) rectangle (-1.85,.5);
\draw  (-1,.25) node {$(n)$};
\draw  (-.2,0) rectangle (-1.75,.5);
\draw  (2,.25) node {$(n)$};
\draw  (0.5,1) node {$n$};

\draw[->] (1.5,.5)--(1.5,1.5)[very thick];
\draw[->] (1.25,.5)--(1.25,1.5)[very thick];
\draw[->] (2.5,.5)--(2.5,1.5)[very thick];

\draw[<-] (-.5,.5)--(-.5,1.5)[very thick];
\draw[<-] (-.25,.5)--(-.25,1.5)[very thick];
\draw[<-] (-1.5,.5)--(-1.5,1.5)[very thick];

\draw (-1,1) node {$\hdots$};

\draw[->](1.25,1.5) arc (0:180:0.75cm) [very thick];
\draw[->](1.5,1.5) arc (0:180:1cm) [very thick];
\draw[->](2.5,1.5) arc (0:180:2cm) [very thick];

\end{tikzpicture}
$$ 
where every strand is labeled $i$ and the $n$ in the middle of the center circle indicates that there is a union of $n$ disjoint circles corresponding to $(I \theta I)^n$. Now, slide each of these circles from the inside towards the outside using that
\begin{equation}\label{eq:slide}
\begin{tikzpicture}[scale=.75][>=stealth]
\draw[<-](0,0) arc (0:360:0.5cm) [very thick];
\draw[->](.25,-1) -- (.25,1) [very thick];
\draw  (1,0) node {$=$};

\draw[<-](3,0) arc (0:360:0.5cm) [very thick];
\draw[->](1.5,-1) -- (1.5,1) [very thick];
\draw  (4,0) node {$+ \ \ 2$};

\draw[->](5,-1) -- (5,1) [very thick];
\filldraw [blue] (5,0) circle (3pt);
\end{tikzpicture}
\end{equation}
This fact is an easy consequence of the relations among 2-morphisms in a Heisenberg 2-representation. After moving all the circles to the outside we end up with a bunch of circles and solid dots. Using that 
\begin{itemize}
\item a dot squares to zero 
\item a counter-clockwise circle with no dots is zero
\end{itemize}
this simplifies to give some nonzero multiple of 
$$
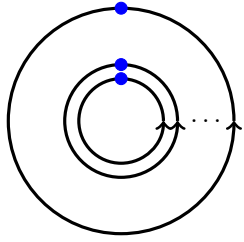
\begin{tikzpicture}[scale=.75][>=stealth]
\draw[->](.25,0) arc (0:360:0.75cm) [very thick];
\draw[->](.5,0) arc (0:360:1cm) [very thick];

\draw (1,0) node {$\hdots$};

\draw[->](1.5,0) arc (0:360:2cm) [very thick];

\filldraw [blue] (-.5,1) circle (3pt);
\filldraw [blue] (-.5,.75) circle (3pt);
\filldraw [blue] (-.5,2) circle (3pt);
\end{tikzpicture}
$$ 
where there are $n$ circles. The result follows since each of these circles evaluates to the identity. 

Finally, suppose $\la \theta, \alpha_i \ra = 0$ where $\theta$ is some linear combination of clockwise circles. Then the analogue of (\ref{eq:slide}) states that this linear combination slides through upper pointing strands. This means that 
$$(I \theta II) = (III \theta) \in \End^2(\Q_i^{(n)} \1_{m+n} \P_i^{(n)} \1_m)$$ 
and hence $I \theta I \in \End^2(\Q_i^{(n)} \1_{m+n} \P_i^{(n)})$ cannot induce an isomorphism between any summands $\1_m$. 
\end{proof}

\begin{prop}\label{prop:theta}
For $\theta \in Y_\k$ the maps
\begin{align*}
I \theta I: \F_{i,-m} \1_\l \E_{i,m} \rightarrow \F_{i,-m} \1_\l \E_{i,m} \la 2 \ra  &\text{ if } \l_i + m \le 2 \text{ and } \\
I \theta I: \E_{i,m} \1_\l \F_{i,-m} \rightarrow  \E_{i,m} \1_\l \F_{i,-m} \la 2 \ra &\text{ if } \l_i + m \ge -2
\end{align*}
induce an isomorphism between all (resp. none of the) summands $\1_\l \la \cdot \ra$ of the same degree on either side if $\la \theta, \alpha_i \ra \ne 0$ (resp. $\la \theta, \alpha_i \ra = 0$).
\end{prop}
\begin{proof}
We consider the case $\l + m \ge 0$ (the other case is very similar). Examining the proof of Proposition \ref{prop:EFcomp} shows that the direct sum $\1_\l \otimes_\k V_{\l-1+m}$ inside $\E_{i,m} \1_\l \F_{i,-m}$ comes from the term $\Q^{(\l-3+m)} \P^{(\l-3+m)} \cong \1_\l \otimes_\k V_{\l-3} \bigoplus \A$ where the precise form of $\A$ is not important. The map $I \theta I$ then restricts to the endomorphism $I \theta I \in \End^2(\Q^{(\l-3+m)} \1_n \P^{(\l-3+m)})$ for some $n \in \Z$. The result now follows by applying Lemma \ref{lem:theta}.
\end{proof}


\section{The commutation relation of $\E_i$ and $\E_{i,1}$}\label{sec:commE}

In general $\E_{i,m}$ and $\E_{i,n}$ do not commute. But in the simplest case they do commute up to a shift. 

\begin{prop}\label{prop:Escommute}
For any $n \in \Z$ we have 
\begin{equation*}
\E_{i,n-1} \E_{i,n} \1_\l \cong \E_{i,n} \E_{i,n-1} \la -2 \ra \1_\l \text{ and } 
\F_{i,n-1} \F_{i,n} \1_\l \cong \F_{i,n} \F_{i,n-1} \la 2 \ra \1_\l. 
\end{equation*}
\end{prop}

We use the rest of this section to prove the first relation above (the second relation is obtained by taking the adjoint of the first relation). The idea of the proof is to show that both sides are homotopic to a complex
\begin{equation}\label{eq:EEtcpx}
\left[ \dots \rightarrow 
\begin{matrix} \P^{(2)} \Q^{(2^{a+1},1^3)} \la -2 \ra \\ \oplus \P^{(1^2)} \Q^{(2^{a+2},1)} \la -2 \ra \\ \oplus \P^{(2)} \Q^{(2^{a+2},1)} \la -2 \ra \end{matrix} 
\rightarrow \begin{matrix} \P \Q^{(2^{a+1},1^2)} \la -1 \ra \\ \oplus \P \Q^{(2^{a+2})} \la -1 \ra \end{matrix}
\rightarrow \Q^{(2^{a+1},1)} \right] [1] \la -2 \ra
\end{equation}
where $a := \l+n \ge 0$. In the equation above, the term in cohomological degree $(-l+1)$ is 
$$
\bigoplus_{\begin{matrix} \scriptstyle k_1+k_2=l-1 \\ \scriptstyle k_1 \ge k_2 
\end{matrix}}\P^{(k_1,k_2-1)} \Q^{(2^{a+1+k_2}, 1^{k_1-k_2})} \la -l \ra
\bigoplus_{\begin{matrix} \scriptstyle k_1+k_2=l-1 \\ \scriptstyle k_1 \ge k_2+
1 \end{matrix}} 
\P^{(k_1-1,k_2)} \Q^{(2^{a+1+k_2}, 1^{k_1-k_2})} \la -l \ra. 
$$
Using Lemma \ref{lem:differentials} one checks that there are three possible nonzero maps out of each summand above, all of which are given by a cap:
\begin{eqnarray*}
&f_1& : \P^{(k_1,k_2-1)} \Q^{(2^{a+1+k_2}, 1^{k_1-k_2})} \longrightarrow \P^{(k_1-1,k_2-1)} \Q^{(2^{a+1+k_2}, 1^{k_1-k_2-1})} \\
&f_2& : \P^{(k_1,k_2-1)} \Q^{(2^{a+1+k_2}, 1^{k_1-k_2})} \longrightarrow \P^{(k_1-1,k_2-1)} \Q^{(2^{a+k_2}, 1^{k_1-k_2+1})} \\
&f_3& : \P^{(k_1,k_2-1)} \Q^{(2^{a+1+k_2}, 1^{k_1-k_2})} \longrightarrow \P^{(k_1,k_2-2)} \Q^{(2^{a+k_2}, 1^{k_1-k_2+1})} \\
&g_1& : \P^{(k_1-1,k_2)} \Q^{(2^{a+1+k_2}, 1^{k_1-k_2})} \longrightarrow \P^{(k_1-1,k_2-1)} \Q^{(2^{a+1+k_2}, 1^{k_1-k_2-1})} \\
&g_2& : \P^{(k_1-1,k_2)} \Q^{(2^{a+1+k_2}, 1^{k_1-k_2})} \longrightarrow \P^{(k_1-1,k_2-1)} \Q^{(2^{a+k_2}, 1^{k_1-k_2+1})} \\
&g_3& : \P^{(k_1-1,k_2)} \Q^{(2^{a+1+k_2}, 1^{k_1-k_2})} \longrightarrow \P^{(k_1-2,k_2)} \Q^{(2^{a+1+k_2}, 1^{k_1-k_2-1})}.
\end{eqnarray*}
The differential in (\ref{eq:EEtcpx}) is of the form 
\begin{equation}\label{eq:diff}
\partial = a_1f_1+a_2f_2+a_3f_3+b_1g_1+b_2g_2+b_3g_3 \ \ \text{ for some } \ \ a_1,a_2,a_3,b_1,b_2,b_3 \in \k.
\end{equation}
Using the following lemma we will check that $a_2=0=b_2$ and $a_1,a_3,b_1,b_3 \in \k^\times$ which determines the complex uniquely up to homotopy. 

\begin{lemma}\label{lem:EEtcpx} 
Consider a complex as in (\ref{eq:EEtcpx}) with differential $\partial$ as in (\ref{eq:diff}). If $b_1,b_3 \in \k^\times$ then $a_1,a_3 \in \k^\times$ while $a_2=0=b_2$. Moreover, any two such complexes with this property are homotopy equivalent. 
\end{lemma}
\begin{proof}
Suppose we have a complex as in (\ref{eq:EEtcpx}) where $b_1,b_3 \in \k^\times$. 

The first step is to show that $a_2=0=b_2$. To see this consider the following three compositions
$$
\xymatrix{
& \P^{(k_1-1,k_2-1)} \Q^{(2^{a+1+k_2}, 1^{k_1-k_2-1})} \la 1 \ra \ar[dr]^{a_2f_2} & \\
\P^{(k_1-1,k_2)} \Q^{(2^{a+1+k_2},1^{k_1-k_2})} \ar[ur]^{b'_1g_1} \ar[r]^{b'_2g_2} \ar[dr]^{b'_3g_3} & 
\P^{(k_1-1,k_2-1)} \Q^{(2^{a+k_2}, 1^{k_1-k_2+1})} \la 1 \ra \ar[r]^{b_3g_3} & 
\hspace{-.1cm} \P^{(k_1-2,k_2-1)} \Q^{(2^{a+k_2}, 1^{k_1-k_2})} \la 2 \ra \\
& \P^{(k_1-2,k_2)} \Q^{(2^{a+1+k_2}, 1^{k_1-k_2-1})} \la 1 \ra \ar[ur]^{b_2g_2} & 
}$$
By induction we can assume that $b_2 = 0$. Looking at the top two compositions, this means that $a_2f_2 \circ b'_1g_1 + b_3g_3 \circ b'_2g_2 = 0$. But one can check that the compositions $f_2 \circ g_1$ and $g_3 \circ g_2$ are linearly independent since they span the two dimensional vector space
$$\Hom(\P^{(k_1-1,k_2)} \Q^{(2^{a+1+k_2}, 1^{k_1-k_2})}, \P^{(k_1-2,k_2-1)} \Q^{(2^{a+k_2}, 1^{k_1-k_2})} \la 2 \ra) \cong \k^2.$$
This space is 2-dimensional since the two boxes we add to go from partition $(k_1-2,k_2-1)$ to $(k_1-1,k_2)$ and from $(2^{a+k_2},1^{k_1-k_2})$ to $(2^{a+1+k_2}, 1^{k_1-k_2})$ occur in different columns and rows. This implies that $a_2=0$ and $b'_2=0$. Thus, by induction, we always have $a_2=0=b_2$.

Next we show that $a_1,a_3 \in \k^\times$. First consider the following two compositions
$$
\xymatrix{
& \P^{(k_1-1,k_2-1)} \Q^{(2^{a+1+k_2}, 1^{k_1-k_2-1})} \la 1 \ra \ar[rd]^{a_1f_1} & \\
\P^{(k_1-1,k_2)} \Q^{(2^{a+1+k_2}, 1^{k_1-k_2})} \ar[ru]^{b'_1g_1} \ar[rd]^{b'_3g_3} & & \hspace{-1cm} \P^{(k_1-2,k_2-1)} \Q^{(2^{a+1+k_2}, 1^{k_1-k_2-2})} \la 2 \ra \\
& \P^{(k_1-2,k_2)} \Q^{(2^{a+1+k_2}, 1^{k_1-k_2-1})} \la 1 \ra \ar[ru]^{b_1g_1} &
  }
$$
These two compositions both span the one dimensional vector space 
$$\Hom (\P^{(k_1-1,k_2)} \Q^{(2^{a+1+k_2}, 1^{k_1-k_2})}, \P^{(k_1-2,k_2-1)} \Q^{(2^{a+1+k_2}, 1^{k_1-k_2-2})} \la 2 \ra) \cong \k.$$
Since $b_1,b'_3$ are both nonzero this means $a_1 \ne 0$. 

Similarly, we can consider the two compositions 
$$
\xymatrix{
& \P^{(k_1-1,k_2-1)} \Q^{(2^{a+1+k_2}, 1^{k_1-k_2-1})} \la 1 \ra \ar[dr]^{a_3f_3} & \\
\P^{(k_1,k_2-1)} \Q^{(2^{a+1+k_2}, 1^{k_1-k_2})} \ar[ur]^{a'_1f_1} \ar[dr]^{a'_3f_3} & & \hspace{-1cm} \P^{(k_1-1,k_2-2)} \Q^{(2^{a+k_2}, 1^{k_1-k_2})} \la 2 \ra. \\
& \P^{(k_1,k_2-2)} \Q^{(2^{a+k_2}, 1^{k_1-k_2+1})} \la 1 \ra \ar[ur]^{a_1f_1} & 
}$$
These two compositions both span the one dimensional vector space 
$$\Hom(\P^{(k_1,k_2-1)} \Q^{(2^{a+1+k_2}, 1^{k_1-k_2})}, \P^{(k_1-1,k_2-2)} \Q^{(2^{a+k_2}, 1^{k_1-k_2})} \la 2 \ra) \cong \k.$$
We know that $a_1,a'_1$ are nonzero and by induction we can assume $a_3 \ne 0$. This implies that $a'_3 \ne 0$ and hence, by induction, all $a_3$ are nonzero. 

Finally we show that any two such complexes are homotopy equivalent to each other. The idea is very simple. Suppose you have a complex 
\begin{eqnarray}\label{eq:littlecpx}
A \overset{\alpha_1}{\underset{\alpha_2}{\rightrightarrows}} \begin{matrix} B_1 \\ B_2 \end{matrix} \overset{\beta_1}{\underset{\beta_2}{\rightrightarrows}} C
\end{eqnarray}
where $\Hom(A,B_1) \cong \Hom(A,B_2) \cong \Hom(B_1,C) \cong \Hom(B_2,C) \cong \Hom(A,C) \cong \k$ are spanned by $\alpha_1,\alpha_2,\beta_1,\beta_2$ and $\beta_1 \circ \alpha_1 = - \beta_2 \circ \alpha_2$ respectively. Then any other complex where the four maps above are nonzero is homotopic to it via a map which acts by certain multiples of the identity on $A,B_1,B_2$ and $C$. This is a simple exercise which we leave to the reader. 

If we now look at (\ref{eq:EEtcpx}) and recall that each differential is made up of maps $f_1,f_3$ or $g_1,g_3$ it follows that (\ref{eq:EEtcpx}) is made up of little complexes like (\ref{eq:littlecpx}). Thus starting from the far right, we can repeatedly apply the homotopy above to show that any two such complexes are homotopy equivalent. 
\end{proof}

{\bf Computation of $\E_{i,n-1} \E_{i,n} \1_\l$.} This composition is isomorphic to 
$$\left[ \dots \rightarrow \P^{(2)} \Q^{(1^{a+4})} \la -2 \ra \rightarrow \P \Q^{(1^{a+3})} \la -1 \ra \rightarrow \Q^{(1^{a+2})} \right] \left[ \dots \rightarrow \P^{(2)} \Q^{(1^{a+3})} \la -2 \ra \rightarrow \P \Q^{(1^{a+2})} \la -1 \ra \rightarrow \Q^{(1^{a+1})} \right]$$
which means that the term in cohomological degree $-l$ is 
\begin{eqnarray*}
& & \bigoplus_{k_1+k_2=l} \P^{(k_2)} \Q^{(1^{a+2+k_2})} \P^{(k_1)} \Q^{(1^{a+1+k_1})} \la -l \ra \\ 
&\cong& \bigoplus_{k_1+k_2=l} \P^{(k_2)} \left( \P^{(k_1)} \Q^{(1^{a+2+k_2})} \oplus \P^{(k_1-1)} \Q^{(1^{a+1+k_2})} \otimes_\k V_1 \oplus \P^{(k_1-2)} \Q^{(1^{a+k_2})} \right) \Q^{(1^{a+1+k_1})} \la -l \ra.
\end{eqnarray*}
If we collect terms with a shift of $\la -l \ra$ they must occur in cohomological degrees $(-l-1), -l$ and $(-l+1)$. These are 
\begin{equation}\label{eq:A}
\bigoplus_{k_1+k_2=l} \P^{(k_2)} \P^{(k_1-2)} \Q^{(1^{a+k_2})} \Q^{(1^{a+1+k_1})} \longrightarrow \bigoplus_{k_1+k_2=l-1} \P^{(k_2)} \P^{(k_1-1)} \Q^{(1^{a+1+k_2})} \Q^{(1^{a+1+k_1})}
\end{equation}
in cohomological degree $-l$ and $(-l+1)$ and
\begin{equation}\label{eq:B}
\bigoplus_{k_1+k_2=l+1} \P^{(k_2)} \P^{(k_1-1)} \Q^{(1^{a+1+k_2})} \Q^{(1^{a+1+k_1})} \longrightarrow \bigoplus_{k_1+k_2=l} \P^{(k_2)} \P^{(k_1)} \Q^{(1^{a+2+k_2})} \Q^{(1^{a+1+k_1})} 
\end{equation}
occuring in cohomological degree $(-l-1)$ and $-l$. 

There are two parts to the differentials in (\ref{eq:A}). To describe the first we rewrite (\ref{eq:A}) as 
\begin{equation}\label{eq:A'}
\bigoplus_{k_1+k_2=l-1} \P^{(k_2)} \P^{(k_1-1)} \Q^{(1^{a+k_2})} \Q^{(1^{a+2+k_1})} \longrightarrow \bigoplus_{k_1+k_2=l-1} \P^{(k_2)} \P^{(k_1-1)} \Q^{(1^{a+1+k_2})} \Q^{(1^{a+1+k_1})}.
\end{equation}
Then the first part of the differential is given by the composition

$$
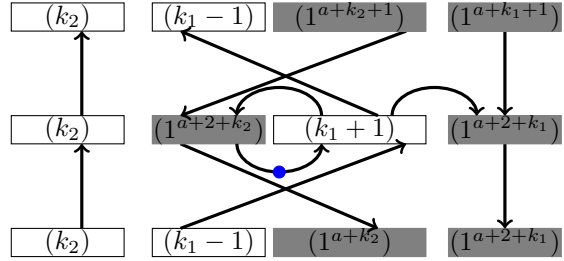
\begin{tikzpicture}[scale=.75][>=stealth]

\draw (-.25,1) rectangle (1.75,1.5);
\draw (.75,1.25) node {$(k_2)$};
\draw (2.25,1) rectangle (4.25,1.5);
\draw (3.25,1.25) node {$(k_1-1)$};
\filldraw [gray](4.4,1) rectangle (7.1,1.5);
\draw (5.75,1.25) node {$(1^{a+k_2})$};
\filldraw[gray] (7.5,1) rectangle (9.5,1.5);
\draw (8.5,1.25) node {$(1^{a + 2+k_1})$};

\draw (1,1.5) -- (1,3)[->][very thick];
\draw (8.5,1.5) -- (8.5,3)[<-][very thick];
\draw (2.75,1.5) -- (6.75,3) [->] [very thick];
\draw (6.25,1.5) -- (2.75,3) [<-] [very thick];

\draw [shift = {+(0,2)}](-.25,1) rectangle (1.75,1.5);
\draw [shift = {+(0,2)}](.75,1.25) node {$(k_2)$};
\filldraw[gray] [shift = {+(0,2)}](2.25,1) rectangle (4.25,1.5);
\draw [shift = {+(0,2)}](3.25,1.25) node {$(1^{a+2+k_2})$};
\draw [shift = {+(0,2)}](4.4,1) rectangle (7.1,1.5);
\draw [shift = {+(0,2)}](5.75,1.25) node {$(k_1+1)$};
\filldraw[gray] [shift = {+(0,2)}](7.5,1) rectangle (9.5,1.5);
\draw [shift = {+(0,2)}](8.5,1.25) node  {$(1^{a + 2+k_1})$};

\draw (3.75,3) arc (180:360:.75cm and .5cm)[->][very thick];
\filldraw [blue] (4.5,2.5) circle (3pt);

\draw  [shift = {+(0,2)}](1,1.5) -- (1,3)[->][very thick];
\draw  [shift = {+(0,2)}](8.5,1.5) -- (8.5,3)[<-][very thick];
\draw  [shift = {+(0,2)}](2.75,1.5) -- (6.75,3) [<-] [very thick];
\draw  [shift = {+(0,2)}](6.25,1.5) -- (2.75,3) [->] [very thick];

\draw (3.75,3.5) arc (180:0:.75cm and .5cm)[<-][very thick];
\draw (6.5,3.5) arc (180:0:.75cm and .5cm)[->][very thick];

\draw[shift = {+(0,4)}](-.25,1) rectangle (1.75,1.5);
\draw [shift = {+(0,4)}](.75,1.25) node {$(k_2)$};
\draw [shift = {+(0,4)}](2.25,1) rectangle (4.25,1.5);
\draw [shift = {+(0,4)}](3.25,1.25) node {$(k_1-1)$};
\filldraw [gray][shift = {+(0,4)}](4.4,1) rectangle (7.1,1.5);
\draw [shift = {+(0,4)}](5.75,1.25) node {$(1^{a+k_2+1})$};
\filldraw[gray] [shift = {+(0,4)}](7.5,1) rectangle (9.5,1.5);
\draw [shift = {+(0,4)}](8.5,1.25) node {$(1^{a+k_1+1})$};

\end{tikzpicture}
$$
which is equal to a scalar multiple of
$$
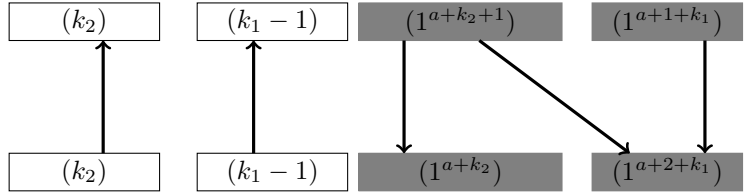
\begin{tikzpicture}\label{figure1}[scale=.75][>=stealth]

\draw (-.25,1) rectangle (1.75,1.5);
\draw (.75,1.25) node {$(k_2)$};
\draw (2.25,1) rectangle (4.25,1.5);
\draw (3.25,1.25) node {$(k_1-1)$};
\filldraw [gray](4.4,1) rectangle (7.1,1.5);
\draw (5.75,1.25) node {$(1^{a+k_2})$};
\filldraw[gray] (7.5,1) rectangle (9.5,1.5);
\draw (8.5,1.25) node {$(1^{a + 2+k_1})$};

\draw[shift = {+(0,2)}](-.25,1) rectangle (1.75,1.5);
\draw [shift = {+(0,2)}](.75,1.25) node {$(k_2)$};
\draw [shift = {+(0,2)}](2.25,1) rectangle (4.25,1.5);
\draw [shift = {+(0,2)}](3.25,1.25) node {$(k_1-1)$};
\filldraw [gray][shift = {+(0,2)}](4.4,1) rectangle (7.1,1.5);
\draw [shift = {+(0,2)}](5.75,1.25) node {$(1^{a+k_2+1})$};
\filldraw[gray] [shift = {+(0,2)}](7.5,1) rectangle (9.5,1.5);
\draw [shift = {+(0,2)}](8.5,1.25) node {$(1^{a+1+k_1})$};

\draw (1,1.5) -- (1,3) [->][very thick];
\draw (3,1.5) -- (3,3) [->][very thick];
\draw (5,1.5) -- (5,3) [<-][very thick];
\draw (9,1.5) -- (9,3) [<-][very thick];
\draw (8,1.5) -- (6,3) [<-][very thick];

\end{tikzpicture}
$$

To see the other part of the differential we can also rewrite (\ref{eq:A}) as 
\begin{equation}\label{eq:A''}
\bigoplus_{k_1+k_2=l-1} \P^{(k_2)} \P^{(k_1-1)} \Q^{(1^{a+k_2})} \Q^{(1^{a+2+k_1})} \longrightarrow \bigoplus_{k_1+k_2=l-1} \P^{(k_2-1)} \P^{(k_1)} \Q^{(1^{a+k_2})} \Q^{(1^{a+2+k_1})}
\end{equation}
and then there is a similar map given by 
$$
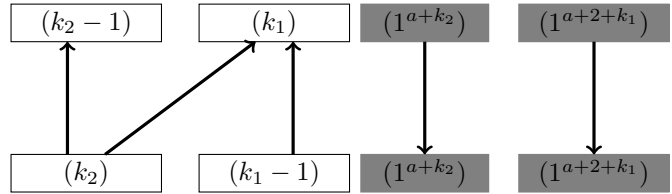
\begin{tikzpicture}\label{figure2}[scale=.75]
\draw (-.25,1) rectangle (1.75,1.5);
\draw (.75,1.25) node {$(k_2)$};
\draw (2.25,1) rectangle (4.25,1.5);
\draw (3.25,1.25) node {$(k_1-1)$};
\filldraw [gray](4.4,1) rectangle (6.1,1.5);
\draw (5.25,1.25) node {$(1^{a+k_2})$};
\filldraw[gray] (6.5,1) rectangle (8.5,1.5);
\draw (7.5,1.25) node {$(1^{a + 2+k_1})$};

\draw[shift = {+(0,2)}](-.25,1) rectangle (1.75,1.5);
\draw [shift = {+(0,2)}](.75,1.25) node {$(k_2-1)$};
\draw [shift = {+(0,2)}](2.25,1) rectangle (4.25,1.5);
\draw [shift = {+(0,2)}](3.25,1.25) node {$(k_1)$};
\filldraw [gray][shift = {+(0,2)}](4.4,1) rectangle (6.1,1.5);
\draw [shift = {+(0,2)}](5.25,1.25) node {$(1^{a+k_2})$};
\filldraw[gray] [shift = {+(-1,2)}](7.5,1) rectangle (9.5,1.5);
\draw [shift = {+(-1,2)}](8.5,1.25) node {$(1^{a+2+k_1})$};

\draw (.5,1.5) -- (.5,3) [->][very thick];
\draw (3.5,1.5) -- (3.5,3) [->][very thick];
\draw (5.25,1.5) -- (5.25,3) [<-][very thick];
\draw (7.5,1.5) -- (7.5,3) [<-][very thick];

\draw (1,1.5) -- (3,3) [->][very thick];

\end{tikzpicture}
$$

{\bf Claim.} The map in (\ref{eq:A}) is injective. 

Let us consider a general indecomposable term on the left side of (\ref{eq:A'}) or (\ref{eq:A''}). Using Proposition \ref{prop:rels1} such a term is of the form $\P^{(m,n)} \Q^{(2^{a+m'},1^{a+n'})}$ where $m \le n$ and $m+n=l-1=2m'+n'$. There is one such summand for each $(k_1,k_2)$ where 
$$m \le \min(k_1-1,k_2) \text{ and } m' \le \min(k_1+2,k_2).$$
On the right hand side of (\ref{eq:A'}) there is one such summand corresponding to each pair $(k_1,k_2)$ where 
$$m \le \min(k_1-1,k_2) \text{ and } m' \le \min(k_1+1,k_2+1)$$
and the map in figure \ref{figure1} induces an isomorphism between any two such summands corresponding to the same pair $(k_1,k_2)$. Likewise, on the right side of (\ref{eq:A''}) there is one such summand corresponding to each $(k_1,k_2)$ where 
$$m \le \min(k_1,k_2-1) \text{ and } m' \le \min(k_1+2,k_2)$$
and the map in \ref{figure2} also induces an isomorphism between summands corresponding to the same pair $(k_1,k_2)$. 

Using the inequalities above and looking at (\ref{eq:A'}), we see that the map in (\ref{figure1}) is injective on summands $\P^{(m,n)} \Q^{(2^{a+m'},1^{a+n'})}$ unless $k_1+2 \le k_2$ in which case there is precisely one term, namely the one corresponding to $k_1+2=m'$ on the left, which maps to zero. Notice that for such a term to exist on the left side of (\ref{eq:A'}) we must also have $m \le \min(m'-3,l+1-m')$ (and in particular $m \le m'-3$). 

On the other hand, looking at (\ref{eq:A''}), we see that the map in (\ref{figure2}) is an isomorphism between all summands $\P^{(m,n)} \Q^{(2^{a+m'},1^{a+n'})}$ unless $k_2 \le k_1-1$ in which case there is precisely one term mapped to zero, namely the one corresponding to $k_2=m$. This time such a term exists on the left hand side of (\ref{eq:A''}) only if $m' \le \min(l+1-m,m)$ (and in particular $m' \le m$). 

Since we cannot have both $m \le m'-3$ and $m' \le m$ either (\ref{figure1}) or (\ref{figure2}) is injective on all summands of the form $\P^{(m,n)} \Q^{(2^{a+m'},1^{a+n'})}$. The map in (\ref{eq:A}) is upper triangular and hence also injective. 

Now we need to figure out what terms remain on the right hand side of (\ref{eq:A}) after cancelling terms. We can replace (\ref{eq:A'}) by 
\begin{equation}\label{eq:A'''}
\bigoplus_{{\begin{matrix} \scriptstyle k_1+k_2=l-1 \\ \scriptstyle k_2 \ge k_1+2 \end{matrix}}} \P^{(k_2)} \P^{(k_1-1)} \Q^{(2^{a+2+k_1}, 1^{k_2-k_1-2})} \longrightarrow \bigoplus_{{\begin{matrix} \scriptstyle k_1+k_2=l-1 \\ \scriptstyle k_2 \le k_1 \end{matrix}}} \P^{(k_2)} \P^{(k_1-1)} \Q^{(2^{a+1+k_2}, 1^{k_1-k_2})}.
\end{equation}
since, using proposition ref{prop:rels1}, we have 
\begin{eqnarray*}
\Q^{(1^{a+k_2})} \Q^{(1^{a+2+k_1})} &\cong& \Q^{(1^{a+1+k_2})} \Q^{(1^{a+1+k_1})} \oplus \Q^{(2^{a+2+k_1},1^{k_2-k_1-2})} \text{ if } k_2 \ge k_1+2 \\
\Q^{(1^{a+1+k_2})} \Q^{(1^{a+1+k_1})} &\cong& \Q^{(1^{a+k_2})} \Q^{(1^{a+2+k_1})} \oplus \Q^{(2^{a+1+k_2}, 1^{k_1-k_2})} \text{ if } k_2 \le k_1 \\
\Q^{(1^{a+k_2})} \Q^{(1^{a+2+k_1})} &\cong& \Q^{(1^{a+1+k_2})} \Q^{(1^{a+1+k_1})} \text{ if } k_2 = k_1+1.
\end{eqnarray*}
Now, switching the roles of $k_1$ and $k_2$ on the left hand side and replacing the new $k_1$ by $k_1+1$ and the new $k_2$ by $k_2-1$ we get that (\ref{eq:A'''}) is equivalent to 
\begin{equation}\label{eq:A''''}
\bigoplus_{{\begin{matrix} \scriptstyle k_1+k_2=l-1 \\ \scriptstyle k_1 \ge k_2 \end{matrix}}} \P^{(k_1+1)} \P^{(k_2-2)} \Q^{(2^{a+1+k_2}, 1^{k_1-k_2})} \longrightarrow \bigoplus_{{\begin{matrix} \scriptstyle k_1+k_2=l-1 \\ \scriptstyle k_1 \ge k_2 \end{matrix}}} \P^{(k_2)} \P^{(k_1-1)} \Q^{(2^{a+1+k_2}, 1^{k_1-k_2})}.
\end{equation}
Again, we can cancel terms using that 
\begin{eqnarray*}
\P^{(k_1-1)} \P^{(k_2)} &\cong& \P^{(k_1+1)} \P^{(k_2-2)} \oplus \P^{(k_1,k_2-1)} \oplus \P^{(k_1-1,k_2)} \text{ if } k_1 \ge k_2+1 \\
\P^{(k_1-1)} \P^{(k_2)} &\cong& \P^{(k_1+1)} \P^{(k_2-2)} \oplus \P^{(k_1,k_2-1)} \text{ if }  k_1 = k_2 
\end{eqnarray*}
to obtain
$$
\bigoplus_{\begin{matrix} \scriptstyle k_1+k_2=l-1 \\ \scriptstyle k_1 \ge k_2 \end{matrix}} 
\P^{(k_1,k_2-1)} \Q^{(2^{a+1+k_2}, 1^{k_1-k_2})} \la -l \ra 
\bigoplus_{\begin{matrix} \scriptstyle k_1+k_2=l-1 \\ \scriptstyle k_1 \ge k_2+1 \end{matrix}}
 \P^{(k_1-1,k_2)} \Q^{(2^{a+1+k_2}, 1^{k_1-k_2})} \la -l \ra 
$$
in cohomological degree $(-l+1)$ (where we have added back the $\la - l \ra$ shift). Notice that these are the same as the terms in the complex (\ref{eq:EEtcpx}). 

Now we also need to examine (\ref{eq:B}). Fortunately, things are much simpler here. We rewrite (\ref{eq:B}) as 
\begin{equation}\label{eq:B'}
\bigoplus_{k_1+k_2=l} \P^{(k_2)} \P^{(k_1)} \Q^{(1^{a+1+k_2})} \Q^{(1^{a+2+k_1})} \longrightarrow \bigoplus_{k_1+k_2=l} \P^{(k_2)} \P^{(k_1)} \Q^{(1^{a+2+k_2})} \Q^{(1^{a+1+k_1})}.
\end{equation}
Then the part of the differential which looks like that in (\ref{figure1}) induces an isomorphism between the two sides. The total differential is upper triangular and hence also induces an isomorphism. Thus all the terms in (\ref{eq:B}) cancel out. 

{\bf The differentials.} Finally, we need to compute the differentials. In light of Lemma \ref{lem:EEtcpx} it suffices to show that the differentials of the form $g_1$ and $g_3$ are nonzero. 

This is trickier than it looks since the cancellation lemma was applied many times. Let us consider the map $g_3$. In the original complex for $\E_{i,n-1} \E_{i,n} \1_\l$ we see this map show up as the composition 

$$
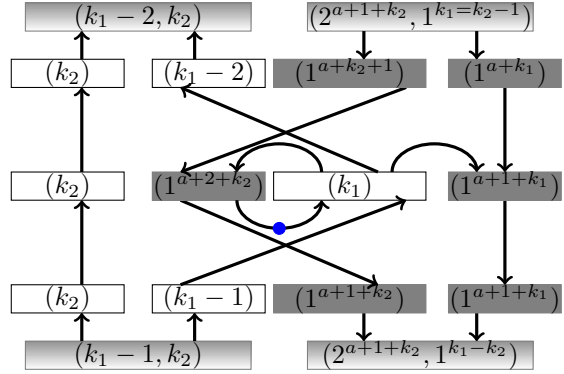
\begin{tikzpicture}[scale=.75][>=stealth]
\shadedraw[gray]  (0,0) rectangle (4,.5);
\draw (2,.25) node {$(k_1-1,k_2)$};
\shadedraw[gray]  (5,0) rectangle (9,.5);
\draw (7,.25) node {$(2^{a+1+k_2},1^{k_1-k_2})$};
\draw (1,.5) -- (1,1) [->][very thick];
\draw (3,.5) -- (3,1) [->][very thick];
\draw (6,.5) -- (6,1) [<-][very thick];
\draw (8,.5) -- (8,1) [<-][very thick];

\draw (-.25,1) rectangle (1.75,1.5);
\draw (.75,1.25) node {$(k_2)$};
\draw (2.25,1) rectangle (4.25,1.5);
\draw (3.25,1.25) node {$(k_1-1)$};
\filldraw [gray](4.4,1) rectangle (7.1,1.5);
\draw (5.75,1.25) node {$(1^{a+1+k_2})$};
\filldraw[gray] (7.5,1) rectangle (9.5,1.5);
\draw (8.5,1.25) node {$(1^{a + 1+k_1})$};

\draw (1,1.5) -- (1,3)[->][very thick];
\draw (8.5,1.5) -- (8.5,3)[<-][very thick];
\draw (2.75,1.5) -- (6.75,3) [->] [very thick];
\draw (6.25,1.5) -- (2.75,3) [<-] [very thick];

\draw [shift = {+(0,2)}](-.25,1) rectangle (1.75,1.5);
\draw [shift = {+(0,2)}](.75,1.25) node {$(k_2)$};
\filldraw[gray] [shift = {+(0,2)}](2.25,1) rectangle (4.25,1.5);
\draw [shift = {+(0,2)}](3.25,1.25) node {$(1^{a+2+k_2})$};
\draw [shift = {+(0,2)}](4.4,1) rectangle (7.1,1.5);
\draw [shift = {+(0,2)}](5.75,1.25) node {$(k_1)$};
\filldraw[gray] [shift = {+(0,2)}](7.5,1) rectangle (9.5,1.5);
\draw [shift = {+(0,2)}](8.5,1.25) node  {$(1^{a + 1+k_1})$};

\draw (3.75,3) arc (180:360:.75cm and .5cm)[->][very thick];
\filldraw [blue] (4.5,2.5) circle (3pt);

\draw  [shift = {+(0,2)}](1,1.5) -- (1,3)[->][very thick];
\draw  [shift = {+(0,2)}](8.5,1.5) -- (8.5,3)[<-][very thick];
\draw  [shift = {+(0,2)}](2.75,1.5) -- (6.75,3) [<-] [very thick];
\draw  [shift = {+(0,2)}](6.25,1.5) -- (2.75,3) [->] [very thick];

\draw (3.75,3.5) arc (180:0:.75cm and .5cm)[<-][very thick];
\draw (6.5,3.5) arc (180:0:.75cm and .5cm)[->][very thick];

\draw[shift = {+(0,4)}](-.25,1) rectangle (1.75,1.5);
\draw [shift = {+(0,4)}](.75,1.25) node {$(k_2)$};
\draw [shift = {+(0,4)}](2.25,1) rectangle (4.25,1.5);
\draw [shift = {+(0,4)}](3.25,1.25) node {$(k_1-2)$};
\filldraw [gray][shift = {+(0,4)}](4.4,1) rectangle (7.1,1.5);
\draw [shift = {+(0,4)}](5.75,1.25) node {$(1^{a+k_2+1})$};
\filldraw[gray] [shift = {+(0,4)}](7.5,1) rectangle (9.5,1.5);
\draw [shift = {+(0,4)}](8.5,1.25) node {$(1^{a+k_1})$};

\shadedraw[gray] [shift = {+(0,6)}](0,0) rectangle (4,.5);
\draw [shift = {+(0,6)}](2,.25) node {$(k_1-2,k_2)$};
\shadedraw[gray]  [shift = {+(0,6)}](5,0) rectangle (9,.5);
\draw [shift = {+(0,6)}](7,.25) node {$(2^{a+1+k_2},1^{k_1=k_2-1})$};
\draw [shift = {+(0,5)}](1,.5) -- (1,1) [->][very thick];
\draw [shift = {+(0,5)}](3,.5) -- (3,1) [->][very thick];
\draw [shift = {+(0,5)}](6,.5) -- (6,1) [<-][very thick];
\draw [shift = {+(0,5)}](8,.5) -- (8,1) [<-][very thick];

\end{tikzpicture}
$$
which is equal to a scalar multiple of
$$
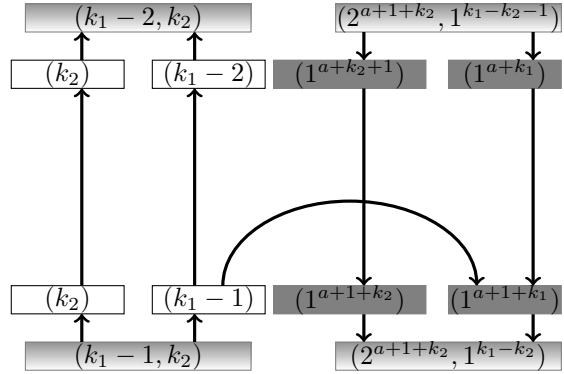
\begin{tikzpicture}[scale=.75][>=stealth]
\shadedraw[gray]  (0,0) rectangle (4,.5);
\draw (2,.25) node {$(k_1-1,k_2)$};
\shadedraw[gray]  (5.5,0) rectangle (9.5,.5);
\draw (7.5,.25) node {$(2^{a+1+k_2},1^{k_1-k_2})$};
\draw (1,.5) -- (1,1) [->][very thick];
\draw (3,.5) -- (3,1) [->][very thick];
\draw (6,.5) -- (6,1) [<-][very thick];
\draw (9,.5) -- (9,1) [<-][very thick];

\draw (-.25,1) rectangle (1.75,1.5);
\draw (.75,1.25) node {$(k_2)$};
\draw (2.25,1) rectangle (4.25,1.5);
\draw (3.25,1.25) node {$(k_1-1)$};
\filldraw [gray](4.4,1) rectangle (7.1,1.5);
\draw (5.75,1.25) node {$(1^{a+1+k_2})$};
\filldraw[gray] (7.5,1) rectangle (9.5,1.5);
\draw (8.5,1.25) node {$(1^{a + 1+k_1})$};

\draw[shift = {+(0,4)}](-.25,1) rectangle (1.75,1.5);
\draw [shift = {+(0,4)}](.75,1.25) node {$(k_2)$};
\draw [shift = {+(0,4)}](2.25,1) rectangle (4.25,1.5);
\draw [shift = {+(0,4)}](3.25,1.25) node {$(k_1-2)$};
\filldraw [gray][shift = {+(0,4)}](4.4,1) rectangle (7.1,1.5);
\draw [shift = {+(0,4)}](5.75,1.25) node {$(1^{a+k_2+1})$};
\filldraw[gray] [shift = {+(0,4)}](7.5,1) rectangle (9.5,1.5);
\draw [shift = {+(0,4)}](8.5,1.25) node {$(1^{a+k_1})$};

\shadedraw[gray] [shift = {+(0,6)}](0,0) rectangle (4,.5);
\draw [shift = {+(0,6)}](2,.25) node {$(k_1-2,k_2)$};
\shadedraw[gray]  [shift = {+(0,6)}](5.5,0) rectangle (9.5,.5);
\draw [shift = {+(0,6)}](7.5,.25) node {$(2^{a+1+k_2},1^{k_1-k_2-1})$};
\draw [shift = {+(0,5)}](1,.5) -- (1,1) [->][very thick];
\draw [shift = {+(0,5)}](3,.5) -- (3,1) [->][very thick];
\draw [shift = {+(0,5)}](6,.5) -- (6,1) [<-][very thick];
\draw [shift = {+(0,5)}](9,.5) -- (9,1) [<-][very thick];

\draw (1,1.5) -- (1,5) [->][very thick];
\draw (3,1.5) -- (3,5) [->][very thick];
\draw (6,1.5) -- (6,5) [<-][very thick];
\draw (9,1.5) -- (9,5) [<-][very thick];
\draw (3.5,1.5) arc (180:0:2.25cm and 1.5cm)[->][very thick];

\end{tikzpicture}
$$
The right hand side is clearly equal to the map 
$$g_3: \P^{(k_1-1,k_2)} \Q^{(2^{a+1+k_2}, 1^{k_1-k_2})} \longrightarrow \P^{(k_1-1,k_2)} \Q^{(2^{a+1+k_2}, 1^{k_1-k_2-1})} \la 1 \ra.$$
Thus $g_3$ shows up as the map induced by the differential
$$\P^{(k_2)} \Q^{(1^{a+1+k_2})} \P^{(k_1)} \Q^{(1^{a+1+k_1})} 
\xrightarrow{ 
\;\; \vcenter{\xy (-2,-1)*{}; (2,-1)*{} **\crv{(-2,3) & (2,3)}?(1)*\dir{>}; (2,-3)*{};(-2,3)*{}; \endxy} \;\; 
{\xy {\ar (0,-3)*{};(0,3)*{} };(1.5,0)*{};(-1.5,0)*{};\endxy}
{\xy {\ar (0,-3)*{};(0,3)*{} };(1.5,0)*{};(-1.5,0)*{};\endxy}
}
\P^{(k_2)} \Q^{(1^{a+1+k_2})} \P^{(k_1-1)} \Q^{(1^{a+k_1})} \la 1 \ra.$$
However, it is possible that in the cancellation process this map becomes zero. A little bit of reflection convinces one that this can only happen if the other differential 
$$\P^{(k_2)} \Q^{(1^{a+1+k_2})} \P^{(k_1)} \Q^{(1^{a+1+k_1})} 
\xrightarrow{
{\xy {\ar (0,-3)*{};(0,3)*{} };(1.5,0)*{};(-1.5,0)*{};\endxy}
{\xy {\ar (0,-3)*{};(0,3)*{} };(1.5,0)*{};(-1.5,0)*{};\endxy}
\;\; \vcenter{\xy (-2,-1)*{}; (2,-1)*{} **\crv{(-2,3) & (2,3)}?(1)*\dir{>}; (2,-3)*{};(-2,3)*{}; \endxy} \;\; 
} 
\P^{(k_2-1)} \Q^{(1^{a+k_2})} \P^{(k_1)} \Q^{(1^{a+1+k_1})} \la 1 \ra$$
also induces the map $g_3$ since then in the process of applying the cancellation lemma these two maps could cancel. Fortunately, the right hand side $\P^{(k_2-1)} \Q^{(1^{a+k_2})} \P^{(k_1)} \Q^{(1^{a+1+k_1})}$ does not contain any summand $\P^{(k_1-1,k_2)} \Q^{(2^{a+1+k_2}, 1^{k_1-k_2-1})}$ so this does not happen. 

Similarly, one can show that $g_1$ also occurs in the differential of (\ref{eq:EEtcpx}). Thus $\E_{i,n-1} \E_{i,n} \1_\l$ is indeed homotopic to the (unique up to homotopy) complex (\ref{eq:EEtcpx}) with nonzero multiples of $g_1$ and $g_3$ in the differential. 

{\bf Computation of $\E_{i,n} \E_{i,n-1} \1_\l$.} This is isomorphic to
$$\left[ \dots \rightarrow \P^{(2)} \Q^{(1^{a+5})} \la -2 \ra \rightarrow \P \Q^{(1^{a+4})} \la -1 \ra \rightarrow \Q^{(1^{a+3})} \right] \left[ \dots \rightarrow \P^{(2)} \Q^{(1^{a+2})} \la -2 \ra \rightarrow \P \Q^{(1^{a+1})} \la -1 \ra \rightarrow \Q^{(1^a)} \right]$$
which simplifies to give
$$\bigoplus_{k_1+k_2=l} \P^{(k_2)} \left[\P^{(k_1)} \Q^{(1^{a+k_2+3})} \oplus \P^{(k_1-1)} \Q^{(1^{a+k_2+2})} \otimes_\k V_1 \oplus \P^{(k_1-2)} \Q^{(1^{a+k_2+2})} \right] \Q^{(1^{a+k_1})} \la -l \ra$$
in cohomological degree $-l$. This computation is similar so we just sketch it. Again we collect terms with a grading shift $\la -l \ra$ to obtain
\begin{equation}\label{eq:C}
\bigoplus_{k_1+k_2=l} \P^{(k_2)} \P^{(k_1-2)} \Q^{(1^{a+k_2+1})} \Q^{(1^{a+k_1})} \longrightarrow \bigoplus_{k_1+k_2=l-1} \P^{(k_2)} \P^{(k_1-1)} \Q^{(1^{a+k_2+2})} \Q^{(1^{a+k_1})}
\end{equation}
and
\begin{equation}\label{eq:D}
\bigoplus_{k_1+k_2=l-1} \P^{(k_2)} \P^{(k_1-1)} \Q^{(1^{a+k_2+2})} \Q^{(1^{a+k_1})} \longrightarrow \bigoplus_{k_1+k_2=l} \P^{(k_2)} \P^{(k_1)} \Q^{(1^{a+k_2+3})} \Q^{(1^{a+k_1})}.
\end{equation}
It turns out that (\ref{eq:C}) is an isomorphism and that (\ref{eq:D}) is injective. 

{\bf Calculation of (\ref{eq:C}).} We rewrite (\ref{eq:C}) as
$$
\bigoplus_{k_1+k_2=l-1} \P^{(k_2)} \P^{(k_1-1)} \Q^{(1^{a+k_2+1})} \Q^{(1^{a+k_1+1})} \longrightarrow \bigoplus_{k_1+k_2=l-1} \P^{(k_2)} \P^{(k_1-1)} \Q^{(1^{a+k_2+2})} \Q^{(1^{a+k_1})}
$$
and then cancel to get
$$
\bigoplus_{\begin{matrix} \scriptstyle k_1+k_2=l-1 \\ \scriptstyle k_2 \ge k_1 \end{matrix}}
\P^{(k_2)} \P^{(k_1-1)} \Q^{(2^{a+k_1+1}, 1^{k_2-k_1})} \longrightarrow
\bigoplus_{\begin{matrix} \scriptstyle k_1+k_2=l-1 \\ \scriptstyle k_2+2 \le k_1 \end{matrix}}
\P^{(k_2)} \P^{(k_1-1)} \Q^{(2^{a+k_2+2}, 1^{k_1-k_2-2})}.
$$
We then rewrite this as 
$$
\bigoplus_{\begin{matrix} \scriptstyle k_1+k_2=l-1 \\ \scriptstyle k_2 \ge k_1 \end{matrix}}
\P^{(k_2)} \P^{(k_1-1)} \Q^{(2^{a+k_1+1}, 1^{k_2-k_1})} \longrightarrow
\bigoplus_{\begin{matrix} \scriptstyle k_1+k_2=l-1 \\ \scriptstyle k_2 \ge k_1 \end{matrix}}
\P^{(k_1-1)} \P^{(k_2)} \Q^{(2^{a+k_1+1}, 1^{k_2-k_1})}
$$
which turns out to be an isomorphism. So (\ref{eq:C}) is homotopic to zero.

{\bf Calculation of (\ref{eq:D}).} One can check using the same argument as before that the map in equation (\ref{eq:D}) is surjective. We do not repeat this argument but instead just keep track of the terms left over after cancellation. First we rewrite (\ref{eq:D}) as 
$$
\bigoplus_{k_1+k_2=l} \P^{(k_2)} \P^{(k_1)} \Q^{(1^{a+k_2+2})} \Q^{(1^{a+k_1+1})} \longrightarrow \bigoplus_{k_1+k_2=l} \P^{(k_2)} \P^{(k_1)} \Q^{(1^{a+k_2+3})} \Q^{(1^{a+k_1})}
$$
which simplifies to 
$$
\bigoplus_{\begin{matrix} \scriptstyle k_1+k_2=l \\ \scriptstyle k_2 \ge k_1-1 \end{matrix}}
\P^{(k_2)} \P^{(k_1)} \Q^{(2^{a+k_1+1},1^{k_2-k_1+1})} \longrightarrow
\bigoplus_{\begin{matrix} \scriptstyle k_1+k_2=l \\ \scriptstyle k_2+3 \le k_1 \end{matrix}}
\P^{(k_2)} \P^{(k_1)} \Q^{(2^{a+k_2+3}, 1^{k_1-k_2-3})}.
$$
We then rewrite both sides to obtain
$$
\bigoplus_{\begin{matrix} \scriptstyle k_1+k_2=l-1 \\ \scriptstyle k_1 \ge k_2 \end{matrix}}
\P^{(k_2+1)} \P^{(k_1)} \Q^{(2^{a+k_2+2},1^{k_1-k_2})} \longrightarrow
\bigoplus_{\begin{matrix} \scriptstyle k_1+k_2=l-1 \\ \scriptstyle k_1 \ge k_2 \end{matrix}}\P^{(k_2-1)} \P^{(k_1+2)} \Q^{(2^{a+k_2+2}, 1^{k_1-k_2})}.
$$
This in turn simplifies to give
$$
\bigoplus_{\begin{matrix} \scriptstyle k_1+k_2=l-1 \\ \scriptstyle k_1 \ge k_2 \end{matrix}}
\P^{(k_1+1,k_2)} \Q^{(2^{a+k_2+2},1^{k_2-k_1})} \la -l \ra 
\bigoplus_{\begin{matrix} \scriptstyle k_1+k_2=l-1 \\ \scriptstyle k_1 \ge k_2+1 \end{matrix}}
\P^{(k_1,k_2+1)} \Q^{(2^{a+k_2+2},1^{k_2-k_1})} \la -l \ra
$$
in cohomological degree $(-l-1)$ (where we have added back the $\la - l \ra$ shift). Replacing $l,k_1,k_2$ by $l-2,k_1-1,k_2-1$ we see that these are the same terms as those in the complex (\ref{eq:EEtcpx}) with the extra shift $\la 2 \ra$. One can check as before that $g_1$ and $g_3$ appear in the differentials and hence $\E_{i,n} \E_{i,n-1} \la -2 \ra \1_\l$ is also homotopic to (\ref{eq:EEtcpx}). Thus we are done. 


\section{Commutation of $\E_i$ and $\F_j$}

\begin{prop}\label{prop:EiFjcomm} For any $i \ne j \in I$ we have $\F_{j,n} \E_{i,m} \1_\l \cong \E_{i,m} \F_{j,n} \1_\l$. 
\end{prop}
\begin{proof}
There are two cases, depending of whether $i$ and $j$ are connected by an edge. If $\la i,j \ra = 0$ then $\P_i$ and $\Q_i$ commute with $\P_j$ and $\Q_j$ and the result follows immediately. So suppose $\la i,j \ra = -1$. There are several cases to consider. 

{\bf Case 1.} Suppose that $\la \l, \alpha_i \ra + m \ge -1$ and $\la \l, \alpha_j \ra - n \le -1$. This means that
\begin{eqnarray*}
\E_{i,m} \1_\l &\cong& 
\left[ \dots \rightarrow \P_i^{(k)} \Q_i^{(1^{a+k})} \la -k \ra \rightarrow \dots \rightarrow \P_i \Q_i^{(1^{a+1})} \la -1 \ra \rightarrow \Q_i^{(1^{a})} \right] \\
\F_{j,n} \1_{\l+\alpha_i+m\delta} &\cong& 
\left[ \Q_j^{(b)} \rightarrow \P_j \Q_j^{(b+1)} \la 1 \ra \rightarrow \dots \rightarrow \P_j^{(1^l)} \Q_j^{(b+l)} \la l \ra \rightarrow \dots \right] \la b \ra [-b] \\
\F_{j,n} \1_\l &\cong& 
\left[ \Q_j^{(b-1)} \rightarrow \P_j \Q_j^{(b)} \la 1 \ra \rightarrow \dots \rightarrow \P_j^{(1^{l'})} \Q_j^{(b-1+l')} \la l' \ra \rightarrow \dots \right] \la b-1 \ra [-b+1] \\ 
\E_{i,m} \1_{\l-\alpha_j+n\delta} &\cong& 
\left[ \dots \rightarrow \P_i^{(k')} \Q_i^{(1^{a+1+k'})} \la -k' \ra \rightarrow \dots \rightarrow \P_i \Q_i^{(1^{a+2})} \la -1 \ra \rightarrow \Q_i^{(1^{a+1})} \right]
\end{eqnarray*}
where $a = \la \l, \alpha_i \ra + 1 + m$ and $b = - \la \l, \alpha_j \ra + n$.  

Now, the terms in $\F_{j,n} \E_{i,m} \1_\l$ in cohomological degree $(h+b)$ are 
$$\bigoplus_{l-k=h} \P_j^{(1^l)} \Q_j^{(b+l)} \P_i^{(k)} \Q_i^{(1^{a+k})} \1_\l \la h+b \ra.$$
Since $\Q_j^{(b+l)} \P_i^{(k)} \cong \P_i^{(k)} \Q_j^{(b+l)} \oplus \P_i^{(k-1)} \Q_j^{(b+l-1)}$ this simplifies to give 
\begin{equation}\label{eq:9}
\bigoplus_{l-k=h} 
\left[ \P_j^{(1^l)} \P_i^{(k)} \Q_j^{(b+l)} \Q_i^{(1^{a+k})} \1_\l \oplus \P_j^{(1^l)} \P_i^{(k-1)} \Q_j^{(b+l-1)} \Q_i^{(1^{a+k})} \1_\l \right] \la h+b \ra.
\end{equation}
Likewise, the terms in $\E_{i,m} \F_{j,n} \1_\l$ in cohomological degree $(h'+b-1)$ are 
$$\bigoplus_{l'-k'=h'} \P_i^{(k')} \Q_i^{(1^{a+1+k'})} \P_j^{(1^{l'})} \Q_j^{(b-1+l')} \1_\l \la h'+b-1 \ra$$
which simplifies to give 
\begin{equation}\label{eq:9'}
\bigoplus_{l'-k'=h'} 
\left[ \P_i^{(k')} \P_j^{(1^{l'})} \Q_i^{(1^{a+1+k'})} \Q_j^{(b-1+l')} \oplus \P_i^{(k')} \P_j^{(1^{l'-1})} \Q_i^{(1^{a+k'})}  \Q_j^{({b-1+l'})} \right] \la h'+b-1 \ra.
\end{equation}
Now it is easy to see that the terms in (\ref{eq:9'}) match up with the terms in (\ref{eq:9}) when $h'=h+1$ and $(k',l')=(k,l+1)$ or $(k',l')=(k-1,l)$. Thus the complexes for $\E_{i,m} \F_{j,n} \1_\l$ and $\F_{j,n} \E_{i,m} \1_\l$ match up term by term and it is not hard to check that the differentials are the same. 

{\bf Case 2.} Now, consider the case that $\la \l, \alpha_i \ra + m \ge -1$ but $\la \l, \alpha_j \ra -n \ge 0$. This means that:
\begin{eqnarray*}
\E_{i,m} \1_\l &\cong& 
\left[ \dots \rightarrow \P_i^{(k)} \Q_i^{(1^{a+k})} \la -k \ra \rightarrow \dots \rightarrow \P_i \Q_i^{(1^{a+1})} \la -1 \ra \rightarrow \Q_i^{(1^{a})} \right] \\
\F_{j,n} \1_{\l+\alpha_i+m\delta} &\cong&
\left[ \P_j^{(1^b)} \rightarrow \P_j^{(1^{b+1})} \Q_j \la 1 \ra \rightarrow \dots \rightarrow \P_j^{(1^{b+l})} \Q_j^{(l)} \la l \ra \rightarrow \dots \right] \\
\F_{j,n} \1_\l &\cong&
\left[ \P_j^{(1^{b+1})} \rightarrow \P_j^{(1^{b+2})} \Q_j \la 1 \ra \rightarrow \dots \rightarrow \P_j^{(1^{b+1+l'})} \Q_j^{({l'})} \la l' \ra \rightarrow \dots \right] \\
\E_{i,m} \1_{\l-\alpha_j+n\delta} &\cong& 
\left[ \dots \rightarrow \P_i^{(k')} \Q_i^{(1^{a+1+k'})} \la -k' \ra \rightarrow \dots \rightarrow \P_i \Q_i^{(1^{a+2})} \la -1 \ra \rightarrow \Q_i^{(1^{a+1})} \right] 
\end{eqnarray*}
where $a = \la \l, \alpha_i \ra + 1 + m$ and $b = \la \l, \alpha_j \ra - n$. The terms in $\F_{j,n} \E_{i,m} \1_\l$ in cohomological degree $h$ are
$$\bigoplus_{l-k=h} \P_j^{(1^{b+l})} \Q_j^{(l)} \P_i^{(k)} \Q_i^{(1^{a+k})} \la h \ra$$
which simplifies to give 
\begin{equation}\label{eq:10}
\bigoplus_{l-k=h} \left[ \P_j^{(1^{b+l})} \P_i^{(k)} \Q_j^{(l)} \Q_i^{(1^{a+k})} \1_\l \oplus \P_j^{(1^{b+l})} \P_i^{(k-1)} \Q_j^{({l-1})} \Q_i^{(1^{a+k})} \1_\l \right] \la h \ra.
\end{equation}
Likewise, the terms in $\E_{i,m} \F_{j,n} \1_\l$ in cohomological degree $h'$ are
$$\bigoplus_{l'-k'=h'} \P_i^{(k')} \Q_i^{(1^{a+1+k'})} \P_j^{(1^{b+1+l'})} \Q_j^{({l'})} \1_\l \la h' \ra$$
which simplifies to give
\begin{equation}\label{eq:10'}
\bigoplus_{l'-k'=h'} \left[  \P_i^{(k')} \P_j^{(1^{b+1+l'})} \Q_i^{(1^{a+1+k'})} \Q_j^{({l'})} \1_\l \oplus \P_i^{(k')} \P_j^{(1^{b+l'})} \Q_i^{(1^{a+k'})} \Q_j^{({l'})} \1_\l \right] \la h' \ra.
\end{equation}
It is easy to see that the terms in (\ref{eq:10}) match up with those in (\ref{eq:10'}) via the identification $(k',l')=(k,l)$ or $(k',l')=(k-1,l-1)$. Again, it is not hard to check that the differentials also match up. 

{\bf Case 3 and 4.} There are two further cases which to consider, namely when $\la \l, \alpha_i \ra + m \le 0$ and either $\la \l, \alpha_j \ra - n \le -1$ or $\la \l, \alpha_j \ra - n \ge 0$. These are proven in exactly the same way as above and so we omit the details. 
\end{proof}

\section{The commutation relation of $\E_i$ and $\E_{j,1}$}

\begin{prop}\label{prop:cones}
If $\la i, j \ra = -1$ then there exist unique nonzero maps
$$\E_{i} \E_{j,1} \1_\l \xrightarrow{\alpha} \E_{j,1} \E_{i} \la 1 \ra \1_\l \ \ \text{ and } \ \  \E_{j} \E_{i,1} \1_\l \xrightarrow{\beta} \E_{i,1} \E_{j} \la 1 \ra \1_\l$$
in $\K$ and we have $\Cone(\alpha) \cong \Cone(\beta)$. Meanwhile, if $\la i,j \ra = 0$ then $\E_{i,m}$ and $\E_{j,n}$ commute for any $m,n \in \{0,1\}$. 
\end{prop}

The commutation of $\E_{i,m}$ and $\E_{j,n}$ when $\la i,j \ra = 0$ is obvious since $\P_i$ commutes with $\Q_j$ in this case. It remains to prove the first assertion when $\la i,j \ra = -1$. 

First we need to show that $\alpha$ and $\beta$ are unique (up to rescaling). 

\begin{lemma} If $\la i, j \ra = -1$ then $\Hom(\E_i \E_{j,1} \1_\l, \E_{j,1} \E_i \la 1 \ra \1_\l) \cong \k \cong \End(\E_i \E_{j,1} \1_\l)$. 
\end{lemma}
\begin{proof}
This is a formal consequence of the other relations in a $(\hg,\theta)$ action together with the fact that $\End(\1_\l) \cong \k$. More precisely, in \cite{CK} Lemma 4.5 we prove that given any a catgorical $\sl_3$ action generated by $\E_i$ and $\E_j$ (and their adjoints) we have $\dim \Hom(\E_i \E_j, \E_j \E_i \la 1 \ra) = 1$ (assuming neither terms are zero). To do this we only use that $\1_{\l+r \alpha} = 0$ for $r \gg 0$ or $r \ll 0$ (where $\alpha := \alpha_i+\alpha_j$), biadjointness of $\E_i$ and $\E_j$, that $\E_i$ and $\F_j$ commute and the commutator relations of $[\E_i, \F_i]$ and $[\E_j, \F_j]$. Since $\E_i$ and $\E_{j,1}$ (together with adjoints) also generate an $\sl_3$ action the result follows.
\end{proof}

We now give an explicit description of $\alpha$ and $\beta$ as maps of complexes of 1-morphisms in $\Kom(\K)$. Let us suppose that $\la \l, \alpha_i \ra = a-1 \ge 0$ and $\la \l, \alpha_j \ra = b+1 \ge 0$. Then the general terms of $\E_{j,1} \1_\l$ and $\E_i \1_{\l+\alpha_j+\delta}$ are 
$$\P_j^{(n)} \Q_j^{(1^{b+n+1})} \la -n \ra \text{ and } \P_i^{(n)} \Q_i^{(1^{a-1+n})} \la -n \ra.$$ 
This means that $\E_i \E_{j,1} \1_\l$ is a complex which looks like 
\begin{equation}\label{eq:Eij'}
\dots \rightarrow \bigoplus_{k = 0}^{n+1} \P_i^{(n+1-k)} \Q_i^{(1^{a+n-k})} \P_j^{(k)} \Q_j^{(1^{b+1+k})} \la -n-1 \ra \rightarrow 
\bigoplus_{k = 0}^n \P_i^{(n-k)} \Q_i^{(1^{a-1+n-k})} \P_j^{(k)} \Q_j^{(1^{b+1+k})} \la -n \ra \rightarrow \dots
\end{equation}
where the right hand term is in cohomological degree $-n$. 

Similarly, the general terms of $\E_i \1_\l$ and $\E_{j,1} \1_{\l+\alpha_i}$ are 
$$\P_i^{(n)} \Q_i^{(1^{a+n})} \la -n \ra \text{ and }
\P_j^{(n)} \Q_j^{(1^{b+n})} \la -n \ra.
$$
This means that $\E_{j,1} \E_i \la 1 \ra \1_\l$ is a complex which looks like
\begin{equation}\label{eq:Ej'i}
\dots \rightarrow 
\bigoplus_{k = 0}^{n+1} \P_j^{(k)} \Q_j^{(1^{b+k})} \P_i^{(n+1-k)} \Q_i^{(1^{a+n+1-k})} \la -n \ra
\rightarrow 
\bigoplus_{k = 0}^n \P_j^{(k)} \Q_j^{(1^{b+k})} \P_i^{(n-k)} \Q_i^{(1^{a+n-k})} \la -n+1 \ra \rightarrow \dots
\end{equation}
where the right hand term is in cohomological degree $-n$. 

On the other hand, $\E_j \E_{i,1} \1_\l$ looks like 
\begin{equation}\label{eq:Eji'} 
\dots \rightarrow 
\bigoplus_{k=0}^{n+1} \P_j^{(k)} \Q_j^{(1^{b-1+k})} \P_i^{(n+1-k)} \Q_i^{(1^{a+n+2-k})} \la -n-1 \ra \rightarrow 
\bigoplus_{k=0}^{n} \P_j^{(k)} \Q_j^{(1^{b-1+k})} \P_i^{(n-k)} \Q_i^{(1^{a+n+1-k})} \la -n \ra \rightarrow \dots
\end{equation}
while $\E_{i,1} \E_j \la 1 \ra \1_\l$ looks like 
\begin{equation}\label{eq:Ei'j}
\dots \rightarrow \bigoplus_{k=0}^{n+1} \P_i^{(n+1-k)} \Q_i^{(1^{a+n+1-k})} \P_j^{(k)} \Q_j^{(1^{b+k})} \la -n \ra 
\rightarrow \bigoplus_{k=0}^{n+1} \P_i^{(n-k)} \Q_i^{(1^{a+n-k})} \P_j^{(k)} \Q_j^{(1^{b+k})} \la -n+1 \ra \rightarrow \dots.
\end{equation}

We now write down the pictures which define the map of complexes $ \E_i \E_{j,1} \rightarrow \E_{j,1} \E_i \la 1 \ra$. The chain map will take 
$$\P_i^{(n+1-k)} \Q_i^{(1^{a+n-k})} \P_j^{(k)} \Q_j^{(1^{b+1+k})} \longrightarrow 
\P_j^{(k)} \Q_j^{(1^{b+k})} \P_i^{(n+1-k)} \Q_i^{(1^{a+n+1-k})} \oplus
\P_j^{(k)} \Q_j^{(1^{b+k})} \P_i^{(n+1-k)} \Q_i^{(1^{a+n+1-k})}.$$

The map to the first summand is
\begin{equation}\label{eq:diagrams1}
\begin{tikzpicture}[>=stealth]
\draw (-.25,1) rectangle (1.75,1.5);
\draw (.75,1.25) node {$(n+1-k)_i$};
\filldraw [gray] (2.25,1) rectangle (4.25,1.5);
\draw (3.25,1.25) node {$(1^{a+n-k})_i$};
\draw (4.4,1) rectangle (7.1,1.5);
\draw (5.75,1.25) node {$(k)_j$};
\filldraw [gray] (7.5,1) rectangle (9.5,1.5);
\draw (8.5,1.25) node {$(1^{b+1+k})_j$};

\draw [shift = {+(0,2)}](-.25,1) rectangle (1.75,1.5);
\draw [shift = {+(0,2)}](.75,1.25) node {$({k+1})_j$};
\filldraw [gray] [shift = {+(0,2)}](2.25,1) rectangle (4.25,1.5);
\draw [shift = {+(0,2)}](3.25,1.25) node {$(1^{b+1+k})_j$};
\draw [shift = {+(0,2)}](4.4,1) rectangle (7.1,1.5);
\draw [shift = {+(0,2)}](5.75,1.25) node {$(n-k)_i$};
\filldraw[gray] [shift = {+(0,2)}](7.5,1) rectangle (9.5,1.5);
\draw [shift = {+(0,2)}](8.5,1.25) node {$(1^{a+n-k})_i$};

\draw (1,1.5) -- (6,3) [->][very thick];
\draw (3,1.5) -- (8,3) [<-][very thick];
\draw (6,1.5) -- (1,3) [->][very thick];
\draw (8,1.5) -- (3,3) [<-][very thick];

\draw (.5,1.5) -- (.5,3) [->][very thick];

\filldraw [blue] (.5,2.25) circle (3pt);
\end{tikzpicture}
\end{equation}
where the solid dot is a degree one $i-j$ dot. Similarly, the map to the second summand is

\begin{equation}\label{eq:diagrams2}
\begin{tikzpicture}[>=stealth]
\draw (-.25,1) rectangle (1.75,1.5);
\draw (.75,1.25) node {$(n+1-k)_i$};
\filldraw [gray] (2.25,1) rectangle (4.25,1.5);
\draw (3.25,1.25) node {$(1^{a+n-k})_i$};
\draw  (4.4,1) rectangle (7.1,1.5);
\draw (5.75,1.25) node {$(k)_j$};
\filldraw [gray] (7.5,1) rectangle (9.5,1.5);
\draw (8.5,1.25) node {$(1^{b+1+k})_j$};

\draw [shift = {+(0,2)}](-.25,1) rectangle (1.75,1.5);
\draw [shift = {+(0,2)}](.75,1.25) node {$({k})_j$};
\filldraw [gray] [shift = {+(0,2)}](2.25,1) rectangle (4.25,1.5);
\draw [shift = {+(0,2)}](3.25,1.25) node {$(1^{b+k})_j$};
\draw [shift = {+(0,2)}](4.4,1) rectangle (7.1,1.5);
\draw [shift = {+(0,2)}](5.75,1.25) node {$({n+1-k})_i$};
\filldraw [gray] [shift = {+(0,2)}](7.5,1) rectangle (9.5,1.5);
\draw [shift = {+(0,2)}](8.5,1.25) node {$(1^{a+n-k+1})_i$};

\draw (1,1.5) -- (6,3) [->][very thick];
\draw (3,1.5) -- (8,3) [<-][very thick];
\draw (6,1.5) -- (1,3) [->][very thick];
\draw (8,1.5) -- (3,3) [<-][very thick];

\draw (9,1.5) -- (9,3) [->][very thick];

\filldraw [blue] (9,2.25) circle (3pt);
\end{tikzpicture}
\end{equation}

We need to check that these diagrams define a chain map, meaning that the map above commutes with the differential. Diagrammatically, this amounts to checking that diagrams  

$$
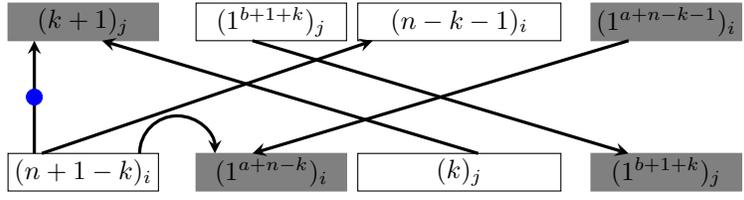
\begin{tikzpicture}[>=stealth]
\draw (-.25,1) rectangle (1.75,1.5);
\draw (.75,1.25) node {$(n+1-k)_i$};
\filldraw [gray] (2.25,1) rectangle (4.25,1.5);
\draw (3.25,1.25) node {$(1^{a+n-k})_i$};
\draw (4.4,1) rectangle (7.1,1.5);
\draw (5.75,1.25) node {$(k)_j$};
\filldraw [gray] (7.5,1) rectangle (9.5,1.5);
\draw (8.5,1.25) node {$(1^{b+1+k})_j$};

\filldraw [gray] [shift = {+(0,2)}](-.25,1) rectangle (1.75,1.5);
\draw [shift = {+(0,2)}](.75,1.25) node {$({k+1})_j$};
\draw [shift = {+(0,2)}](2.25,1) rectangle (4.25,1.5);
\draw [shift = {+(0,2)}](3.25,1.25) node {$(1^{b+1+k})_j$};
\draw [shift = {+(0,2)}](4.4,1) rectangle (7.1,1.5);
\draw [shift = {+(0,2)}](5.75,1.25) node {$(n-k-1)_i$};
\filldraw[gray] [shift = {+(0,2)}](7.5,1) rectangle (9.5,1.5);
\draw [shift = {+(0,2)}](8.5,1.25) node {$(1^{a+n-k-1})_i$};

\draw (.2,1.5) -- (4.6,3) [->][very thick];
\draw (3,1.5) -- (8,3) [<-][very thick];
\draw (6,1.5) -- (1,3) [->][very thick];
\draw (8,1.5) -- (3,3) [<-][very thick];

\draw (.1,1.5) -- (.1,3) [->][very thick];

\filldraw [blue] (.1,2.25) circle (3pt);

\draw (1.5,1.5) arc (180:0:.5cm and .5cm)[->][very thick];

\end{tikzpicture},
$$
and
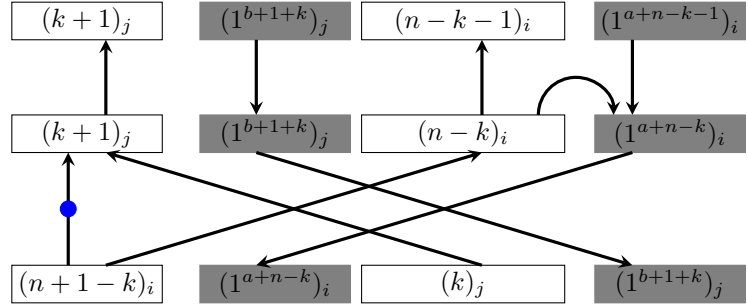
$$
\begin{tikzpicture}[>=stealth]
\draw (-.25,1) rectangle (1.75,1.5);
\draw (.75,1.25) node {$(n+1-k)_i$};
\filldraw [gray] (2.25,1) rectangle (4.25,1.5);
\draw (3.25,1.25) node {$(1^{a+n-k})_i$};
\draw (4.4,1) rectangle (7.1,1.5);
\draw (5.75,1.25) node {$(k)_j$};
\filldraw [gray] (7.5,1) rectangle (9.5,1.5);
\draw (8.5,1.25) node {$(1^{b+1+k})_j$};

\draw [shift = {+(0,2)}](-.25,1) rectangle (1.75,1.5);
\draw [shift = {+(0,2)}](.75,1.25) node {$({k+1})_j$};
\filldraw [gray] [shift = {+(0,2)}](2.25,1) rectangle (4.25,1.5);
\draw [shift = {+(0,2)}](3.25,1.25) node {$(1^{b+1+k})_j$};
\draw [shift = {+(0,2)}](4.4,1) rectangle (7.1,1.5);
\draw [shift = {+(0,2)}](5.75,1.25) node {$(n-k)_i$};
\filldraw[gray] [shift = {+(0,2)}](7.5,1) rectangle (9.5,1.5);
\draw [shift = {+(0,2)}](8.5,1.25) node {$(1^{a+n-k})_i$};

\draw (1,1.5) -- (6,3) [->][very thick];
\draw (3,1.5) -- (8,3) [<-][very thick];
\draw (6,1.5) -- (1,3) [->][very thick];
\draw (8,1.5) -- (3,3) [<-][very thick];

\draw (.5,1.5) -- (.5,3) [->][very thick];

\filldraw [blue] (.5,2.25) circle (3pt);

\draw (6.75,3.5) arc (180:0:.5cm and .5cm)[->][very thick];

\draw [shift = {+(0,3.5)}](-.25,1) rectangle (1.75,1.5);
\draw [shift = {+(0,3.5)}](.75,1.25) node {$({k+1})_j$};
\filldraw [gray] [shift = {+(0,3.5)}](2.25,1) rectangle (4.25,1.5);
\draw [shift = {+(0,3.5)}](3.25,1.25) node {$(1^{b+1+k})_j$};
\draw [shift = {+(0,3.5)}](4.4,1) rectangle (7.1,1.5);
\draw [shift = {+(0,3.5)}](5.75,1.25) node {$(n-k-1)_i$};
\filldraw[gray] [shift = {+(0,3.5)}](7.5,1) rectangle (9.5,1.5);
\draw [shift = {+(0,3.5)}](8.5,1.25) node {$(1^{a+n-k-1})_i$};

\draw [shift = {+(0,1.5)}](1,2) -- (1,3) [->][very thick];
\draw [shift = {+(0,1.5)}](3,2) -- (3,3) [<-][very thick];
\draw [shift = {+(0,1.5)}](6,2) -- (6,3) [->][very thick];
\draw [shift = {+(0,1.5)}](8,2) -- (8,3) [<-][very thick];

\end{tikzpicture}
$$
commute. To see this we simplify the second diagram. The two middle box idempotents in the middle level (the $1^{(b+1+k)}$ and the $(n-k)$) can be absorbed into the idempotents in the bottom level. Now slide the cap downwards. It moves through the first line for free. To pass it through the second line involves creating a sum of two diagrams, the first term of which just moves the cap through; the second term creates a subdiagram
$$
\begin{tikzpicture}[>=stealth]
\draw (-.25,1) rectangle (1.75,1.5);
\draw (.75,1.25) node {$(2)_j$};

\draw [shift = {+(0,2)}](-.25,1) rectangle (1.75,1.5);
\draw [shift = {+(0,2)}](.75,1.25) node {$(2)_i$};

\draw (.5,1.5) -- (.5,3) [->][very thick];
\draw (1,1.5) -- (1,3) [->][very thick];
\filldraw [blue] (.5,2.0) circle (3pt);
\filldraw [blue] (1,2.5) circle (3pt);
\draw (2,2.25) node {$= $}; 

\draw [shift = {+(3,0)}](-.25,1) rectangle (1.75,1.5);
\draw [shift = {+(3,0)}](.75,1.25) node {$(2)_j$};

\draw [shift = {+(3,0)}][shift = {+(0,2)}](-.25,1) rectangle (1.75,1.5);
\draw [shift = {+(3,0)}][shift = {+(0,2)}](.75,1.25) node {$(2)_i$};

\draw [shift = {+(3,0)}](.5,1.5) -- (1,2) [very thick];
\draw [shift = {+(3,0)}](1,1.5) -- (.5,2) [very thick];

\draw [shift = {+(3,0)}](.5,2) -- (.5,3) [->][very thick];
\draw [shift = {+(3,0)}](1,2) -- (1,3) [->][very thick];
\filldraw [shift = {+(3,0)}][blue] (.5,2.25) circle (3pt);
\filldraw [shift = {+(3,0)}][blue] (1,2.5) circle (3pt);
\draw [shift = {+(3,0)}](2,2.25) node {$= \ \ -$}; 

\draw [shift = {+(6,0)}](-.25,1) rectangle (1.75,1.5);
\draw [shift = {+(6,0)}](.75,1.25) node {$(2)_j$};

\draw [shift = {+(6,0)}][shift = {+(0,2)}](-.25,1) rectangle (1.75,1.5);
\draw [shift = {+(6,0)}][shift = {+(0,2)}](.75,1.25) node {$(2)_i$};

\draw [shift = {+(6,0)}](.5,2.5) -- (1,3) [->][very thick];
\draw [shift = {+(6,0)}](1,2.5) -- (.5,3) [->][very thick];

\draw [shift = {+(6,0)}](.5,1.5) -- (.5,2.5) [very thick];
\draw [shift = {+(6,0)}](1,1.5) -- (1,2.5) [very thick];
\filldraw [shift = {+(6,0)}][blue] (.5,2) circle (3pt);
\filldraw [shift = {+(6,0)}][blue] (1,2.25) circle (3pt);
\draw [shift = {+(6,0)}](2,2.25) node {$= \ \ -$}; 

\draw  [shift = {+(9,0)}](-.25,1) rectangle (1.75,1.5);
\draw [shift = {+(9,0)}] (.75,1.25) node {$(2)_j$};

\draw  [shift = {+(9,0)}][shift = {+(0,2)}](-.25,1) rectangle (1.75,1.5);
\draw  [shift = {+(9,0)}][shift = {+(0,2)}](.75,1.25) node {$(2)_i$};

\draw  [shift = {+(9,0)}](.5,1.5) -- (.5,3) [->][very thick];
\draw  [shift = {+(9,0)}](1,1.5) -- (1,3) [->][very thick];
\filldraw  [shift = {+(9,0)}][blue] (.5,2.0) circle (3pt);
\filldraw  [shift = {+(9,0)}][blue] (1,2.5) circle (3pt);

\end{tikzpicture}
$$
In the above graphical computation the minus sign comes from the two degree one dots passing one another with respect to the horizontal; from the above computation we see that this subdiagram is zero. 

 We conclude that we have a map of complexes $\E_i \E_{j,1} \rightarrow \E_{j,1} \E_i \la 1 \ra$.  A straightforward check shows that this map of complexes is nonzero. This defines the chain map $\alpha$,  and we define the chain map $\beta$ similarly. 

\subsubsection{Proof of Proposition \ref{prop:cones}}
Looking at equations (\ref{eq:Eij'}) and (\ref{eq:Ej'i}) we see that $\Cone(\alpha)$ is a complex where 
\begin{equation}\label{eq:cone1}
\bigoplus_{k = 0}^n \P_i^{(n-k)} \Q_i^{(1^{a-1+n-k})} \P_j^{(k)} \Q_j^{(1^{b+1+k})} \la -n \ra \bigoplus_{k = 0}^{n+1} \P_j^{(k)} \Q_j^{(1^{b+k})} \P_i^{(n+1-k)} \Q_i^{(1^{a+n+1-k})} \la -n \ra
\end{equation}
is the term appearing in cohomological degree $-n-1$. The second term above can be rewritten as
\begin{equation}\label{eq:cone1a}
\bigoplus_{k = 0}^{n} \P_j^{(k)} \Q_j^{(1^{b-1+k})} \P_i^{(n-k)} \Q_i^{(1^{a+n+1-k})} \la -n \ra \bigoplus_{k = 0}^{n+1} \P_j^{(k)} \P_i^{(n+1-k)} \Q_j^{(1^{b+k})} \Q_i^{(1^{a+n+1-k})} \la -n \ra.
\end{equation}
On the other hand, looking at equations (\ref{eq:Eji'}) and (\ref{eq:Ei'j}) we find that $\Cone(\beta)$ is a complex where 
\begin{equation}\label{eq:cone2}
\bigoplus_{k=0}^{n} \P_j^{(k)} \Q_j^{(1^{b-1+k})} \P_i^{(n-k)} \Q_i^{(1^{a+n+1-k})} \la -n \ra \bigoplus_{k=0}^{n+1} \P_i^{(n+1-k)} \Q_i^{(1^{a+n+1-k})} \P_j^{(k)} \Q_j^{(1^{b+k})} \la -n \ra
\end{equation}
is the term appearing in cohomological degree $-n-1$. The second term above can be rewritten as 
\begin{equation}\label{eq:cone2a}
\bigoplus_{k=0}^n \P_i^{(n-k)} \Q_i^{(1^{a+n-k-1})} \P_j^{(k)} \Q_j^{(1^{b+k+1})} \la -n \ra \bigoplus_{k=0}^{n+1} \P_i^{(n+1-k)} \P_j^{(k)} \Q_i^{(1^{a+n+1-k})} \Q_j^{(1^{b+k})} \la -n \ra.
\end{equation}
Let us denote by $\A^{-n-1}$ the direct sum of the first term in (\ref{eq:cone1}) and all of (\ref{eq:cone1a}) and likewise by $\B^{-n-1}$ the direct sum of the first term in (\ref{eq:cone2}) and all of (\ref{eq:cone2a}). These are the degree $-n-1$ terms in the complexes $\Cone(\alpha)$ and $\Cone(\beta)$. It is straight-forward to check that $\A^{-n-1} \cong \B^{-n-1}$ by just matching terms. 

It remains to show that the differentials in $\Cone(\alpha)$ and $\Cone(\beta)$ agree. We do this by applying Lemma \ref{lem:2}. The first and second conditions in Lemma \ref{lem:2} follow from Lemmas \ref{lem:onedimhoms} and \ref{lem:square}. The third condition follows from a computation almost identical to that used for the proof of the relations in Section \ref{sec:commE}. We include the indecomposable 1-morphism $\P_i^{(a+1)} \P_j^{(1^{b})} \Q_i^{(1^{c})} \Q_j^{(d)}$ into the appropriate term of $\Cone(\alpha)$, apply the boundary map and then project onto each indecomposable summand. The composition is a collection of diagrams which can be simplified. Doing this we find the degree one map from Lemma \ref{lem:onedimhoms}. This map can be shown to be nonzero by taking its closure as explained in the example of Section \ref{sec:Kom(H)}. 

\begin{lemma}\label{lem:2}
Consider a finite complex 
$$\A^\bullet := \dots \longrightarrow \oplus_\ell \A^{-n-1}_\ell \longrightarrow \oplus_\ell \A^{-n}_\ell \longrightarrow \dots \longrightarrow \oplus_\ell \A^0_\ell$$
in the homotopy category of some additive category. Suppose that it satisfies the following:
\begin{enumerate}
\item $\Hom(\A^{-n}_\ell, \A^{-n+1}_{\ell'})$ is either zero or one-dimensional for all $\ell, \ell',n$ 
\item for any $\ell,\ell_1,\ell_2,n$ such that  
$$\Hom(\A^{-n-1}_\ell, \A^{-n}_{\ell_1}) \cong \k \cong \Hom(\A^{-n-1}_\ell, \A^{-n}_{\ell_2})$$ 
there exists $\ell'$ and a nonzero map in $\Hom(\A^{-n-1}_\ell, \A^{-n+1}_{\ell'})$ which factors through $\A^{-n}_{\ell_1}$ and $\A^{-n}_{\ell_2}$ 
\item for any $\ell,n \ne 0$ there exists a nonzero map with domain $\A^{-n}_\ell$. 
\end{enumerate}
Now suppose $\B^\bullet$ is another complex such that $\A^{-n} \cong \B^{-n}$ for all $n$ and $\B^\bullet$ satisfies the same conditions as $\A^\bullet$ above. Then $\A^\bullet \cong \B^\bullet$. 
\end{lemma}
\begin{proof}
The proof is based on the observation that given a complex
\begin{eqnarray*}
A \overset{\alpha_1}{\underset{\alpha_2}{\rightrightarrows}} \begin{matrix} B_1 \\ B_2 \end{matrix} \overset{\beta_1}{\underset{\beta_2}{\rightrightarrows}} C
\end{eqnarray*}
where $\Hom(A,B_1) \cong \Hom(A,B_2) \cong \Hom(B_1,C) \cong \Hom(B_2,C) \cong \Hom(A,C) \cong \k$ are spanned by $\alpha_1,\alpha_2,\beta_1,\beta_2$ and $\beta_1 \circ \alpha_1 = - \beta_2 \circ \alpha_2$ then any three of the maps determines uniquely the fourth.

We now apply this to our problem. Fix an isomorphism $\A^{-n} \xrightarrow{\sim} \B^{-n}$ and proceed by induction starting from the far right. For the base case we note that $\A_\ell^1=0$ so there is a unique nonzero map out of each $\A_\ell^{-1}$. Acting by a multiple of the identity on $\A_\ell^{-1}$ this map can be scaled so that it equals to that in $\B^\bullet$. For the induction step consider a nonzero map $\A_\ell^{-n-1} \rightarrow \A_{\ell'}^{-n}$ and rescale $\A_\ell^{-n-1}$ so that this map agrees with that in $\B^\bullet$. Then using the observation above (and induction) it follows that all the other maps out of $\A_\ell^{-n-1}$ must also agree with those in $\B^\bullet$. 
\end{proof}

\section{Applications, conjectures and further comments}\label{sec:applications}

In this section, we set the field $\k=\C$ to be the complex numbers. Fix $\G \subset SL_2(\C)$, a non-trivial finite subgroup.  By the McKay correspondence, such subgroups are classified by finite Dynkin diagrams of type $A,D$ or $E$. To such a diagram there are three associated Lie algebras: a finite dimensional simply-laced Lie algebra $\g$, the affine Lie algebra $\hg$ and the associated toroidal Lie algebra $\widehat{\hg}$ with $\g \subset \hg \subset \widehat{\hg}$. We now describe two (essentially equivalent) categorifications of the basic representation of $\dU_q(\widehat{\hg})$ using the finite subgroup $\G \subset SL_2(\C)$.

\subsection{2-representations via Hilbert schemes}

Let $X_\G = \widehat{\C^2/\G}$ denote the minimal resolution of the quotient singularity $\C^2/\G$ and denote by $X_\G^{[n]}$ the Hilbert scheme of $n$ points on $X_\G$. The diagonal $\C^\times$-action on $\C^2$ induces a $\C^\times$-action on $X_\G^{[n]}$. Let $DCoh^{\C^\times}(X_\G)$ denote the derived category of $\C^\times$-equivariant coherent sheaves on $X_\G$. Its Grothendieck group (tensored with $\C(q)$) is denoted $K^{\C^\times}(X_\G)$.

In \cite{CLi1} we constructed a level one integrable 2-representation of $\widehat{\h}$ on $\bigoplus_{n \in \N} DCoh(X_\G^{[n]})$ where $\widehat{\h}$ is the toroidal Heisenberg algebra. Corollary \ref{cor:main1} explains how we obtain a $(\hg,\theta)$ action from a 2-representation of $\h$. Adding the extra affine vertex to the story this immediately implies the following theorem.

\begin{Theorem}\label{thm:geometry}  
The action of $\widehat{\h}$ on $\oplus_{n \in \N} DCoh(X_\G^{[n]})$ induces a $(\widehat{\hg},\theta)$ action on $\oplus_{\alpha \in \hY, n\in \N} DCoh(X_\G^{[n]})$.
\end{Theorem}

\begin{cor}\label{cor:Ktheory}  
The quantum toroidal algebra $\dU_q(\widehat{\hg})$ acts on 
$$\bigoplus_{\alpha \in \hY, n\in \N} K^{\C^\times}(X_\G^{[n]}) = \bigoplus_{n \in \N} K^{\C^\times}(X_\G^{[n]}) \otimes_{\C(q)} \C(q)[\hY].$$  
The resulting module is the basic representation. 
\end{cor}

Theorem \ref{thm:geometry} and Corollary \ref{cor:Ktheory} were conjectured by Nakajima in \cite{Nak-moduli}. Moreover, Corollary \ref{cor:Ktheory} is both an affinization and a $q$-deformation of the work of Nakajima and Grojnowski which gives affine Lie algebra actions on cohomology of Hilbert schemes.

\subsection{2-representations via wreath products}

A 2-representation of $\widehat{\hg}$ can also be constructed using the representation theory of finite dimensional superalgebras. Let $B_\G := \C[\G] \ltimes \wedge^*(\C^2)$ and set $B_\G(n) = \C[S_n] \ltimes B_\G^{\otimes n}$ (we include $n=0$ in this definition, setting $B_\G(0) = \C$). The natural $\Z$ grading on $\wedge^*(\C^2)$ makes $B_\G(n)$ into a $\Z$-graded superalgebra. Let $\mathcal{B}_\G(n)$ denote the category of finitely generated, graded $B_\G(n)$ supermodules.

In \cite{CLi1}, a level one 2-representation of $\widehat{\h}$ was constructed on $\bigoplus_{n \in \N} \mathcal{B}_\G(n)$ categorifying the Fock space representation of $\widehat{\h}$. We do not need to use derived categories of modules since, in contrast to the action on Hilbert schemes, the Heisenberg generators $\P_i$ and $\Q_i$ act by exact functors. However, to obtain a 2-representation of $\widehat{\hg}$ we still need to pass to the homotopy category $\Kom(\mathcal{B}_\G(n))$, since $\E_i$ and $\F_i$ are given by complexes. 

\begin{Theorem}\label{thm:algebra}
The action of $\oplus_{n \in \N} \widehat{\h}$ on $ \mathcal{B}_\G(n)$ induces a $(\widehat{\hg},\theta)$ action on $\oplus_{\alpha \in \hY, n \in \N} \Kom(\mathcal{B}_\G(n))$.
\end{Theorem}

Theorem \ref{thm:algebra} is very similar in spirit to the constructions of toroidal basic representations in \cite{FJW}.

\subsection{The Kac-Moody description}\label{sec:KM-relationship}

The affine Lie algebra $\hg = \g \otimes \k[t,t^{-1}] \oplus \k c$ is also a Kac-Moody Lie algebra. This Kac-Moody presentation has generators $\{e_i,f_i,h_i\}_{i \in \hI}$. The isomorphism between the Kac-Moody and loop presentations is defined as follows. 

Let $\theta$ denote the highest root of $\g$ and $\g_\theta$ the associated root space. Choose elements $E_\theta \in \g_{-\theta}$, $F_\theta \in \g_\theta$ such that $\theta(H_\theta) = -2$, where $H_\theta := [E_\theta,F_\theta]$. Then we define
$$e_0 \mapsto E_\theta \otimes t, \ \ f_0 \mapsto F_\theta \otimes t^{-1} \ \text{ and } \ h_0 \mapsto H_\theta \otimes 1 + c$$
while $e_i \mapsto E_i$ and $f_i \mapsto F_i$ if $i \in I$.

Categorifications of Kac-Moody algebras have been defined by Khovanov-Lauda \cite{KL1,KL2,KL3} and Rouquier \cite{R}. They are given by a 2-category $\U_Q(\hg)_{KM}$ which depends on some scalars $Q$ (we use the notation from \cite{CLa}). The 2-category in \cite{KL3} corresponds to a particular choice of such scalars. The 2-categories for different choices of $Q$ are sometimes but not always isomorphic. In general the space of isomorphism classes of these 2-categories is the first homology of the associated Dynkin diagram. 

We next describe a relationship between the categorification of the basic representation constructed in the Kac-Moody setting and the categorification in the current paper. We do this from a geometric and then an algebraic point of view.

\subsubsection{Quiver varieties}

The basic representation of $\hg$ was constructed geometrically by Nakajima using the $\G$-equivariant geometry of the Hilbert scheme of points on $\C^2$ (see \cite{Nak-ICM}). More precisely, $\G \subset SL_2$ acts on $\C^2$ and hence on all the Hilbert schemes ${\C^2}^{[n]}$. The fixed point components $({\C^2}^{[n]})^\G$ are Nakajima quiver varieties of affine type. Nakajima defines an action of $\hg$ on $\bigoplus_{n \in \N} H_{mid}(({\bA^2}^{[n]})^\G,\C)$, where $H_{mid}$ denotes the middle cohomology, giving the basic representation. One can also carry out this construction by replacing homology with $\C^\times$-equivariant K-theory, in which case the quantum affine Kac-Moody algebra $\dU_q(\hg)$ acts. This action was subsequently lifted to derived categories of coherent sheaves. 

\begin{Theorem}\cite{CKL4,Cau}\label{thm:quiver}
For some choice of scalars $Q$ there exists a 2-representation of the 2-category $\U_Q(\hg)_{KM}$ on $\bigoplus_{n \in \N} DCoh^{\C^\times}(({\C^2}^{[n]})^\G)$.
\end{Theorem}

More precisely, in \cite{CKL4} we constructed a geometric categorical $\hg$ action on $\bigoplus_{n \in \N} DCoh^{\C^\times}(({\C^2}^{[n]})^\G)$. Such a geometric action induces a $(\g,\theta)$ action which, as shown in \cite{Cau}, must carry an action of the KLR algebras. Putting this together with the main result in \cite{CLa} then implies Theorem \ref{thm:quiver}. 

The Grothendieck group of $DCoh^{\C^\times}(({\C^2}^{[n]})^\G)$ contains more than just the basic representation of $\dU_q(\hg)$. This is because the entire quantum toroidal algebra $\dU_q(\widehat{\hg})$ acts and this gives its basic representation \cite{Nak-ICM}. In light of this we conjecture the following. 

\begin{conj}\label{conj:geometry}
The 2-representation from Theorem \ref{thm:quiver} extends to a 2-representation of the toroidal algebra. 
\end{conj}

\begin{Remark}
The toroidal algebra in the conjecture above should be thought of as the affinization of the Kac-Moody quantum affine algebra. In particular, $\U_Q(\hg)_{KM}$ should act (for some choice of $Q$) together with the loop algebra part which acts like in the definition in Section \ref{sec:g-action} but at level zero rather than level one. This conjecture was not proven in \cite{CKL4} in part because there was no such definition of a 2-representation of toroidal algebras.
\end{Remark}

So there are two possible categorifications of the basic representation of $\dU_q(\widehat{\hg})$ using derived categories of coherent sheaves, that of Theorem \ref{thm:geometry} and that of Conjecture \ref{conj:geometry}. These categorifications are in some ways quite different. For example, in the categorification involving $({\C^2}^{[n]})^\G$, the Kac-Moody generators $\E_i$ and $\F_i$ are explicitly described, while the homogeneous Heisenberg generators $\P_i$ and $\Q_i$ are not as easily visible. On the other hand, in the categorification involving $\widehat{\C^2/\G}^{[n]}$ the Heisenberg generators $\P_i$ and $\Q_i$ acquire a simpler geometric interpretation while the Kac-Moody generators $\E_i$ and $\F_i$ are given by more complicated categorified vertex operators.  

However, the varieties $({\C^2}^{[n]})^\G$ and $\widehat{\C^2/\G}^{[k]}$ are closely related. Both can be realized as Nakajima quiver varieties, but for different stability conditions and hence are derived equivalent. Subsequently one can conjecture the following. 

\begin{conj}\label{conj:equiv}
There is an equivalence between the 2-representations of $\widehat{\hg}$ from Theorem \ref{thm:geometry} and Conjecture \ref{conj:geometry}.
\end{conj}

The above conjecture is complicated by the fact that the isomorphism between the loop and Kac-Moody presentations of the quantum affine algebra is somewhat subtle. In particular, to prove Conjecture \ref{conj:equiv} one should assign endofunctors to the Kac-Moody 1-morphism $e_0$ lifting the relation $e_0 = E_\theta \otimes t$ between the Kac-Moody and loop descriptions of this operator. Such an assignment is not given in the current paper.

\subsubsection{Cyclotomic KLR algebras and wreath products}

There is a parallel algebraic version of these categorifications. In \cite{KL1,KL3,R} a family of cyclotomic quiver Hecke algebras -- also known as cyclotomic KLR algebras --  were defined. In particular, the following theorem was conjectured in \cite{KL1,KL3} and subsequently proven in \cite{KK,W}.

\begin{Theorem}\label{thm:KLR}
The 2-categories $\U_Q(\hg)_{KM}$ act on $\oplus_{\lambda} R^{\Lambda_0}_{Q,\l}-\mbox{mod}$ where $R^{\Lambda_0}_{Q,\l}$ is the cyclotomic KLR algebra for the weight space $\l$ in the basic representation $V_{\Lambda_0}$ (and for some choice of scalars $Q$).
\end{Theorem}

In light of the connections between KLR algebras and quiver varieties \cite{VV,R} this theorem is an algebraic analogue of Theorem \ref{thm:quiver}. Subsequently we expect that Theorem \ref{thm:KLR} can be extended to give a 2-representation of quantum toroidal algebras. Then there should be an equivalence between the categorification of the basic representation of $\dU_q(\widehat{\hg})$ using toroidal cyclotomic KLR algebras and the categorification of Theorem \ref{thm:algebra}.

However, since toroidal cyclotomic KLR algebras have not been defined, we now restrict the construction of Theorem \ref{thm:algebra} from the quantum toroidal algebra to the quantum affine algebra in order to formulate a precise conjecture relating cyclotomic KLR algebras to the algebras $B_\G(n)$ of Section \ref{sec:applications}.

The isomorphism classes of indecomposable projective $B_\G$ modules, $\{\P_i\}_{i \in \hI}$ are in bijection with the nodes of the affine Dynkin diagram. The endomorphism algebra
$$ B'_\G := \End_{B_\G} (\oplus_{i \in I} \P_i) $$
where we omit the projective module correponding to the affine node, is a subalgebra of $B_\G$.  Let $B'_\G(n) = \k[S_n] \ltimes {B'_\G}^{\otimes n}$.

\begin{conj}\label{conj:alg}
For each $n \in \N$ there is some weight $\l$ of the form $\l = w \cdot \Lambda_0 - n \delta$ (and some choice of scalars $Q$) such that $R_{Q,\lambda}^{\Lambda_0}$ is Morita equivalent to $B'_\G(n)$.
\end{conj}
\begin{Remark}
Note that there is a braid group acting and hence, assuming the conjecture above, there is a derived Morita equivalence between $R_{Q,\lambda}^{\Lambda_0}$ and the algebra $B'_\G(n)$ for any $\l = w \cdot \Lambda_0 - n \delta$. Part of the content of Conjecture \ref{conj:alg} is that this equivalence is non-derived for an appropriate $\l$. 
\end{Remark}

\subsection{Braid groups} 

An application of categorical actions of quantum groups in the Kac-Moody setting is that they induce actions of the associated braid group \cite{CR,CK}. More generally, we expect the following. 

\begin{conj}\label{conj:braid}
A categorical action of $\widehat{\hg}$ induces an action of the double affine braid group.
\end{conj}

In the conjecture above a categorical action means either a $(\widehat{\hg},\theta)$ action or some analogous geometric or algebraic definition.


\end{document}